\documentclass[final,onefignum,onetabnum,letterpaper]{siamonline250211}
\usepackage[scale=0.8]{geometry} % Reduce document margins

\usepackage{lipsum}
\usepackage{amsfonts}
\usepackage{epstopdf}
\ifpdf
  \DeclareGraphicsExtensions{.eps,.pdf,.png,.jpg}
\else
  \DeclareGraphicsExtensions{.eps}
\fi

\usepackage{datetime}
\newdateformat{monthyeardate}{%
  \monthname[\THEMONTH] \THEDAY, \THEYEAR}

\usepackage{academicons}
\usepackage{xcolor}
\renewcommand{\orcid}[1]{\href{https://orcid.org/#1}{\textcolor[HTML]{A6CE39}{orcid.org/#1}}}

\usepackage{amsmath}
\allowdisplaybreaks
\usepackage{amssymb}
\usepackage{commath}
\usepackage{mathtools}
\usepackage{array} % needed for the >{\displaystyle}c column specifiers in the figure of Section 7
\usepackage{bbm}

\usepackage{color}
\usepackage{graphicx}
\usepackage[small]{caption}
\usepackage{subcaption}

\usepackage{relsize}
\usepackage{adjustbox}
\usepackage{algorithm}
\usepackage[noend]{algpseudocode}
\usepackage{booktabs}
\usepackage{tikz}
\usetikzlibrary{patterns, decorations.pathreplacing, automata, arrows, positioning, calc, patterns.meta}

\usepackage{verbatim}
\usepackage{cleveref}
\usepackage{csquotes}

\usepackage{enumitem}
\setlist[enumerate]{leftmargin=.5in}
\setlist[itemize]{leftmargin=.5in}

\newsiamthm{problem}{Problem}
\newsiamremark{remark}{Remark}
\newsiamremark{hypothesis}{Hypothesis} 
\crefname{hypothesis}{Hypothesis}{Hypotheses}
\newsiamremark{example}{Example}
\newsiamthm{claim}{Claim}
\newsiamthm{conjecture}{Conjecture}

\headers{Energy-stable overset grid methods using sub-cell SBP operators}{J.\ Glaubitz, J.\ Lampert, A.\,R.\ Winters, and J.\ Nordstr\"om}

\title{Towards provable energy-stable overset grid methods using sub-cell summation-by-parts operators
\thanks{
\monthyeardate\today 
\corresponding{Jan Glaubitz} 
}}

\author{
Jan Glaubitz\thanks{
Department of Mathematics, Link\"oping University, 58183 Linköping, Sweden
(\email{jan.glaubitz@liu.se}, \orcid{0000-0002-3434-5563};  \email{andrew.ross.winters@liu.se}, \orcid{0000-0002-5902-1522}; \email{jan.nordstrom@liu.se}, \orcid{0000-0002-7972-6183})
}
\and 
Joshua Lampert\thanks{
Department of Mathematics, University of Hamburg, Bundesstraße 55, 20146 Hamburg, Germany
(\email{joshua.lampert@uni-hamburg.de}, \orcid{0009-0007-0971-6709})
}
\and 
Andrew R.\ Winters\footnotemark[2] 
\and
Jan Nordstr\"om\footnotemark[2] \thanks{
Department of Mathematics and Applied Mathematics, University of Johannesburg, P.O.\ Box 524, Auckland Park 2006, Johannesburg, South Africa
}
}

\usepackage{amsopn}

\DeclareMathOperator{\diag}{diag}

\newcommand{\scp}[2]{\left\langle{#1, #2}\right\rangle}

\renewcommand{\d}{\mathrm{d}}
\newcommand{\intd}{\, \mathrm{d}}
\newcommand{\N}{\mathbb{N}}

\newcommand{\R}{\mathbb{R}}

\newcommand{\fnum}{f^{\operatorname{num}}}

\newcommand{\vecfnum}{\boldsymbol{f}^{\operatorname{num}}} % vector-valued numerical flux (from macros.tex; matches the \boldsymbol{f} notation used for the flux)

\usepackage{lineno}
\ifpdf
\hypersetup{
  pdftitle={Sub-cell SBP operators},
  pdfauthor={Jan Glaubitz}
}
\fi

\begin{document}

\maketitle

% REQUIRE D
\begin{abstract}
	% 1. Why should people care?
% 2. What is the specific research question?
% 3. What is done here?
% 4. What are the key findings?
% 5. What are the implications of these findings?
Overset grid methods handle complex geometries by overlapping simpler, geometry-fitted grids to cover the original, more complex domain.
However, ensuring their stability---particularly at high orders---remains a theoretical challenge: although overset grid methods perform robustly in extensive practical use, general stability proofs are not available.
In this work, we address this gap by developing a discrete counterpart to the recent well-posedness analysis of Kopriva, Gassner, and Nordstr\"om for continuous overset domain initial-boundary-value problems.
To this end, we introduce the novel concept of sub-cell summation-by-parts (SBP) operators.
These discrete derivative operators mimic integration by parts at a sub-cell level.
By exploiting this sub-cell SBP property, we develop provably conservative and energy-stable overset grid methods for fixed one-dimensional overset domains that do not change with time or under grid refinement, providing a step toward stability proofs for overset grid methods based on the energy method.
\end{abstract}

% REQUIRED

% REQUIRED

% Repository where code can be found (if there is any) 
\begin{Code}
    \url{https://github.com/JoshuaLampert/2025_overset_grid_sub-cell}
\end{Code}

% Once the paper is published
\begin{DOI}
	\url{https://doi.org/10.1016/j.jcp.2026.115347}
\end{DOI}

\section{Introduction}
\label{sec:introduction}

% 1) Why should others care?
% 2) What has been done already? What is the problem with what has already been done?
% 3) What is done here?
% 4) What are the findings?
% 5) What are the implications and applications?

% General motivation
Numerical simulations of many physical phenomena in science and engineering rely on solving hyperbolic conservation laws on complex geometries.
In such contexts, the cost of generating and refining grids can become significant.
To mitigate this challenge, several approaches have been developed to simplify or eliminate the mesh generation process.
Prominent members among these approaches are mesh-free radial basis function (RBF) methods built on point clouds \cite{fornberg2015primer}, cut finite element methods (CutFEM) that operate on unfitted meshes \cite{burman2015cutfem}, and overset grid methods \cite{thompson1998handbook}---which are the focus of this work.

% Overset grid methods
Overset grid methods, also known as Chimera methods, have been employed for over four decades to facilitate simulations involving complex geometries by overlapping structured and body-fitted grids.
These methods are especially beneficial for problems involving bodies in relative motion, such as helicopters or wind turbines, where single-grid approaches require mesh movement or re-meshing at each time step.
Overset grids allow for independently moving meshes, simplifying grid management.
An additional advantage is the use of a fitted grid near the body of interest, combined with a Cartesian mesh in the far field, which yields improved computational efficiency.

A typical overset grid setup overlays a high-resolution, body-fitted mesh onto a coarser background mesh. 
The solutions on overlapping grids are coupled through interface conditions applied at the overset boundaries, and data is often exchanged in overlap regions to ensure consistency across grids \cite{nordstrom2014flexible,frenander2016stable,brazell2016overset,frenander2017stable}.
The overset grid methodology has matured into a well-established framework, supported by a dedicated biennial conference series \cite{oversetSymposium2024}, and implemented in numerous software packages such as \cite{henshaw2002overture, buning2004cfd, chan2018chimera}.
Overset grids have been successfully integrated into a broad spectrum of discretization frameworks, including finite difference (FD) methods \cite{steger1987use}; finite element (FE) methods, continuous \cite{hansbo2003finite}, discontinuous \cite{galbraith2014discontinuous,brazell2016overset}, and hybridizable \cite{kauffman2017overset}; finite volume (FV) methods \cite{brown1997overture}; and spectral methods \cite{kopriva1989computation,merrill2016spectrally}.
Applications span fields such as aerodynamics \cite{steger1987use}, solid mechanics \cite{appelo2012numerical}, meteorology \cite{kageyama2004yin}, and electromagnetics \cite{henshaw2006high}.

% Stability of overset grid problems
Overset grid methods perform robustly in practical use, and much is known about their stability from normal-mode and Gustafsson--Kreiss--Sundstr\"om (GKS) type analyses \cite{gustafsson1972stability,starius1980composite,berger1985stability} as well as discrete spectral studies \cite{petersson2010stable,angel2018high} (see also the related works below).
Nevertheless, energy-based stability proofs that hold for arbitrary grid configurations remain unavailable for advection-dominated problems; in particular, no fully multi-dimensional, provably energy-stable overset grid scheme is currently available for systems of hyperbolic conservation laws.
% Recent theoretical work
One key challenge in establishing a general stability proof is that the original linear initial-boundary-value problem (IBVP) must be well-posed.
Without well-posedness of the continuous overset problem, stability of a discretization that mimics it cannot be established via the energy method.
The steps used to demonstrate well-posedness can often be directly followed to prove stability for numerical schemes \cite{nordstrom2017roadmap}.
The recent works \cite{kopriva2022theoretical,kopriva2022theoretical2} were the first to establish well-posedness of the fixed overset grid problem for linear hyperbolic conservation laws and entropy stability of the fixed overset grid problem for non-linear hyperbolic conservation laws, respectively---i.e., within the viewpoint that the overlapping domains are fixed at the continuous level and do not change under grid refinement.
The proofs in \cite{kopriva2022theoretical,kopriva2022theoretical2} rely on two main tools:
(1) Performing integration by parts (IBP) both \emph{on the whole domain and sub-domains}, and (2) enforcing boundary and interface conditions weakly using penalty terms \cite{nordstrom2014flexible,frenander2016stable,brazell2016overset,frenander2017stable,nordstrom2021neural}.
Here, we focus on developing a discrete analog to the first main tool:
IBP on sub-cells.

\subsection*{Our contribution}

We introduce the novel concept of \emph{sub-cell SBP operators}.
Classical SBP operators \cite{svard2014review,fernandez2014review,chen2020review,glaubitz2022summation} discretely mimic IBP on the entire cell on which they are defined. These operators yield provably conservative and energy-stable numerical methods for energy-bounded \emph{non-overset domain} IBVPs when combined with weakly enforced boundary and interface conditions.
In contrast, the proposed sub-cell SBP operators mimic IBP on the whole cell \emph{as well as} on pre-defined sub-cells that partition the domain into two non-overlapping pieces.
This extension allows us to discretely mimic the theoretical framework of \cite{kopriva2022theoretical,kopriva2022theoretical2} and develop provably conservative and energy-stable methods for energy-bounded \emph{one-dimensional IBVPs on fixed overset domains which do not change with time or under grid refinement}.

We further establish the existence of sub-cell SBP operators and provide a constructive procedure for their design.
In particular, we show that a sub-cell SBP operator on a generic cell $\Omega_{\rm ref} = [\omega_L,\omega_R]$ with sub-cells $\Omega_{L} = [\omega_L,\omega_M]$ and $\Omega_{R} = [\omega_M,\omega_R]$ exists if and only if traditional SBP operators exist on both $\Omega_{L}$ and $\Omega_{R}$. 
While the present work focuses on one-dimensional hyperbolic conservation laws as a proof of concept, it should be viewed as a first step: whether the approach extends to general multi-dimensional overset configurations remains an open question (see \Cref{sec:extension}).

For clarity, our theoretical analysis is restricted to the scalar linear advection equation.
Nevertheless, our computational experiments demonstrate that the sub-cell SBP overset methodology also yields conservative and entropy-stable schemes for more complex problems, including the inviscid Burgers', linear Maxwell's, and compressible Euler equations.

\subsection*{Related works}

The treatment of grid interfaces for hyperbolic conservation laws has a long history, with early foundational work establishing the importance of conservation at interfaces for ensuring convergence to the correct weak solution and accurate shock propagation \cite{berger1987conservation}. 
In particular, Berger developed a general framework for constructing conservative interface conditions on discontinuous and overlapping grids, including mesh refinement in space and time, by deriving interface fluxes consistent with discrete conservation principles \cite{berger1987conservation}. 
Complementary to this, stability aspects were rigorously analyzed through root conditions and the application of Gustafsson–Kreiss–Sundström theory, demonstrating stability for schemes such as Lax–Wendroff under mesh refinement \cite{berger1985stability}.
The GKS framework itself goes back to \cite{gustafsson1972stability}, and normal-mode analyses of composite (overlapping) mesh methods date back at least to \cite{starius1980composite}.
More recently, discrete spectral analyses have been used to study and stabilize overset grid discretizations of wave equations, where added upwind dissipation suppresses weak growth associated with the grid coupling \cite{angel2018high}.
Building on these foundations, a substantial body of work has focused on conservative interpolation and coupling strategies for overset and multiblock grids. 
Notable contributions include conservative interpolation techniques for overlapping grids \cite{chesshire1994scheme,part1994conservative}, fully conservative interface algorithms \cite{wang1995fully}, and studies of shock propagation across overlapping grids \cite{wu1999steady}. 
High-order accurate methods on overset grids have also been developed, including compact finite-difference schemes \cite{sherer2005high} and variable high-order multiblock approaches for complex multiscale flows \cite{sjogreen2009variable,lani2013variable}. 
In addition, stability and accuracy issues for grid refinement and interface treatments have been further explored in the context of wave propagation problems \cite{petersson2010stable}. 
More recently, alternative paradigms such as interpolation-free conservative overset methods have been proposed to address some of the inherent limitations of classical interpolation-based coupling \cite{devlin2023direct}.

Complementary to these approaches, which rely on interpolation, flux matching, or tailored interface constructions, the present work explores an operator-based viewpoint in which conservation and energy stability are encoded in the algebraic structure of the discrete derivative operators.
Within the one-dimensional, fixed-overlap setting considered here, this structure yields provable conservation and energy stability; its extension to multiple dimensions based on the one-dimensional results is only partly developed.
Finally, it should also be noted that existing stability analyses for overset grid methods have primarily focused on scalar or linearized systems in one spatial dimension, employing energy methods in conjunction with SBP-FD discretizations \cite{bodony2011provably,reichert2011stable,reichert2012energy,sharan2018time}.
However, these previous efforts rely on the assumption that the coefficient matrices are simultaneously diagonalizable, which limits their applicability in more general settings.

\subsection*{Outline}

In \Cref{sec:continuous}, we establish the continuous overset grid problem for hyperbolic conservation laws in one dimension and derive some of its properties, such as conservation and energy stability.
Notably, the continuous analysis in \Cref{sec:continuous} relies on performing IBP on both entire cells and pre-defined sub-cells.
To discretely mimic continuous analysis, we introduce the concept of sub-cell SBP operators in \Cref{sec:operators}.
In \Cref{sec:application}, we then formulate sub-cell SBP semi-discretizations for one-dimensional hyperbolic conservation laws, and prove that these are conservative and energy-stable for the linear advection equation.
Furthermore, in \Cref{sec:existence}, we establish the existence of the proposed sub-cell SBP operators and provide a procedure for their construction. 
In \Cref{sec:tests}, we demonstrate the computational performance of the sub-cell SBP-based overset grid methods for a series of experiments, including the scalar linear advection equation, inviscid Burgers' equation, linear system of Maxwell's equations, and compressible Euler equations. 
Finally, in \Cref{sec:extension}, we outline how the proposed framework of sub-cell SBP operators can be extended to multiple dimensions
We offer concluding thoughts, open problems, and an outlook on future research in \Cref{sec:summary}.
\section{The continuous problem}
\label{sec:continuous}

We begin by setting up the continuous problem and deriving some of its properties, which the numerical method will later mimic using the proposed sub-cell SBP operators.
We focus on fixed overset domain problems, where the overlapping domains do not change over time or during grid refinement.

\subsection{The problem in one space dimension}
\label{sub:cont_problem}

Consider the following one-dimensional initial-boundary-value problem (IBVP) on $\Omega = [a,d]$:
\begin{equation}\label{eq:IBVP_general}
\begin{aligned}
    \boldsymbol{w}_t + \boldsymbol{f}(\boldsymbol{w})_x
    		& = 0,\quad && x \in (a,d), \ t > 0 \\
    \boldsymbol{w}(x,0)
    		& = \boldsymbol{w}_0(x), \quad && x \in [a,d], \\
	\boldsymbol{w}(a,t)
    		& = \boldsymbol{g}_L(t), \quad && t>0, \\
	\boldsymbol{w}(d,t)
    		& = \boldsymbol{g}_R(t), \quad && t>0,
\end{aligned}
\end{equation}
where $\boldsymbol{w} = [w_1,\dots,w_{D}]^T$ are the unknown conserved variables, $\boldsymbol{f} = [f_1,\dots,f_D]^T$ is a given flux function, $\boldsymbol{w}_0 = [(w_0)_1, \dots, (w_0)_D]^T$ is the initial data, and $\boldsymbol{g}_L = [(g_L)_1, \dots, (g_L)_D]^T, \boldsymbol{g}_R = [(g_R)_1, \dots, (g_R)_D]^T$ are inflow boundary data at the left and right boundary of the domain, respectively.
If there is no inflow boundary condition at one of the boundaries, say at $d$, then $\boldsymbol{g}_R(t) = \boldsymbol{w}(d,t)$---meaning that we don't make any restrictions on the solution there.
The scalar overset domain problem seeks the solution to \cref{eq:IBVP_general} by finding the solutions $u$ and $v$ on two domains $\Omega_u = [a,c]$ and $\Omega_v = [b,d]$:
\begin{equation}\label{eq:LR_IBVP_general}
	\text{(L-IBVP)} \left\{\begin{aligned}
		\boldsymbol{u}_t + \boldsymbol{f}(\boldsymbol{u})_x
			& = 0, \quad x \in (a,c), \\
		\boldsymbol{u}(x,0)
			& = \boldsymbol{w}_0(x), \\
		\boldsymbol{u}(a,t)
			& = \boldsymbol{g}_L(t), \\
		\boldsymbol{u}(c,t)
			& = \boldsymbol{v}(c,t),
	\end{aligned} \right. \qquad
	\text{(R-IBVP)} \left\{\begin{aligned}
		\boldsymbol{v}_t + \boldsymbol{f}(\boldsymbol{v})_x
			& = 0, \quad x \in (b,d), \\
		\boldsymbol{v}(x,0)
			& = \boldsymbol{w}_0(x), \\
		\boldsymbol{v}(b,t)
			& = \boldsymbol{u}(b,t), \\
		\boldsymbol{v}(d,t)
			& = \boldsymbol{g}_R(t), \\
	\end{aligned} \right.
\end{equation}
with $a < b < c < d$.
\Cref{fig:overlap_1d} illustrates the set-up of the overset domain problem \cref{eq:LR_IBVP_general}.

\begin{figure}[tb]
	\centering
	\includegraphics[width=0.6\textwidth]{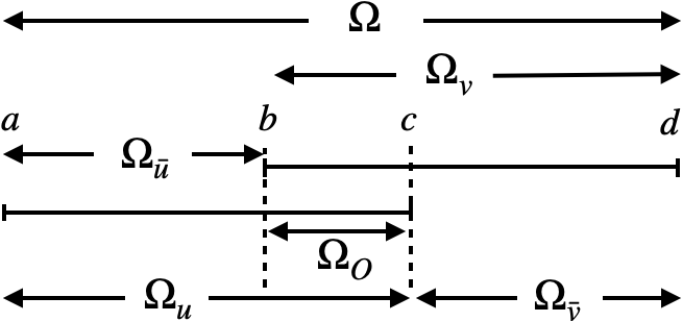}
  	\caption{
  	Overset domain definitions in one spatial dimension.
	The figure is taken from \cite{kopriva2022theoretical}.
  	}
  	\label{fig:overlap_1d}
\end{figure}

While the sub-cell SBP methodology proposed here applies to the general overset domain problem \cref{eq:LR_IBVP_general}, our analysis---proving conservation and energy stability---focuses on the scalar linear advection equation with $f(w) = \alpha w$ and $\alpha \in \R$.
Without loss of generality, we assume $\alpha > 0$.
In this case, \cref{eq:IBVP_general} becomes
\begin{equation}\label{eq:IBVP_scalar}
\begin{aligned}
    w_t + \alpha w_x
    		& = 0,\quad && x \in (a,d), \ t > 0 \\
    w(x,0)
    		& = w_0(x), \quad && x \in [a,d], \\
	w(a,t)
    		& = g_L(t), \quad && t>0.
\end{aligned}
\end{equation}
Furthermore, the overset domain problem \cref{eq:LR_IBVP_general} reduces to
\begin{equation}\label{eq:LR_IBVP_scalar}
	\text{(L-IBVP)} \left\{\begin{aligned}
		u_t + \alpha u_x
			& = 0, \quad x \in (a,c), \\
		u(x,0)
			& = w_0(x), \\
		u(a,t)
			& = g_L(t),
	\end{aligned} \right. \qquad
	\text{(R-IBVP)} \left\{\begin{aligned}
		v_t + \alpha v_x
			& = 0, \quad x \in (b,d), \\
		v(x,0)
			& = w_0(x), \\
		v(b,t)
			& = u(b,t).
	\end{aligned} \right.
\end{equation}
Notably, the original IBVP \cref{eq:IBVP_scalar} for $w$ is known to be well-posed and energy-stable \cite{kreiss2004initial}.
Below, we review some of the results from \cite{kopriva2022theoretical} that established similar fundamental properties of the solutions $u$ and $v$ for the overset domain problem \cref{eq:LR_IBVP_scalar}, including conservation and energy stability.

\subsection{Conservation}
\label{sub:cont_cons}

Recall that the exact solution of \cref{eq:IBVP_scalar} satisfies
\begin{equation}\label{eq:IBVP_cons}
	\frac{\d}{\d t} \int_{a}^{d} w \intd x = - \alpha \left[ w(d,t) - g_L(t) \right].
\end{equation}
This means that the total amount of the quantity $w$ (e.g., mass) is neither created nor destroyed inside the domain and only changes due to the flux across the boundaries.
This property is referred to as \emph{conservation}.
Conservation is an essential property in many physical applications, which should also hold for the overset IBVP \cref{eq:LR_IBVP_scalar}.

A natural counterpart of the linear function $\int_{a}^d w \intd x$ in \cref{eq:IBVP_cons} to the overset IBVPs \cref{eq:LR_IBVP_scalar} is the following \emph{overset integral}:
\begin{equation}\label{eq:overset_integral}
	\mathcal{I}(u,v)
		= \int_{a}^{c} u \intd x
		+ \int_{b}^{d} v \intd x
		- \int_{b}^{c} u \intd x.
\end{equation}
Note that the first and second integral in \cref{eq:overset_integral} both cover the overlap region $\Omega_{O} = \Omega_u \cap \Omega_v = [b,c]$.
The third integral removes the contributions of $u$ on $\Omega_{O}$ so that the overlap region does not contribute twice.

\begin{remark}\label{rem:convexCombo}
	In \cref{eq:overset_integral}, one can replace $\int_{b}^{c} u \intd x$ by a convex combination of $\int_{b}^{c} u \intd x$ and $\int_{b}^{c} v \intd x$ to account for left or right-traveling information while maintaining conservation and stability; see \cite{kopriva2022theoretical}.
\end{remark}

To establish \emph{conservation of the overset IBVP} \cref{eq:LR_IBVP_scalar}, we show that $\mathcal{I}(u,v)$ satisfies a relation similar to \cref{eq:IBVP_cons}.
Let $u$ and $v$ be solutions of the overset IBVP \cref{eq:LR_IBVP_scalar}.
Differentiation in time and applying IBP in space to the integrals in \cref{eq:overset_integral} yields
\begin{equation}\label{eq:LR_IBVP_cons1}
	\frac{\d}{\d t} \mathcal{I}(u,v)
		= - \alpha \left[ u(c,t) - u(a,t) \right]
		- \alpha \left[ v(d,t) - v(b,t) \right]
		+ \alpha \left[ u(c,t) - u(b,t) \right].
\end{equation}
Recall that $u(a,t) = g_L(t)$ and $v(b,t) = u(b,t)$ due to the boundary and interface conditions in \cref{eq:LR_IBVP_scalar}.
Hence, \cref{eq:LR_IBVP_cons1} reduces to
\begin{equation}\label{eq:LR_IBVP_cons2}
	\frac{\d}{\d t} \mathcal{I}(u, v)
		= - \alpha \left[ v(d,t) - g_L(t) \right].
\end{equation}
Notably, \cref{eq:LR_IBVP_cons2} is similar to \cref{eq:IBVP_cons}, in the sense that the overset integral only changes due to boundary effects.
Hence, \cref{eq:LR_IBVP_cons2} establishes conservation of the overset IBVP \cref{eq:LR_IBVP_scalar}.

\subsection{Energy stability}
\label{sub:cont_energy}

% Original IBVP
Besides being conservative, exact solutions of the linear advection equation \cref{eq:IBVP_scalar} are also energy bounded.
That is, the growth of their energy over time is bounded in terms of the boundary data, as demonstrated by the energy method:
\begin{equation}\label{eq:cont_energy_rate}
	\frac{\d}{\d t} \| w \|_{L^2}^2
		= - 2 \alpha \int_{a}^{d} w \partial_x w \intd x
		= -\alpha \left[  w(d,t)^2 - g_L(t)^2 \right]
		\leq \alpha g_L(t)^2.
\end{equation}

% Overset IBVP
To establish stability for the overset IBVP \cref{eq:LR_IBVP_scalar}, a similar bound is required.
To this end, we follow \cite{kopriva2022theoretical} and introduce the following \emph{overset energy}:
\begin{equation}\label{eq:overset_energy}
	\mathcal{E}(u,v)
		= \| u \|_{\Omega_u}^2
		+ \| v \|_{\Omega_v}^2
		- \| u \|_{\Omega_O}^2,
\end{equation}
where $\| u \|_{\Omega_u}^2 = \int_a^c u^2 \intd x$, $\| v \|_{\Omega_v}^2 = \int_b^d v^2 \intd x$, and $\| u \|_{\Omega_O}^2 = \int_b^c u^2 \intd x$.
The idea here is the same as for the overset integral $\mathcal{I}(u,v)$ in \cref{eq:overset_integral}, which is not to count the energy in $\Omega_O$ twice.

% Energy stability
To establish energy stability of the overset IBVP \cref{eq:LR_IBVP_scalar}, we show that $\mathcal{E}(u,v)$ satisfies a relation similar to \cref{eq:cont_energy_rate}.
As before, let $u$ and $v$ be solutions of the overset IBVP \cref{eq:LR_IBVP_scalar}.
The energy method implies
\begin{equation}\label{eq:LR_IBVP_energy1}
	\frac{\d}{\d t} \mathcal{E}(u,v)
		= -\alpha \left[ u(c,t)^2 - u(a,t)^2 \right]
		- \alpha \left[  v(d,t)^2 - v(b,t)^2 \right]
		+ \alpha \left[  u(c,t)^2 - u(b,t)^2 \right].
\end{equation}
Importantly, \cref{eq:LR_IBVP_energy1} follows from applying IBP to each of the three sub-domains $\Omega_u$, $\Omega_v$, and $\Omega_O$.
Recalling from \cref{eq:LR_IBVP_scalar} that $u(a,t) = g_L(t)$ and $v(b,t) = u(b,t)$ transforms \cref{eq:LR_IBVP_energy1} into
\begin{equation}\label{eq:LR_IBVP_energy2}
	\frac{\d}{\d t} \mathcal{E}(u,v)
		= - \alpha \left[ v(d,t)^2 - g_L(t)^2 \right]
		\leq \alpha g_L(t)^2.
\end{equation}
Note that \cref{eq:LR_IBVP_energy2} is similar to \cref{eq:cont_energy_rate}, bounding the growth of the energy in terms of the boundary data.
\section{Sub-cell SBP operators}
\label{sec:operators}

As demonstrated above, conservation and energy stability for exact solutions of the overset domain problem \cref{eq:LR_IBVP_scalar} can be shown using IBP on each of the three sub-domains $\Omega_u = [a,c]$, $\Omega_v = [b,d]$, and $\Omega_O = [b,c]$.
To derive conservative and energy-stable semi-discretizations of \cref{eq:LR_IBVP_scalar}, it is therefore desirable to have discrete derivative and integration operators that mimic IBP on given sub-domains---henceforth also called sub-cells.
In \Cref{def:subcell_SBP} below, we propose such operators, which we refer to as sub-cell SBP operators.

Consider a generic cell $\Omega_{\rm ref} = [\omega_L,\omega_R]$ and let $\omega_M$ be a point between $\omega_L$ and $\omega_R$, separating the cell $\Omega_{\rm ref} = [\omega_L,\omega_R]$ into the two sub-cells $\Omega_{L} = [\omega_L,\omega_M]$ and $\Omega_{R} = [\omega_M,\omega_R]$.
Let $\mathbf{x}_L = [ (x_L)_1,\dots,(x_L)_{N_L}]$ and $\mathbf{x}_R = [ (x_R)_1,\dots,(x_R)_{N_R}]$ be vectors of nodes on $\Omega_L$ and $\Omega_R$, respectively, with $\omega_L \leq (x_L)_1 < \dots < (x_L)_{N_L} \leq \omega_M$ and $\omega_M \leq (x_R)_1 < \dots < (x_R)_{N_R} \leq \omega_R$.
Furthermore, $\mathbf{x} = [\mathbf{x}_L, \mathbf{x}_R] \in \Omega_{\rm ref}^N$ with $N = N_L + N_R$ is a vector of nodes on $\Omega_{\rm ref}$.
Notably, $(x_L)_{N_L} = (x_R)_1$ if both, $\mathbf{x}_L$ and $\mathbf{x}_R$, include $\omega_M$.
Finally, slightly abusing notation, we write $x_n \in \mathbf{x}_{L/R}$ when $x_n$ is a point in the grid $\mathbf{x}_{L/R}$ on $\Omega_{L/R}$, i.e., there exists an index $i$ such that $x_n = (x_{L/R})_i$.

\begin{definition}\label{def:subcell_SBP}
	By $(\Omega_{\rm ref}, \Omega_{L}, \Omega_{R})$ we denote a triple containing a cell $\Omega_{\rm ref} = [\omega_L, \omega_R]$ and the two sub-cells $\Omega_{L} = [\omega_L,\omega_M]$ and $\Omega_{R} = [\omega_M,\omega_R]$ with nodes $\mathbf{x}_L \in \Omega_L^{N_L}$ and $\mathbf{x}_R  \in \Omega_R^{N_R}$.
	Furthermore, let $\mathcal{F} \subset C^1(\Omega_{\rm ref})$ be a finite-dimensional function space.
	An operator $D = P^{-1} Q \in \R^{N \times N}$ approximating $\partial_x$ on the nodes $\mathbf{x} = [\mathbf{x}_L, \mathbf{x}_R] \in \Omega_{\rm ref}^N$ with $N = N_L + N_R$ is an \emph{$\mathcal{F}$-exact sub-cell SBP operator for $(\Omega_{\rm ref}, \Omega_{L}, \Omega_{R})$} if the following conditions are satisfied:
	\begin{enumerate}
		\item[(i)]
		$D \mathbf{f} = \mathbf{f'}$ for all $f \in \mathcal{F}$

		\item[(ii)]
		$P_{L/R} = \diag( \mathbf{p}_{L/R} )$ with $(p_{L/R})_n > 0$ if $x_n \in \mathbf{x}_{L/R}$ and $(p_{L/R})_n = 0$ otherwise

		\item[(iii)]
		$P_{L/R} D = Q_{L/R}$ with $Q_{L/R} + Q_{L/R}^T = B_{L/R}$

		\item[(iv)]
		$\mathbf{f}^T B_{L/R} \mathbf{g} = fg|_{\partial \Omega_{L/R}}$ for all $f, g \in \mathcal{F}$

		\item[(v)]
		$P = P_L + P_R$, $B = B_L + B_R$, and $Q = Q_L + Q_R$

	\end{enumerate}
	Here, $\mathbf{f} = [ f(x_1), \dots, f(x_N) ]^T$ and $\mathbf{f'} = [ f'(x_1), \dots, f'(x_N) ]^T$ denote the nodal values of $f \in \mathcal{F}$ and its derivative at the nodes $\mathbf{x} = [x_1,\dots,x_N]$.
\end{definition}

Relation (i) ensures that $D$ accurately approximates $\partial_x$ at the nodes $\mathbf{x}$ by requiring $D$ to be exact for all functions from $\mathcal{F}$.
Condition (ii) guarantees that $P_L$ and $P_R$ induce a discrete inner product and norm on $\Omega_{L}$ and $\Omega_R$, respectively, which are given by $\scp{\mathbf{f}}{\mathbf{g}}_{P_{L/R}} = \mathbf{f}^T P_{L/R}  \mathbf{g} = \sum_{n=1}^N (p_{L/R})_n f_n g_n$ and $\|\mathbf{f}\|^2_{P_{L/R}} = \scp{\mathbf{f}}{\mathbf{f}}_{P_{L/R}}$ for $\mathbf{f},\mathbf{g} \in \R^N$.
Relations (iii) encode the sub-cell SBP properties on the sub-cells $\Omega_{L}$ and $\Omega_R$.
Recall that IBP on $\Omega_{L/R}$ reads
\begin{equation}\label{eq:subcell_IBP}
	\int_{\Omega_{L/R}} f (\partial_x g) \intd x + \int_{\Omega_{L/R}} (\partial_x f) g \intd x
		= fg|_{\partial \Omega_{L/R}},
		\quad \forall f, g \in C^1(\Omega_{L/R}).
\end{equation}
The discrete versions of \cref{eq:subcell_IBP}, which follows from (iii), is
\begin{equation}\label{eq:subcell_SBP}
	\mathbf{f}^T P_{L/R} ( D \mathbf{g} ) + ( D \mathbf{f} )^T P_{L/R} \mathbf{g}
		= \mathbf{f}^T B_{L/R} \mathbf{g},
		\quad \forall \mathbf{f}, \mathbf{g} \in \R^N.
\end{equation}
Note that the two terms on the left-hand side of \cref{eq:subcell_SBP} approximate the related terms on the left-hand side of \cref{eq:subcell_IBP}.
Moreover, (iv) in \Cref{def:subcell_SBP} ensures that also the right-hand side of \cref{eq:subcell_SBP} accurately approximates the right-hand side of \cref{eq:subcell_IBP}.
To this end, the boundary operators $B_L$ and $B_R$ must be exact for all functions $f,g \in \mathcal{F}$.
Finally, (v) implies that $D = P^{-1} Q$ is an FSBP operator with $Q + Q^T = B$, mimicking IBP on the whole cell $\Omega_{\rm ref}$.
\Cref{fig:schematic_sub_cell_SBP} illustrates the basic idea behind the proposed sub-cell SBP operators.

\begin{remark}
	\Cref{def:subcell_SBP} is based on the recently introduced concept of FSBP operators for general (potentially non-polynomial) function spaces described in \cite{glaubitz2022summation,glaubitz2023multi,glaubitz2024summation,glaubitz2024optimization}.
	Making this generalization enables the systematic development of a broad class of provably stable schemes \cite{glaubitz2022energy}, without substantially increasing the complexity of our discussion.
\end{remark}

\begin{figure}[tb]
\centering
\resizebox{.99\textwidth}{!}{%
\begin{tikzpicture}[scale=2.75]

	%% Define the function
    \def\funct#1{1 - 0.2*#1 + 0.2*(#1*#1-2) - 0.03*#1*#1*#1}

    %% Domain
	\coordinate (w_L) at (0,0);
	\coordinate (w_M) at (3,0);
	\coordinate (w_R) at (5,0);
	\coordinate (w_LM) at ($(w_L)!0.5!(w_M)$);
	\coordinate (w_MR) at ($(w_M)!0.5!(w_R)$);
	
    % Draw lower domain line
	\draw[thick, blue1] (w_L) -- (w_M);
	\draw[thick, red1] (w_M) -- (w_R);

	% Vertical ticks (lower)
	\foreach \x in {0,1.5,3,4,5} {
    		\draw[thick] (\x,-0.05) -- (\x,0.05);
	}

    % Labels for domains
    \node[above, blue1] at (1.3,0) {$\Omega_L$};
    \node[above, red1] at (4.2,0) {$\Omega_R$};
    
    % Labels for points
	\node[below] at (w_L) {$\omega_L = (x_L)_1$};
	\node[below] at (w_LM) {$(x_L)_2$};
	\node[below] at (w_M) {$(x_L)_3 = \omega_M = (x_R)_1$};
	\node[below] at (w_MR) {$(x_R)_2$};
	\node[below] at (w_R) {$(x_R)_3 = \omega_R$};
	
	% Draw function 
	\draw[domain=0:5, thick, smooth, variable=\x] plot ({\x}, {\funct{\x}});
	
	% Draw function values 
	\foreach \xx in {3,4,5} {
		\draw[red1, thick, fill]
        (\xx,{\funct{\xx}}) ++(0,0.05) -- ++(0.05,-0.05) -- ++(-0.05,-0.05)
        -- ++(-0.05,0.05) -- cycle;
   	}
	\foreach \xx in {0,1.5,3} {
		\filldraw[blue1] (\xx, {\funct{\xx}}) circle (0.8pt); % marker
	}
	 
	% Draw lines from b and c
	\foreach \xx in {0,1.5,3} {
		\draw[thick, blue1, dashed] (\xx,0) -- (\xx, {\funct{\xx}}); % marker
	}
	\foreach \xx in {3,4,5} {
		\draw[thick, red1, dotted] (\xx,0) -- (\xx, {\funct{\xx}}); % marker
	}

	% Sub-cell SBP formula - left 
	\draw[thick, blue1, decorate, decoration={brace,amplitude=20pt,raise=4ex}]
  (0,0.85) -- (3,0.85) node[midway, yshift=5em]{{$
  		\mathbf{u}^T P_L D \mathbf{v} + \mathbf{u}^T D^T P_L \mathbf{v} = \mathbf{u}^T B_L \mathbf{v}
  	$}};

	% Sub-cell IBP formula - left 
	\draw[thick, blue1, decorate, decoration={brace,amplitude=20pt,mirror,raise=4ex}]
  (w_L) -- (w_M) node[midway, yshift=-5em]{{$
  		\int_{\Omega_L} u\,(\partial_x v) + \int_{\Omega_L} (\partial_x u)\, v = uv|_{\partial \Omega_L}
  	$}};

	% Sub-cell SBP formula - right 
	\draw[thick, red1, decorate, decoration={brace,amplitude=20pt,raise=4ex}]
  (3,0.85) -- (5,0.85) node[midway, yshift=5em]{{$
  		\mathbf{u}^T P_R D \mathbf{v} + \mathbf{u}^T D^T P_R \mathbf{v} = \mathbf{u}^T B_R \mathbf{v}
  	$}};

	% Sub-cell IBP formula - right 
	\draw[thick, red1, decorate, decoration={brace,amplitude=20pt,mirror,raise=4ex}]
  (w_M) -- (w_R) node[midway, yshift=-5em]{{$
  		\int_{\Omega_R} u\,(\partial_x v) + \int_{\Omega_R} (\partial_x u)\, v = uv|_{\partial \Omega_R}
  	$}};

\end{tikzpicture}
}%
\caption{
  	Schematic illustration of how sub-cell SBP operators (top) mimic IBP (bottom) on the left (blue) and right (red) sub-cell
}
\label{fig:schematic_sub_cell_SBP}
\end{figure}
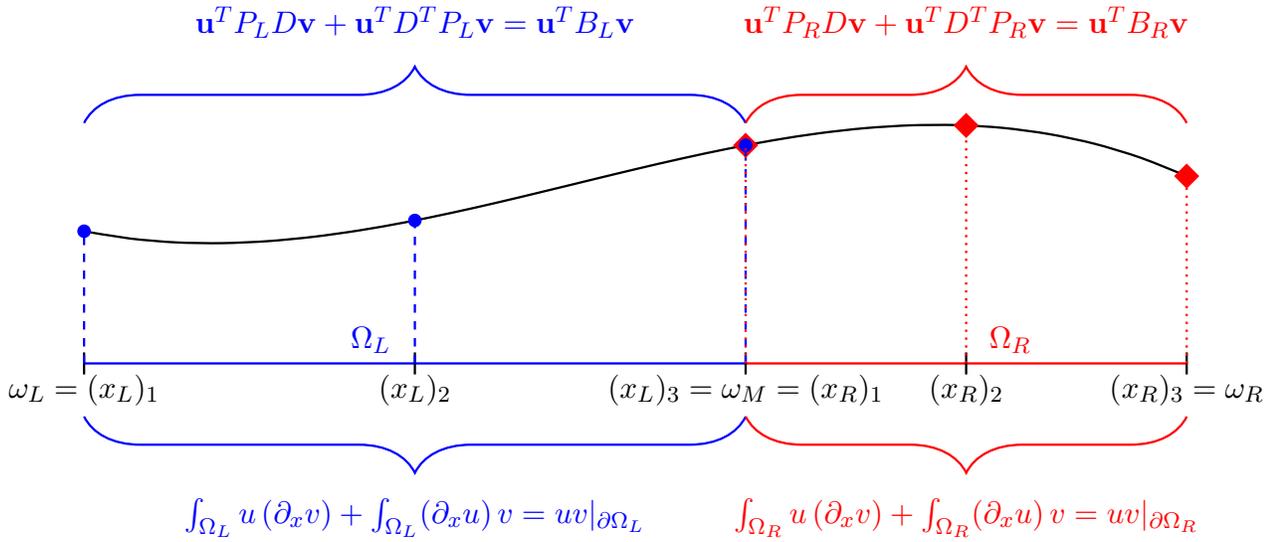
\section{Application to overset domain methods}
\label{sec:application}

We now leverage the proposed sub-cell SBP operators and simultaneous approximation terms (SATs) to develop SBP-SAT semi-discretizations of the overset domain problem \cref{eq:LR_IBVP_general} that are provably conservative and energy-stable for the linear advection problem \cref{eq:LR_IBVP_scalar}.
We again focus on fixed overset domain problems, where the overlapping domains do not change over time or under grid refinement.

\subsection{Single-block overset domain methods}
\label{sub:application_single}

Consider \cref{eq:LR_IBVP_scalar} on $\Omega = [a,d]$ with the two overlapping domains $\Omega_u = [a,c]$ and $\Omega_v = [b,d]$.
Following \Cref{def:subcell_SBP}, let $D_u$ and $D_v$ be sub-cell SBP operators for $(\Omega_u,\Omega_{\bar{u}},\Omega_O)$ and $(\Omega_v,\Omega_{O},\Omega_{\bar{v}})$ approximating $\partial_x$ on the nodes $\mathbf{x}_{u} = [\mathbf{x}_{\bar{u}}; \mathbf{x}_{O_u}] \in \Omega_u^{N_u}$ and $\mathbf{x}_{v} = [\mathbf{x}_{O_v}; \mathbf{x}_{\bar{v}}] \in \Omega_v^{N_v}$, respectively.
That is, $D_u$ is a sub-cell SBP operator on $\Omega_u$ that mimics IBP on $\Omega_u$, $\Omega_{\bar{u}}$, and $\Omega_O$.
We denote the associated operators on $\Omega_u$ as $P_u, Q_u, B_u$, on $\Omega_{\bar{u}}$ as $P_{\bar{u}}, Q_{\bar{u}}, B_{\bar{u}}$, and on $\Omega_O$ as $P_{O_u}, Q_{O_u}, B_{O_u}$.
Specifically, $(p_{\bar{u}})_n > 0$ if and only if $(x_u)_n \in \mathbf{x}_{\bar{u}}$ and $(p_{O_u})_n > 0$ if and only if $(x_u)_n \in \mathbf{x}_{O_u}$.
Similarly, $D_v$ is a sub-cell SBP operator on $\Omega_v$ that mimics IBP on $\Omega_v$, $\Omega_{O}$, and $\Omega_{\bar{v}}$.
We denote the associated operators on $\Omega_v$ as $P_v, Q_v, B_v$, on $\Omega_{O}$ as $P_{O_v}, Q_{O_v}, B_{O_v}$, and on $\Omega_{\bar{v}}$ as $P_{\bar{v}}, Q_{\bar{v}}, B_{\bar{v}}$.

Using the sub-cell SBP operators $D_u$ and $D_v$ introduced above, we consider the following \emph{single-block SBP-SAT semi-discretization} of the overset domain problem \cref{eq:LR_IBVP_general}:
\begin{equation}\label{eq:LR_disc_general}
	\mathbf{u}_t + D_u \mathbf{f}_u
		= P_u^{-1} \mathbb{S}_u, \quad
	\mathbf{v}_t + D_v \mathbf{f}_v
		= P_v^{-1} \mathbb{S}_v
\end{equation}
Here, $\mathbf{u} = [\boldsymbol{u}_1,\dots,\boldsymbol{u}_{N_u}]^T$ and $\mathbf{v} = [\boldsymbol{v}_1,\dots,\boldsymbol{v}_{N_v}]^T$ are the vectors of nodal values of the numerical solutions on $\Omega_u$ and $\Omega_v$ at the nodes $\mathbf{x}_u = [(x_u)_1,\dots,(x_u)_{N_u}] \in \Omega_u^{N_u}$ and $\mathbf{x}_v = [(x_v)_1,\dots,(x_v)_{N_v}] \in \Omega_v^{N_v}$, respectively.
At the same time, $\mathbf{f}_u = [\boldsymbol{f}(\boldsymbol{u}_1), \dots, \boldsymbol{f}(\boldsymbol{u}_{N_u})]^T$ and $\mathbf{f}_v = [\boldsymbol{f}(\boldsymbol{v}_1), \dots, \boldsymbol{f}(\boldsymbol{v}_{N_v})]^T$ are vectors containing the flux values for $\mathbf{u}$ and $\mathbf{v}$.
Furthermore, $\mathbb{S}_u \in \R^{N_u}$ and $\mathbb{S}_v \in \R^{N_v}$ are \emph{sub-cell} SATs defined as
\begin{equation}\label{eq:SATs}
\begin{aligned}
	\mathbb{S}_u
		& = \mathbf{e}_{a,u} [ \vecfnum( \boldsymbol{g}_L, \boldsymbol{u}_a ) - \boldsymbol{f}( \boldsymbol{u}_a ) ]
			- \mathbf{e}_{c,u} [ \beta\vecfnum( \boldsymbol{u}_c, \boldsymbol{v}_{c_L} )
			+ (1-\beta) \vecfnum( \boldsymbol{u}_c , \boldsymbol{v}_{c_R}) - \boldsymbol{f}( \boldsymbol{u}_c ) ] \\
		& \quad + \mathbf{e}_{b_R,u} [ \vecfnum( \boldsymbol{u}_{b_L}, \boldsymbol{u}_{b_R} ) - \boldsymbol{f}( \boldsymbol{u}_{b_R} ) ]
			- \mathbf{e}_{b_L,u} [ \vecfnum( \boldsymbol{u}_{b_L}, \boldsymbol{u}_{b_R} ) - \boldsymbol{f}( \boldsymbol{u}_{b_L} ) ], \\
	\mathbb{S}_v
		& = \mathbf{e}_{b,v} [ \beta \vecfnum( \boldsymbol{u}_{b_L}, \boldsymbol{v}_b )
			+ (1-\beta)\vecfnum( \boldsymbol{u}_{b_R}, \boldsymbol{v}_b ) - \boldsymbol{f}( \boldsymbol{v}_b )] - \mathbf{e}_{d,v} [ \vecfnum( \boldsymbol{v}_d, \boldsymbol{g}_R ) - \boldsymbol{f}( \boldsymbol{v}_d ) ] \\
		& \quad + \mathbf{e}_{c_R,v} [ \vecfnum( \boldsymbol{v}_{c_L}, \boldsymbol{v}_{c_R} ) - \boldsymbol{f}( \boldsymbol{v}_{c_R} ) ]
			- \mathbf{e}_{c_L,v} [ \vecfnum( \boldsymbol{v}_{c_L}, \boldsymbol{v}_{c_R} ) - \boldsymbol{f}( \boldsymbol{v}_{c_L} ) ].
\end{aligned}
\end{equation}

\begin{remark}
	The SATs in \cref{eq:SATs} are general and allow coupling of left- and right-traveling waves.
	The numerical flux terms are a convex combination at the overlap points $b$ and $c$ with the possibly time-dependent parameter $\beta$.
	If $\beta=1$, the SATs are upwind for right-traveling waves ($\alpha > 0$).
	If $\beta=0$, they are upwind for left-traveling waves ($\alpha < 0$).
	For a combination of right- and left-traveling waves, the design of an appropriate $\beta$ value may be similar to flux vector splitting strategies.
	Whatever value of the combination parameter $\beta$ is chosen for the SATs in \cref{eq:SATs}, it must match the parameter used in the convex combination of the overlap integral, as in \Cref{rem:convexCombo}.
\end{remark}

The SATs in \cref{eq:SATs} weakly enforce the boundary and interface conditions in \cref{eq:LR_IBVP_scalar} by employing a numerical flux function $\vecfnum$.
Specifically, the first and second terms for $\mathbb{S}_u$ in \cref{eq:SATs} weakly enforce the inflow boundary condition at the left and right domain boundary, respectively.
At the same time, the third and fourth terms for $\mathbb{S}_u$ in \cref{eq:SATs} weakly couple the two sub-elements $[a,b]$ and $[b,c]$.
These two terms are only present when using a sub-cell SBP operator and vanish when $D_u$ is a standard SBP operator.
Moreover, the SATs involve discrete projection operators $\mathbf{e}_{a,u}, \mathbf{e}_{b_L,u}, \mathbf{e}_{b_R,u}, \mathbf{e}_{c,u} \in \R^{N_u}$, and $\mathbf{e}_{b,v}, \mathbf{e}_{c_L,b}, \mathbf{e}_{c_R,b}, \mathbf{e}_{d,b} \in \R^{N_v}$ as well as the projected numerical solutions $\boldsymbol{u}_a = \mathbf{e}_{a,u}^T \mathbf{u} \approx \boldsymbol{u}(a)$, $\boldsymbol{u}_{b_L} = \mathbf{e}_{b_L,u}^T \mathbf{u} \approx \boldsymbol{u}(b)$, $\boldsymbol{u}_{b_R} = \mathbf{e}_{b_R,u}^T \mathbf{u} \approx \boldsymbol{u}(b)$, $\boldsymbol{u}_c = \mathbf{e}_{c,u}^T \mathbf{u} \approx \boldsymbol{u}(c)$ and $\boldsymbol{v}_b = \mathbf{e}_{b,v}^T \mathbf{v} \approx \boldsymbol{v}(b)$, $\boldsymbol{v}_{c_L} = \mathbf{e}_{c_L,v}^T \mathbf{v} \approx \boldsymbol{v}(c)$, $\boldsymbol{v}_{c_R} = \mathbf{e}_{c_R,v}^T \mathbf{v} \approx \boldsymbol{v}(c)$, $\boldsymbol{v}_d = \mathbf{e}_{d,v}^T \mathbf{v} \approx \boldsymbol{v}(d)$.
We discuss these projections and their importance in detail in \Cref{sub:application_projections}.

For the subsequent theoretical analysis, we focus on the linear advection equation \cref{eq:LR_IBVP_scalar} with $\alpha>0$.
In this case, our analysis requires only $D_u$ to be a sub-cell SBP operator, whereas $D_v$ may be a standard SBP operator.
Hence, the last two terms of $\mathbb{S}_v$ in \cref{eq:SATs} vanish.
Furthermore, we then have $f(w) = \alpha w$ and we use the usual full-upwind numerical flux $\fnum( w_L, w_R ) = \alpha w_L$.
Under these assumptions, the sub-cell semi-discretization \cref{eq:LR_IBVP_general} and SATs \cref{eq:SATs} reduce to
\begin{equation}\label{eq:LR_disc}
\begin{aligned}
	& \mathbf{u}_t + \alpha D_u \mathbf{u}
		= P_u^{-1}\mathbb{S}_u, \quad
		&& \mathbb{S}_u = \alpha \mathbf{e}_{a,u} [ g_L(t) - u_a ] + \alpha \mathbf{e}_{b_R,u} [ u_{b_L} - u_{b_R} ], \\
	& \mathbf{v}_t + \alpha D_v \mathbf{v}
		= P_v^{-1}\mathbb{S}_v, \quad
		&& \mathbb{S}_v = \alpha \mathbf{e}_{b,v} [ u_{b_L} - v_b ].
\end{aligned}
\end{equation}
%A few remarks are in order.

\begin{remark}\label{rem:projections}
	We write ``$u_{b_L}$", ``$u_{b_R}$" and ``$\mathbf{e}_{b_L,u}$", ``$\mathbf{e}_{b_R,u}$" instead of ``$u_{b}$" and ``$\mathbf{e}_{b,u}$" because $\mathbf{e}_{b_L,u}$ and $\mathbf{e}_{b_R,u}$ are one-sided projection operators that are supported on the nodes in $\Omega_{\bar{u}}$ and $\Omega_{O}$, respectively.
	That is, $(e_{b_L,u})_n = 0$ if $(x_u)_n \in \mathbf{x}_{O_u}$ and $(e_{b_R,u})_n = 0$ if $(x_u)_n \in \mathbf{x}_{\bar{u}}$.
	This restriction ensures the existence of the proposed sub-cell SBP operators, as will be shown in \Cref{sec:existence}.
\end{remark}

\subsection{On the discrete projection and boundary operators}
\label{sub:application_projections}

% Intro
We now provide details on the discrete projection operators $\mathbf{e}_{a,u}$, $\mathbf{e}_{b_L,u}$, $\mathbf{e}_{b_R,u}$, $\mathbf{e}_{c,u}$ on $\Omega_u$ and their connection to the boundary operators $B_L$, $B_R$, and $B$ in \cref{def:subcell_SBP}.
The same arguments also apply to the discrete projection operators $\mathbf{e}_{b,v}$, $\mathbf{e}_{c_L,v}$, $\mathbf{e}_{c_R,v}$, $\mathbf{e}_{d,v}$ on $\Omega_v$.

% Projection operators
The projection operators $\mathbf{e}_{a,u}, \mathbf{e}_{c,u} \in \R^{N_u}$ use the nodal values of $u$ on the nodes $\mathbf{x}_u \in \Omega_u^{N_u}$ to approximate $u(a)$ and $u(c)$ as $\mathbf{e}_{a,u}^T \mathbf{u}$ and $\mathbf{e}_{c,u}^T \mathbf{u}$, respectively.
For instance, if $a$ and $c$ coincide with the first and last node of $\mathbf{x}_u$, then $\mathbf{e}_{a,u} = [1,0,\dots,0]^T$ and $\mathbf{e}_{c,u} = [0,\dots,0,1]^T$.
Furthermore, we denote by $\mathbf{e}_{b_L,u} \in \R^{N_u}$ a one-sided projection operator with $\mathbf{e}_{b_L,u}^T \mathbf{u} \approx u(b)$ that is supported on the grid for $\Omega_{\bar{u}}$, i.e., $(e_{b_L,u})_n = 0$ if $(x_u)_n \in \mathbf{x}_{O_u}$.
Similarly, $\mathbf{e}_{b_R,u} \in \R^{N_u}$ denotes a one-sided projection operator with $\mathbf{e}_{b_R,u}^T \mathbf{u} \approx u(b)$ that is supported on the grid for $\Omega_{O_u}$, i.e., $(e_{b_R,u})_n = 0$ if $(x_u)_n \in \mathbf{x}_{\bar{u}}$.
We use one-sided projections at the point $b$ to ensure conservation and energy-stability as well as the existence of the sub-cell SBP operator, as discussed in \Cref{sec:existence}.

% Boundary operators
Using the above projections, we make the common assumption that the boundary operators $B_{\bar{u}}$ and  $B_{O_u}$ can be written as
\begin{equation}\label{eq:B_form}
	B_{\bar{u}} = \mathbf{e}_{b_L,u} \mathbf{e}_{b_L,u}^T - \mathbf{e}_{a,u} \mathbf{e}_{a,u}^T, \quad
	B_{O_u} =  \mathbf{e}_{c,u} \mathbf{e}_{c,u}^T - \mathbf{e}_{b_R,u} \mathbf{e}_{b_R,u}^T.
\end{equation}
Moreover, $B_u = B_{\bar{u}} + B_{O_u}$.
Note that $B_{\bar{u}}, B_{O_u} \in \R^{N_u \times N_u}$ in \cref{eq:B_form} are boundary operators which approximate the right-hand sides in \cref{eq:subcell_IBP} as they satisfy
\begin{equation}
\begin{aligned}
	\mathbf{f}^T B_{\bar{u}} \mathbf{g}
		= ( \mathbf{e}_{b_L,u}^T \mathbf{f} ) ( \mathbf{e}_{b_L,u}^T \mathbf{g} ) - ( \mathbf{e}_{a,u}^T \mathbf{f} ) ( \mathbf{e}_{a,u}^T \mathbf{g} )
		& \approx
		f(b) g(b) - f(a) g(a)
		= fg|_a^b, \\
	\mathbf{f}^T B_{O_u} \mathbf{g}
		= ( \mathbf{e}_{c,u}^T \mathbf{f} ) ( \mathbf{e}_{c,u}^T \mathbf{g} ) - ( \mathbf{e}_{b_R,u}^T \mathbf{f} ) ( \mathbf{e}_{b_R,u}^T \mathbf{g} )
		& \approx
		f(c) g(c) - f(b) g(b)
		= fg|_b^c.
\end{aligned}
\end{equation}
Notably, the operators $B_{\bar{u}}$ and $B_{O_u}$ above satisfy the exactness conditions (iv) in \Cref{def:subcell_SBP} if and only if the projection operators $\mathbf{e}_{a,u}, \mathbf{e}_{b_L,u}, \mathbf{e}_{b_R,u}, \mathbf{e}_{c,u}$ satisfy
\begin{equation}\label{eq:projections_exactness_u}
	\mathbf{e}_{a,u}^T \mathbf{f} = f(a), \quad
	\mathbf{e}_{b_L,u}^T \mathbf{f} = f(b), \quad
	\mathbf{e}_{b_R,u}^T \mathbf{f} = f(b), \quad
	\mathbf{e}_{c,u}^T \mathbf{f} = f(c), \quad
	\forall f \in \mathcal{F}.
\end{equation}
That is, the projection operators must be exact for all vectors corresponding to the nodal values of functions within the function space for which the sub-cell SBP is exact.
Similarly, we have $B_v = \mathbf{e}_{d,v} \mathbf{e}_{d,v}^T - \mathbf{e}_{b,v} \mathbf{e}_{b,v}^T$ and $\mathbf{f}^T B_v \mathbf{g} \approx fg|_b^d$.
For brevity, we write $u_a = \mathbf{e}_{a,u}^T \mathbf{u}$, $u_{b_L} = \mathbf{e}_{b_L,u}^T \mathbf{u}$, $u_{b_R} = \mathbf{e}_{b_R,u}^T \mathbf{u}$, $u_c = \mathbf{e}_{c,u}^T \mathbf{u}$ and $v_b = \mathbf{e}_{b,v}^T \mathbf{v}$, $v_d = \mathbf{e}_{d,v}^T \mathbf{v}$.
We can thus rewrite, for instance, $\mathbf{f}^T B_{\bar{u}} \mathbf{g}$ as $f_{b_L} g_{b_L} - f_a g_a$ and $\mathbf{f}^T B_{v} \mathbf{g}$ as $f_{d} g_{d} - f_b g_b$.

\subsection{Proving discrete conservation}
\label{sub:application_conservation}

We now prove conservation for the sub-cell semi-discretization \cref{eq:LR_disc} of the overset domain problem for the linear advection equation \cref{eq:LR_IBVP_scalar} with $\alpha>0$.
To this end, consider the overset integral $\mathcal{I}(u,v)$ in \cref{eq:overset_integral}.
A discrete version of the overset integral for the sub-cell semi-discretization \cref{eq:LR_disc} is given by
\begin{equation}\label{eq:overset_integral_discrete}
	I(\mathbf{u},\mathbf{v})
		= \mathbf{1}^T P_u \mathbf{u}
		+ \mathbf{1}^T P_v \mathbf{v}
		- \mathbf{1}^T P_{O_u} \mathbf{u},
\end{equation}
where $\mathbf{1}^T P_u \mathbf{u} \approx \int_a^c u \intd x$, $\mathbf{1}^T P_v \mathbf{v} \approx \int_b^d v \intd x$, and $\mathbf{1}^T P_{O_u} \mathbf{u} \approx \int_b^c u \intd x$.
Recall that \cref{eq:LR_IBVP_cons2} established conservation for solutions of the continuous overset IBVP \cref{eq:LR_IBVP_scalar}.
It is therefore desirable to mimic \cref{eq:LR_IBVP_cons2} for the numerical scheme.
To this end, we need to show that the semi-discretization \cref{eq:LR_disc} satisfies
\begin{equation}\label{eq:conservation_discrete4}
	\frac{\d}{\d t} I(\mathbf{u},\mathbf{v})
		= - \alpha \left[ v_d - g_L(t) \right],
\end{equation}
which is the discrete counterpart to \cref{eq:LR_IBVP_cons2}.

\begin{lemma}\label{lem:cons}
	Let $D_u$ be a sub-cell SBP operator for $(\Omega_u,\Omega_{\bar{u}},\Omega_O)$ and let $D_v$ be an SBP operator on $\Omega_v$.
	Furthermore, let $B_{\bar{u}}$ and $B_{O_u}$ be as in \cref{eq:B_form} with $(e_{a,u})_n = 0$ if $(x_u)_n \in \mathbf{x}_{O_u}$ and $(e_{b_R,u})_n = 0$ if $(x_{u})_n \in \mathbf{x}_{\bar{u}}$.
	Then, the sub-cell semi-discretization \cref{eq:LR_disc} is conservative, i.e., \cref{eq:conservation_discrete4} holds.
\end{lemma}

\begin{proof}
Recall from (v) in \Cref{def:subcell_SBP} that $P_u = P_{\bar{u}} + P_{O_u}$.
Hence, we can rewrite \cref{eq:overset_integral_discrete} more compactly as $I(\mathbf{u},\mathbf{v}) = \mathbf{1}^T P_{\bar{u}} \mathbf{u} + \mathbf{1}^T P_v \mathbf{v}$, where $\mathbf{1}^T P_{\bar{u}} \mathbf{u} \approx \int_a^b u \intd x$.
We then get
\begin{equation}\label{eq:conservation_discrete1}
	\frac{\d}{\d t} I(\mathbf{u},\mathbf{v})
		= - \alpha \mathbf{1}^T P_{\bar{u}} D_u \mathbf{u}
		+ \mathbf{1}^T P_{\bar{u}} P_u^{-1} \mathbb{S}_u
		- \alpha \mathbf{1}^T P_v D_v \mathbf{v}
		+ \mathbf{1}^T P_v P_v^{-1} \mathbb{S}_v,
\end{equation}
where we have used that $\mathbf{u}$ and $\mathbf{v}$ satisfy the semi-discretization \cref{eq:LR_disc}.
Observe that (iii) in \Cref{def:subcell_SBP} transforms $\mathbf{1}^T P_{\bar{u}} D_u \mathbf{u}$ into $\mathbf{1}^T B_{\bar{u}} \mathbf{u} - (D_u \mathbf{1})^T P_{\bar{u}} \mathbf{u}$.
Furthermore, if $1 \in \mathcal{F}$, then (i) in \Cref{def:subcell_SBP} implies $D_u \mathbf{1} = \mathbf{0}$ and we get $\mathbf{1}^T P_{\bar{u}} D_u \mathbf{u} = \mathbf{1}^T B_{\bar{u}} \mathbf{u}$.
We also get $\mathbf{1}^T P_v D_v \mathbf{v} = \mathbf{1}^T B_v \mathbf{v}$.
To summarize, the sub-cell SBP properties transform \cref{eq:conservation_discrete1} into
\begin{equation}\label{eq:conservation_discrete2}
	\frac{\d}{\d t} I(\mathbf{u},\mathbf{v})
		= - \alpha \mathbf{1}^T B_{\bar{u}} \mathbf{u}
		+ \mathbf{1}^T P_{\bar{u}} P_u^{-1} \mathbb{S}_u
		- \alpha \mathbf{1}^T B_v \mathbf{v}
		+ \mathbf{1}^T \mathbb{S}_v.
\end{equation}
% Boundary operator assumptions
Next, the notation in \Cref{sub:application_projections} implies $\mathbf{1}^T B_{\bar{u}} \mathbf{u} = u_{b_L} - u_a$ and $\mathbf{1}^T B_v \mathbf{v} = v_d - v_b$.
Furthermore, note that
\begin{equation}
	\mathbf{1}^T P_{\bar{u}} P_u^{-1} \mathbb{S}_u
		= \alpha \mathbf{1}^T P_{\bar{u}} P_u^{-1} \mathbf{e}_{a,u} [g_L(t) - u_a]
		+ \alpha \mathbf{1}^T P_{\bar{u}} P_u^{-1} \mathbf{e}_{b_R,u} [u_{b_L} - u_{b_R}].
\end{equation}
Recall from (iii) in \Cref{def:subcell_SBP} that the norm matrices are diagonal and satisfy $P_u = P_{\bar{u}} + P_{O_u}$.
Hence, $P_{\bar{u}} P_u^{-1} = P_u^{-1} P_u - P_u^{-1} P_{O_u} $, yielding
\begin{equation}
	\alpha \mathbf{1}^T P_{\bar{u}} P_u^{-1} \mathbf{e}_{a,u} [g_L(t) - u_a]
		= \alpha (\mathbf{e}_{a,u}^T \mathbf{1}) [g_L(t) - u_a]
		- \alpha \mathbf{1}^T P_u^{-1} P_{O_u} \mathbf{e}_{a,u} [g_L(t) - u_a].
\end{equation}
We also get $\mathbf{1}^T \mathbb{S}_v = \alpha (\mathbf{e}_{b,v}^T \mathbf{1}) [ u_{b_L} - v_b ]$.
Notably, if $1 \in \mathcal{F}$, then $\mathbf{e}_{a,u}^T \mathbf{1} = 1$ and $\mathbf{e}_{b,v}^T \mathbf{1} = 1$.
Hence, we can rewrite \cref{eq:conservation_discrete2} as
\begin{equation}\label{eq:conservation_discrete3}
	\frac{\d}{\d t} I(\mathbf{u},\mathbf{v})
		= - \alpha [ v_d - g_L(t) ] - \alpha \mathbf{1}^T P_u^{-1} P_{O_u} \mathbf{e}_{a,u} [g_L(t) - u_a]
		+ \alpha \mathbf{1}^T P_u^{-1} P_{\bar{u}} \mathbf{e}_{b_R,u} [u_{b_L} - u_{b_R}].
\end{equation}
It remains to show that the last two terms on the right-hand side of \cref{eq:conservation_discrete3} are zero.
To this end, recall from (ii) in \Cref{def:subcell_SBP} that $P_{O_u}$ is supported on the grid for $\Omega_O$, i.e., $(p_{O_u})_n = 0$ if $(x_u)_n \in \mathbf{x}_{\bar{u}}$.
We thus choose $(e_{a,u})_n = 0$ if $(x_u)_n \in \mathbf{x}_{O_u}$, i.e., the discrete projection operator $\mathbf{e}_{a,u}$ is supported only on the grid for $\Omega_{\bar{u}} = [a,b]$.
This is not a strong restriction, as it is common to include the boundary points in the grid.
In this case, $(x_u)_1 = a$ and thus $\mathbf{e}_{a,u} = [1,0,\dots,0]^T$, which satisfies the above restriction.
We consequently get $P_{O_u} \mathbf{e}_{a,u} = \mathbf{0}$.
For the same reason, we also have $P_{\bar{u}} \mathbf{e}_{b_R,u} = \mathbf{0}$ and therefore the assertion holds.
\end{proof}

\subsection{Proving discrete energy stability}
\label{sub:application_stability}

In addition to being conservative, exact solutions of the continuous overset IBVP \cref{eq:LR_IBVP_scalar} are energy-stable.
That is, \cref{eq:LR_IBVP_energy2} holds for the overset energy \cref{eq:overset_energy}.
To establish stability for the numerical solution, a similar bound to \cref{eq:LR_IBVP_energy2} is required.
To this end, a discrete version of the overset energy \cref{eq:overset_energy} for the sub-cell semi-discretization \cref{eq:LR_disc} is given by
\begin{equation}\label{eq:overset_energy_discrete}
	E(\mathbf{u},\mathbf{v})
		= \mathbf{u}^T P_u \mathbf{u}
		+ \mathbf{v}^T P_v \mathbf{v}
		- \mathbf{u}^T P_{O_u} \mathbf{u},
\end{equation}
where $\mathbf{u}^T P_u \mathbf{u} \approx \int_a^c u^2 \intd x$, $\mathbf{v}^T P_v \mathbf{v} \approx \int_b^d v^2 \intd x$, and $\mathbf{u}^T P_{O_u} \mathbf{u} \approx \int_b^c u^2 \intd x$.
Thus, to establish discrete energy stability, we need to show
\begin{equation}\label{eq:energy_discrete5}
	\frac{\d}{\d t} E(\mathbf{u},\mathbf{v})
		= - \alpha [ v_d^2 - g_L(t)^2 ]
		- \alpha [ g_L(t) - u_a ]^2
		- \alpha [ u_{b_L} - v_b ]^2
		\leq \alpha g_L(t)^2,
\end{equation}
which is the discrete counterpart to \cref{eq:LR_IBVP_energy2}.
Notably, \cref{eq:energy_discrete5} is similar to \cref{eq:LR_IBVP_energy2} but also contains the additional terms ``$- \alpha [ g_L(t) - u_a ]^2$" and ``$- \alpha [ u_{b_L} - v_b ]^2$" that introduce a small amount of artificial dissipation \cite{svard2014review,fernandez2014review}.

\begin{lemma}\label{lem:energy}
	Let $D_u$ be a sub-cell SBP operator for $(\Omega_u,\Omega_{\bar{u}},\Omega_O)$ and let $D_v$ be an SBP operator on $\Omega_v$.
	Furthermore, let $B_{\bar{u}}$ and $B_{O_u}$ be as in \cref{eq:B_form} with $(e_{a,u})_n = 0$ if $(x_u)_n \in \mathbf{x}_{O_u}$ and $(e_{b_R,u})_n = 0$ if $(x_{u})_n \in \mathbf{x}_{\bar{u}}$.
	Then, the sub-cell semi-discretization \cref{eq:LR_disc} is energy-stable, i.e.,  \cref{eq:energy_discrete5} holds.
\end{lemma}

\begin{proof}
Recall that $P_u = P_{\bar{u}} + P_{O_u}$, transforming \cref{eq:overset_energy_discrete} into $E(\mathbf{u},\mathbf{v}) = \mathbf{u}^T P_{\bar{u}} \mathbf{u} + \mathbf{v}^T P_v \mathbf{v}$.
Consider the temporal derivative of $E(\mathbf{u},\mathbf{v})$ and assume $\mathbf{u}$ and $\mathbf{v}$ satisfy the semi-discretization \cref{eq:LR_disc}.
We then get
\begin{equation}\label{eq:energy_discrete1}
	\frac{1}{2} \frac{\d}{\d t} E(\mathbf{u},\mathbf{v})
		= - \alpha \mathbf{u}^T P_{\bar{u}} D_u \mathbf{u}
		+ \mathbf{u}^T P_{\bar{u}} P_u^{-1} \mathbb{S}_u
		- \alpha \mathbf{v}^T P_v D_v \mathbf{v}
		+ \mathbf{v}^T P_v P_v^{-1} \mathbb{S}_v.
\end{equation}
The (sub-cell) SBP properties imply $\mathbf{u}^T P_{\bar{u}} D_u \mathbf{u} = \mathbf{u}^T B_{\bar{u}} \mathbf{u} - \mathbf{u}^T P_{\bar{u}} D_u \mathbf{u}$ and $\mathbf{v}^T P_v D_v \mathbf{v} = \mathbf{v}^T B_v \mathbf{v} - \mathbf{v}^T P_v D_v \mathbf{v}$.
Hence, \cref{eq:energy_discrete1} becomes
\begin{equation}\label{eq:energy_discrete2}
	\frac{\d}{\d t} E(\mathbf{u},\mathbf{v})
		= - \alpha \mathbf{u}^T B_{\bar{u}} \mathbf{u}
		+ 2 \mathbf{u}^T P_{\bar{u}} P_u^{-1} \mathbb{S}_u
		- \alpha \mathbf{v}^T B_v \mathbf{v}
		+ 2 \mathbf{v}^T \mathbb{S}_v.
\end{equation}
% Boundary operator assumptions
The notation in \Cref{sub:application_projections} implies $\mathbf{u}^T B_{\bar{u}} \mathbf{u} = u_{b_L}^2 - u_a^2$ and $\mathbf{v}^T B_v \mathbf{v} = v_d^2 - v_b^2$.
We also note that $P_u = P_{\bar{u}} + P_{O_u}$ allows us to rewrite $\mathbf{u}^T P_{\bar{u}} P_u^{-1} \mathbb{S}_u$ in \cref{eq:energy_discrete1} as
\begin{equation}
	\mathbf{u}^T P_{\bar{u}} P_u^{-1} \mathbb{S}_u
		= \alpha u_a [ g_L(t) - u_a ] - \alpha \mathbf{u}^T P_u^{-1} P_{O_u} \mathbf{e}_{a,u} [g_L(t) - u_a]
		+ \alpha \mathbf{u}^T P_u^{-1} P_{\bar{u}} \mathbf{e}_{b_R,u} [u_{b_L} - u_{b_R}];
\end{equation}
where the same arguments as for deriving \cref{eq:conservation_discrete3} in the proof of \Cref{lem:cons} were used.
We also have $\mathbf{v}^T P_v P_v^{-1} \mathbb{S}_v = \alpha v_b [u_b - v_b]$.
Hence, \cref{eq:energy_discrete2} is rewritten as
\begin{equation}\label{eq:energy_discrete3}
\begin{aligned}
	\frac{\d}{\d t} E(\mathbf{u},\mathbf{v})
		= & \alpha [ g_L(t)^2 - v_d^2 ]
		- \alpha [ g_L(t) - u_a ]^2
		- \alpha [ u_{b_L} - v_b ]^2 \\
		& - 2 \alpha \mathbf{u}^T P_u^{-1} P_{O_u} \mathbf{e}_{a,u} [g_L(t) - u_a]
		+ 2 \alpha \mathbf{u}^T P_u^{-1} P_{\bar{u}} \mathbf{e}_{b_R,u} [u_{b_L} - u_{b_R}].
\end{aligned}
\end{equation}
As before, $(e_{a,u})_n = 0$ if $(x_u)_n \in \mathbf{x}_{O_u}$ and $(e_{b_R,u})_n = 0$ if $(x_u)_n \in \mathbf{x}_{\bar{u}}$ imply $P_{O_u} \mathbf{e}_{a,u} = \mathbf{0}$ and $P_{\bar{u}} \mathbf{e}_{b_R,u} = \mathbf{0}$, respectively, transforming \cref{eq:energy_discrete3} into
\begin{equation}\label{eq:energy_discrete4}
	\frac{\d}{\d t} E(\mathbf{u},\mathbf{v})
		= \alpha [ g_L(t)^2 - v_d^2 ]
		- \alpha [ g_L(t) - u_a ]^2
		- \alpha [ u_{b_L} - v_b ]^2,
\end{equation}
which yields the assertion.
\end{proof}

\begin{remark}[Independence of the overlap width]\label{rem:overlap_width}
	The estimates \cref{eq:conservation_discrete4} and \cref{eq:energy_discrete5} are purely algebraic consequences of the sub-cell SBP properties and hold for every fixed grid.
	In particular, the bound in \cref{eq:energy_discrete5} is independent of the width $c-b$ of the overlap region.
	Consequently, the sub-cell semi-discretization \cref{eq:LR_disc} remains energy-stable, uniformly in the grid size.
	This is also true for grid sequences in which the overlap shrinks under refinement, e.g., when the overlap contains a fixed number of cells so that $c-b \to 0$ as $h \to 0$---a regime in which the discretization can no longer be interpreted as approximating a fixed continuous overset problem.
	The only requirement is the existence of a sub-cell SBP operator on the element containing $b$, which, by \Cref{thm:main} in \Cref{sec:existence}, reduces to the existence of SBP operators on the two sub-cells meeting at $b$.
	We note, however, that if $b$ lies close to a grid node, one sub-cell becomes small and the corresponding entries of $P^{-1}$ scale inversely with its size.
	While this does not affect the energy bound of the semi-discretization, it may impose a more restrictive time-step restriction for explicit time integration, similar to the small-cell problem for cut-cell methods \cite{engwer2020stabilized,fu2021high}.
\end{remark}

\subsection{Systems and non-linear considerations}
\label{sub:nonlinear}

With conservation and energy stability of the sub-cell semi-discretization for the linear advection equation in hand, it is necessary to discuss the coupling and how information is communicated at the overlap points, e.g., $b$ in the overset method.
The general formulations \cref{eq:LR_IBVP_general}, \cref{eq:LR_disc_general}, and \cref{eq:SATs} are not just valid for linear scalar problems, but also for non-linear systems of equations.
This includes two-way coupling between left- and right-traveling waves at the overlap points.
The coupling of solutions between the two grids is achieved using a numerical surface flux function $\vecfnum( \boldsymbol{w}_L, \boldsymbol{w}_R )$, which is a function of left and right solution states.

\begin{lemma}\label{cla:fullyUpwind}
To ensure conservation and energy stability, it is necessary for a general discretization that the coupling in \cref{eq:SATs} at a sub-cell point is performed in a fully upwind manner.
\end{lemma}

\begin{proof}
	Consider an overset discretization as in \Cref{fig:overlap_1d}, where there are two overlapping domains, $\Omega_u = [a,c]$ and $\Omega_v = [b,d]$ with $a<b<c<d$, each with its own mesh.
	For simplicity, we discuss only how conservation may be lost—the loss of provable energy stability is similar.
	Furthermore, we impose periodic boundary conditions so that the physical boundaries do not affect the conservation properties.
	We examine the overset integral of the mass, where we use a convex combination in the overlap integral
	\begin{equation}\label{eq:overset_integral2}
		\mathcal{I}(u,v)
			= \int_{a}^{c} u \intd x
			+ \int_{b}^{d} v \intd x
			- \int_{b}^{c} \beta u + (1-\beta) v \intd x.
	\end{equation}
	The discrete analogue of \eqref{eq:overset_integral2} is
	\begin{equation}\label{eq:overset_integral_discrete2}
		I(\mathbf{u},\mathbf{v})
			= \mathbf{1}^T P_u \mathbf{u}
			+ \mathbf{1}^T P_v \mathbf{v}
			- \beta\mathbf{1}^T P_{O_u} \mathbf{u}
			-(1-\beta)\mathbf{1}^T P_{O_v} \mathbf{v}.
	\end{equation}
	Taking the time derivative of \eqref{eq:overset_integral_discrete2}, we have
	\begin{equation}\label{eq:overset_integral_discrete_dt}
	\begin{aligned}
		\frac{\d}{\d t}I(\mathbf{u},\mathbf{v})
			& = \mathbf{1}^T P_u \mathbf{u}_t
				+ \mathbf{1}^T P_v \mathbf{v}_t
				- \beta\mathbf{1}^T P_{O_u} \mathbf{u}_t
				- (1-\beta)\mathbf{1}^T P_{O_v} \mathbf{v}_t \\[0.1cm]
			& = \mathbf{1}^T P_u \left( -D_u \mathbf{u} + P_u^{-1}\mathbb{S}_u \right)
				+ \mathbf{1}^T P_v \left( -D_v \mathbf{v} + P_v^{-1}\mathbb{S}_v\right ) \\[0.1cm]
			& \quad - \beta\mathbf{1}^T P_{O_u} \left( -D_u \mathbf{u} + P_u^{-1}\mathbb{S}_u \right)
				- (1-\beta)\mathbf{1}^T P_{O_v}\left( -D_v \mathbf{v} + P_v^{-1}\mathbb{S}_v \right),
	\end{aligned}
	\end{equation}
	with SATs as in \cref{eq:SATs}.
	Basic algebraic manipulations and using the sub-cell SBP property yield
	\begin{equation}\label{eq:overset_integral_discrete3}
	\begin{aligned}
		\frac{\d}{\d t}I(\mathbf{u},\mathbf{v})
			& = \beta\left[ \beta \vecfnum( \boldsymbol{u}_{b_L}, \boldsymbol{v}_b )
				+ (1-\beta) \vecfnum( \boldsymbol{u}_{b_R}, \boldsymbol{v}_b ) - \vecfnum( \boldsymbol{u}_{b_L}, \boldsymbol{u}_{b_R} ) \right]\\[0.1cm]
			&	\quad -(1-\beta)\left[ \beta \vecfnum( \boldsymbol{u}_{c}, \boldsymbol{v}_{c_L} ) + (1-\beta) \vecfnum( \boldsymbol{u}_{c}, \boldsymbol{v}_{c_R} ) - \vecfnum( \boldsymbol{v}_{c_L}, \boldsymbol{v}_{c_R} ) \right],
	\end{aligned}
	\end{equation}
with $\boldsymbol{v}_b = \mathbf{e}_{b,v}^T \mathbf{v} \approx \boldsymbol{v}(b)$, $\boldsymbol{u}_{b_L} = \mathbf{e}_{b_L,u}^T \mathbf{u} \approx \boldsymbol{u}(b)$, $\boldsymbol{u}_{b_R} = \boldsymbol{e}_{b_R,u}^T \mathbf{u} \approx \boldsymbol{u}(b)$, $\boldsymbol{u}_c = \mathbf{e}_{c,u}^T \mathbf{u} \approx \boldsymbol{u}(c)$, $\boldsymbol{v}_{c_L} = \mathbf{e}_{c_L,v}^T \mathbf{v} \approx \boldsymbol{v}(c)$, and $\boldsymbol{v}_{c_R} = \mathbf{e}_{c_R,v}^T \mathbf{v} \approx \boldsymbol{v}(c)$.
	To allow a general discussion of systems of equations, we make a slight abuse of notation: the projected quantities from the left and right remain bold.
	Recall that $\boldsymbol{v}_b$ and $\boldsymbol{u}_{b_R}$ are generally not the same.
	However, the same solution value projected from the left, $\boldsymbol{u}_{b_L}$, is used as the first argument in the remaining numerical flux terms above.
	To highlight the conservation properties, we consider two cases.
	For the fully right-traveling wave case, with $\beta=1$, we see that \eqref{eq:overset_integral_discrete3} becomes
	\begin{equation}\label{eq:overset_integral_beta1}
		\frac{\d}{\d t}I(\mathbf{u},\mathbf{v})
			=  \vecfnum( \boldsymbol{u}_{b_L}, \boldsymbol{v}_b ) - \vecfnum( \boldsymbol{u}_{b_L}, \boldsymbol{u}_{b_R} ).
	\end{equation}
	Thus, conservation only holds if $\vecfnum(  \boldsymbol{u}_{b_L}, \boldsymbol{v}_b  ) = \vecfnum( \boldsymbol{u}_{b_L}, \boldsymbol{u}_{b_R} ) = \boldsymbol{f}(\boldsymbol{u}_{b_L})$.
	For the fully left-traveling wave case, with $\beta=0$, \eqref{eq:overset_integral_discrete3} becomes
	\begin{equation}\label{eq:overset_integral_beta0}
		\frac{\d}{\d t}I(\mathbf{u},\mathbf{v})
			= -\left[ \vecfnum( \boldsymbol{u}_{c}, \boldsymbol{v}_{c_R} ) - \vecfnum( \boldsymbol{v}_{c_L}, \boldsymbol{v}_{c_R} ) \right].
	\end{equation}
	In this case, conservation only holds if $\vecfnum( \boldsymbol{u}_{c}, \boldsymbol{v}_{c_R} ) = \vecfnum(  \boldsymbol{v}_{c_L}, \boldsymbol{v}_{c_R}  ) = \boldsymbol{f}(\boldsymbol{v}_{c_R})$.
	Hence, the conditions in \eqref{eq:overset_integral_beta1} and \eqref{eq:overset_integral_beta0} are only satisfied for a fully upwind flux.
\end{proof}

Note that it follows from the above proof that a fully upwind flux is only necessary at a sub-cell point.
All other interfaces can use any numerical flux.
The coupling in \cref{eq:LR_disc} for the linear advection equation is done in a fully upwind way.
We will numerically demonstrate the importance of fully upwind sub-cell coupling, particularly for non-linear problems, in \Cref{sec:compEuler}.

\subsection{Extension to multi-block/-element overset domain methods}
\label{sub:application_multi}

We now outline how to incorporate the proposed sub-cell SBP operators into multi-block/-element overset domain methods, thereby enabling locally conservative and energy-stable discontinuous Galerkin spectral element methods (DGSEMs) for overset-domain problems.
The domains $\Omega_u$ and $\Omega_v$ are each partitioned into disjoint elements, with every element equipped with a discrete derivative operator and coupled via SATs (see \Cref{fig:schematic_multiBlock}).
The resulting multi-element SBP-SAT discretization retains the same conservation and energy stability properties as in the single-block setting.
Crucially, only the elements containing the points $b$ and $c$ require the novel sub-cell SBP operators; all other elements can use conventional SBP operators.
Furthermore, for the linear advection equation with positive velocity, only the element containing the point $b$ needs to be equipped with a sub-cell SBP operator.

\begin{figure}[tb]
\centering
\resizebox{.99\textwidth}{!}{%
\begin{tikzpicture}[scale=1.2]

    %% Lower domain: Omega_u
	\coordinate (a) at (0,0);
	\coordinate (c) at (8,0);
	\coordinate (u_overlap_1) at (2,0);
	\coordinate (u_overlap_2) at (4,0);
	\coordinate (u_overlap_3) at (6,0);

    % Draw lower domain line
	\draw[thick, black] (a) -- (c);

	% Vertical ticks (lower)
	\foreach \x in {0,2,4,6,8} {
    	\draw[thick] (\x,-0.2) -- (\x,0.2);
	}

    % Labels for lower domain
	\node[below left] at (a) {$a$};
	\node[below right, red1] at (c) {$c$};
	\node[below] at ($(a)!0.5!(u_overlap_1)$) {$\Omega_{u,1}$};
	\node[below] at ($(u_overlap_1)!0.5!(u_overlap_2)$) {$\Omega_{u,2}$};
	\node[below, blue1] at ($(u_overlap_2)!0.5!(u_overlap_3)$) {$\Omega_{u,3}$};
	\node[below] at ($(u_overlap_3)!0.5!(c)$) {$\Omega_{u,4}$};

    % Blue hatched overlap (lower)
	\fill[pattern={Lines[
       					distance=2mm,
                  		angle=45,
                  		line width=0.5mm
                 		]},
		pattern color=blue1] ($(u_overlap_2)-(0,-0.1)$) rectangle ($(u_overlap_3)+(0,-0.1)$);

    %% Upper domain: Omega_v
	\coordinate (b) at (5,1);
	\coordinate (d) at (11,1);
	\coordinate (v_overlap_start) at (7,1);
	\coordinate (v_overlap_end) at (9,1);

	% Draw upper domain line
	\draw[thick, black] (b) -- (d);

	% Vertical ticks (upper)
	\foreach \x in {5,7,9,11}{
    		\draw[thick] (\x,0.8) -- (\x,1.2);
	}

	% Labels for upper domain
	\node[above left, blue1] at (b) {$b$};
	\node[above right] at (d) {$d$};
	\node[above] at ($(b)!0.5!(v_overlap_start)$) {$\Omega_{v,1}$};
	\node[above, red1] at ($(v_overlap_start)!0.5!(v_overlap_end)$) {$\Omega_{v,2}$};
	\node[above] at ($(v_overlap_end)!0.5!(d)$) {$\Omega_{v,3}$};

	% Red hatched overlap (upper)
	\fill[pattern={Lines[
       					distance=2mm,
                  		angle=45,
                  		line width=0.5mm
                 		]},
		pattern color=red1] ($(v_overlap_start)-(0,-0.1)$) rectangle ($(v_overlap_end)+(0,-0.1)$);

	% Draw lines from b and c
	\draw[thick, blue1, dashed] (b) -- ($(u_overlap_2)!0.5!(u_overlap_3)$);
	\draw[thick, red1, dashed] (c) -- ($(v_overlap_start)!0.5!(v_overlap_end)$);

\end{tikzpicture}
}%
\caption{
  	Schematic illustration of a multi-element grid for the two sub-domains $\Omega_u$ and $\Omega_v$ of the overset domain.
	The sub-domains are partitioned into four and three non-overlapping elements, respectively.
	Each sub-domain contains exactly one element that includes a boundary point of the other sub-domain: $\Omega_{u,3}$ (blue diagonal lines) contains $b$, while $\Omega_{v,2}$ (red diagonal lines) contains $c$.
	We propose using sub-cell SBP operators only in these two elements.
	All other elements are equipped with a usual SBP operator.
}
\label{fig:schematic_multiBlock}
\end{figure}
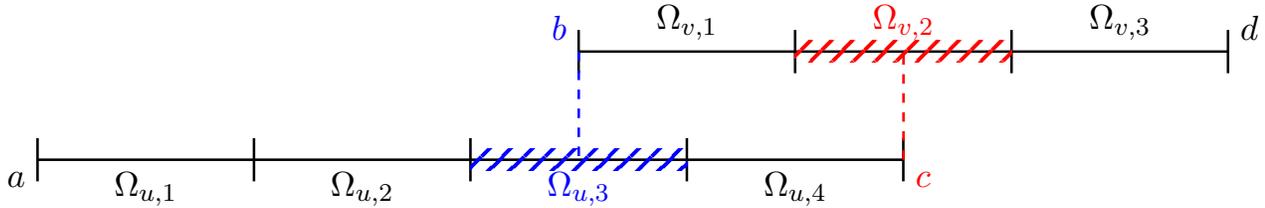

The domains $\Omega_u$ and $\Omega_v$ are now partitioned into $I_u$ and $I_v$ disjoint elements $\{ \Omega_{u_i} \}_{i=1}^{I_u}$ and $\{ \Omega_{v_j} \}_{j=1}^{I_v}$ with $\bigcup_{i=1}^{I_u} = \Omega_u$ and $\bigcup_{j=1}^{I_v} = \Omega_v$, respectively.
In this setting, we consider the following \emph{multi-element SBP-SAT semi-discretization} of the overset domain problem \cref{eq:LR_IBVP_scalar} for the scalar linear advection equation:
\begin{equation}\label{eq:multi_disc}
\begin{aligned}
	& \mathbf{u}_t^i + \alpha D_{u_i} \mathbf{u}^i
		= P_{u_i}^{-1} \mathbb{S}_{u_i}, \quad
		&& i=1,\dots,I_u, \\
	& \mathbf{v}_t^j + \alpha D_{v_j} \mathbf{v}^j
		= P_{v_j}^{-1} \mathbb{S}_{v_j}, \quad
		&& j=1,\dots,I_v.
\end{aligned}
\end{equation}
Here, $\mathbf{u}^i$ and $\mathbf{v}^j$ are the vectors of nodal values of the numerical solutions on $\Omega_{u_i}$ and $\Omega_{v_j}$, respectively.
Moreover, $D_{u_i} = P_{u_i}^{-1} Q_{u_i}$ and $D_{v_j} = P_{v_j}^{-1} Q_{v_j}$ are SBP operators on $\Omega_{u_i}$ and $\Omega_{v_j}$, respectively.
In the above multi-cell setting, only a single element of the multi-element SBP-SAT semi-discretization \cref{eq:multi_disc} must be equipped with sub-cell SBP operators---assuming $\alpha > 0$.
To this end, let $l \in \{1,\dots,I_u\}$ be the index such that $b \in \Omega_{u_l}$.
We choose $D_{u_l}$ to be a sub-cell SBP operator that mimics IBP on the whole cell $\Omega_{u_l} = [x_L^l, x_R^l]$ and on the sub-cells $\Omega_{\bar{u}_l} = [x_L^l, b]$ and $\Omega_{O_l} = [b, x_R^l]$.
We denote the associated discrete operators on $\Omega_{u_l}$ as $P_{u_l}$, $Q_{u_l}$, $B_{u_l}$, on $\Omega_{\bar{u}_l}$ as $P_{\bar{u}_l}$, $Q_{\bar{u}_l}$, $B_{\bar{u}_l}$, and $\Omega_{O_l}$ as $P_{O_l}$, $Q_{O_l}$, $B_{O_l}$.
We use traditional SBP operators on all other elements.

Furthermore, $\mathbb{S}_{u_i}$ and $\mathbb{S}_{v_j}$ in \cref{eq:multi_disc} are SATs on $\Omega_{u_i}$ and $\Omega_{v_j}$, respectively, that weakly enforce boundary and interface conditions.
We allow the numerical solutions to be discontinuous at element interfaces.
Hence, at each element interface, we have two values.
We distinguish between values from inside an element $\Omega_{u_i}$, say $u_L^i$ and $u_R^i$, and values from the neighboring elements, $\Omega_{u_{i-1}}$ and $\Omega_{u_{i+1}}$, say $u_R^{i-1}$ and $u_L^{i+1}$.
The subindices ``L'' and ``R'' indicate that the numerical solution's values at the left and right boundaries are used, respectively.
For conservation and stability (upwinding) reasons---and to allow information to be shared between neighboring elements---we introduce a weak coupling between elements using the SATs in \cref{eq:multi_disc}.
We again consider the usual fully upwind SATs given by
\begin{equation}\label{eq:multi_SATs}
\begin{aligned}
	\mathbb{S}_{u_i}
		& = \alpha \mathbf{e}_{L, u_i} \left[ u_R^{i-1} - u_L^i \right]
		+ \delta_{i,l} \alpha \mathbf{e}_{b_R, u_l} \left[ u_{b_L}^{l} - u_{b_R}^l \right], \quad
		&& i=1,\dots,I_u, \\
	\mathbb{S}_{v_j}
		& = \alpha \mathbf{e}_{L, v_j} \left[ v_R^{j-1} - v_L^j \right], \quad
		&& j=1,\dots,I_v,
\end{aligned}
\end{equation}
with $u_R^0 = g(t)$ and $v_R^0 = u_{b_L}^{l}$.
Here, $u_{L/R}^i = \mathbf{e}_{L/R, u_i}^T \mathbf{u}^i$ and $v_{L/R}^j = \mathbf{e}_{L/R, v_j}^T \mathbf{v}^j$ abbreviate the values of the numerical solutions at the left/right boundary of the elements $\Omega_{u_i}$ and $\Omega_{v_j}$, respectively.
Furthermore, $u_{b_L}^{l} = \mathbf{e}_{b_L, u_l}^T \mathbf{u}^l$  and $u_{b_R}^{l} = \mathbf{e}_{b_R, u_l}^T \mathbf{u}^l$ denote the value of a left- and right-sided projection of $\mathbf{u}_l$ to the point $b$ in $\Omega_{u_l}$.
Finally, $\delta_{i,l}$ in front of the second term for $\mathbb{S}_{u_i}$ in \cref{eq:multi_SATs} is the usual Kronecker delta with $\delta_{i,l} = 1$ if $i=l$ and $\delta_{i,l} = 0$ otherwise, ensuring that the second term only acts on the cell $\Omega_{u_l}$ that contains the left boundry points $b$ of the right sub-domain $\Omega_v$.

Using similar arguments to those in the single-element case, we can demonstrate that the multi-element SBP-SAT semi-discretization \cref{eq:multi_disc} is conservative and energy-stable.
We formalize these findings in \cref{lem:multi} below.

\begin{lemma}\label{lem:multi}
	Let $D_{u_i}$ and $D_{v_j}$ be SBP operators on the elements $\Omega_{u_i}$ and $\Omega_{v_j}$ for $i=1,\dots,I_u$ and $j=1,\dots,I_v$.
	Moreover, let $b \in \Omega_{u_l}$ and let $D_{u_l}$ be a sub-cell SBP operator for $( \Omega_{u_l}, \Omega_{O_l}, \Omega_{\bar{u}_l} )$ with grid points $\mathbf{x}_{u,l} = [ \mathbf{x}_{\bar{u},l}, \mathbf{x}_{O_u,l} ]$.
	Furthermore, let $(e_{L,u_l})_n = 0$ if $(x_{u,l})_n \in \mathbf{x}_{O_u,l}$ and $(e_{b_R,u_l})_n = 0$ if $(x_{u,l})_n \in \mathbf{x}_{\bar{u},l}$.
	Then, the multi-element SBP-SAT semi-discretization \cref{eq:multi_disc} is (i) conservative and (ii) energy-stable.
\end{lemma}

We show statements (i) and (ii) of \cref{lem:multi} separately.

\begin{proof}[Proof of (i) in \cref{lem:multi}]\label{proof:lem_multi_cons}
	Consider the overset integral $\mathcal{I}(u,v)$ in \cref{eq:overset_integral}.
	A discrete multi-element version of the overset integral is given by
	\begin{equation}\label{eq:multi_integral_discrete}
		I(\mathbf{u},\mathbf{v})
			= \underbrace{
				\sum_{i=1}^{I_u} \mathbf{1}^T P_{u_i} \mathbf{u}^i
			}_{\approx \int_a^c u \intd x}
			+ \underbrace{
				\sum_{j=1}^{I_v} \mathbf{1}^T P_{v_j} \mathbf{v}^j
			}_{\approx \int_b^d v \intd x}
			- \underbrace{ \left[
				\mathbf{1}^T P_{O_l} \mathbf{u}^l + \sum_{i=l+1}^{I_u} \mathbf{1}^T P_{u_i} \mathbf{u}^i
			\right] }_{\approx \int_b^c u \intd x}.
	\end{equation}
	Here, $\mathbf{u}$ and $\mathbf{v}$ on the left-hand side of \cref{eq:multi_integral_discrete} denote the collection of the numerical solution values $\mathbf{u}^i$ and $\mathbf{v}^j$, respectively.
	To prove (i) in \cref{lem:multi}, we need to demonstrate
	\begin{equation}\label{eq:multi_integral_discrete2}
		\frac{\d}{\d t} I(\mathbf{u},\mathbf{v})
			= - \alpha \left[ v_R^{I_v} - g(t) \right].
	\end{equation}
	To this end, note that re-writing \cref{eq:multi_integral_discrete} yields
	\begin{equation}\label{eq:multi_integral_discrete3}
		\frac{\d}{\d t} I(\mathbf{u},\mathbf{v})
			= \underbrace{
				\sum_{i=1}^{l-1} \mathbf{1}^T P_{u_i} \mathbf{u}^i_t
				+ \mathbf{1}^T P_{\bar{u}_l} \mathbf{u}^l_t
			}_{\approx \int_a^b u \intd x}
			+ \underbrace{
				\sum_{j=1}^{I_v} \mathbf{1}^T P_{v_j} \mathbf{v}^j_t
			}_{\approx \int_b^d v \intd x}.
	\end{equation}
	Substituting the multi-element SBP-SAT semi-discretization \cref{eq:multi_disc} into \cref{eq:multi_integral_discrete3} and using the same arguments as in \Cref{sub:application_conservation} yields
	\begin{equation}\label{eq:multi_integral_discrete4}
	\begin{aligned}
		\mathbf{1}^T P_{u_i} \mathbf{u}^i_t
			& = -\alpha \left[ u_R^i - u_R^{i-1} \right],
			&& i=1,\dots,l-1, \\
		\mathbf{1}^T P_{u_l} \mathbf{u}^l_t
			& = -\alpha \left[ u_{b_L}^l - u_R^{l-1} \right], && \\
		\mathbf{1}^T P_{v_j} \mathbf{v}^j_t
			& = -\alpha \left[ v_R^j - v_R^{j-1} \right],
			&& j=1,\dots,I_v,
	\end{aligned}
	\end{equation}
	where we have assumed that $1 \in \mathcal{F}$.
	Substituting \cref{eq:multi_integral_discrete4} into \cref{eq:multi_integral_discrete3} yields \cref{eq:multi_integral_discrete2} and thus the assertion.
\end{proof}

\begin{proof}[Proof of (ii) in \cref{lem:multi}]
	Consider the overset energy $\mathcal{E}(u,v)$ in \cref{eq:overset_energy}.
	A discrete multi-element version of the overset energy is given by
	\begin{equation}\label{eq:multi_energy_discrete}
		E(\mathbf{u},\mathbf{v})
			= \underbrace{
				\sum_{i=1}^{I_u} (\mathbf{u}^i)^T P_{u_i} \mathbf{u}^i
			}_{\approx \int_a^c u^2 \intd x}
			+ \underbrace{
				\sum_{j=1}^{I_v} (\mathbf{v}^j)^T P_{v_j} \mathbf{v}^j
			}_{\approx \int_b^d v^2 \intd x}
			- \underbrace{ \left[
				(\mathbf{u}^l)^T P_{O_l} \mathbf{u}^l + \sum_{i=l+1}^{I_u} (\mathbf{u}^i)^T P_{u_i} \mathbf{u}^i
			\right] }_{\approx \int_b^c u^2 \intd x}.
	\end{equation}
	To prove (ii) in \cref{lem:multi}, we need to demonstrate
	\begin{equation}\label{eq:multi_energy_discrete2}
		\frac{1}{2} \frac{\d}{\d t} E(\mathbf{u},\mathbf{v})
			\leq \alpha g_L(t)^2.
	\end{equation}
	To this end, note that re-writing \cref{eq:multi_energy_discrete} yields
	\begin{equation}\label{eq:multi_energy_discrete3}
		\frac{1}{2} \frac{\d}{\d t} E(\mathbf{u},\mathbf{v})
			= \underbrace{
				\sum_{i=1}^{l-1} (\mathbf{u}^i)^T P_{u_i} \mathbf{u}^i_t
				+ (\mathbf{u}^l)^T P_{\bar{u}_l} \mathbf{u}^l_t
			}_{\approx \int_a^b u^2 \intd x}
			+ \underbrace{
				\sum_{j=1}^{I_v} (\mathbf{v}^j)^T P_{v_j} \mathbf{v}^j_t
			}_{\approx \int_b^d v^2 \intd x}.
	\end{equation}
	Substituting the multi-element SBP-SAT semi-discretization \cref{eq:multi_disc} into \cref{eq:multi_integral_discrete3} and using the same arguments as in \Cref{sub:application_stability} yields
	\begin{equation}\label{eq:multi_energy_discrete4}
	\begin{aligned}
		(\mathbf{u}^i)^T P_{u_i} \mathbf{u}^i_t
			& \leq -\alpha \left[ (u_R^i)^2 - (u_R^{i-1})^2 \right],
			&& i=1,\dots,l-1, \\
		(\mathbf{u}^l)^T P_{\bar{u}_l} \mathbf{u}^l_t
			& \leq -\alpha \left[ (u_{b_L}^l)^2 - (u_R^{l-1})^2 \right], && \\
		(\mathbf{v}^j)^T P_{v_j} \mathbf{v}^j_t
			& \leq -\alpha \left[ (v_R^j)^2 - (v_R^{j-1})^2 \right],
			&& j=1,\dots,I_v.
	\end{aligned}
	\end{equation}
	% Combine the above
	Finally, substituting \cref{eq:multi_energy_discrete4} into \cref{eq:multi_energy_discrete3}, we get $\frac{1}{2} \frac{\d}{\d t} E(\mathbf{u},\mathbf{v}) \leq \alpha [g_L(t)^2 - (v_R^{I_v})^2 ]$, which implies \cref{eq:multi_energy_discrete2} and thus the assertion.
\end{proof}

\begin{remark}[Local conservation and energy-stability]
	Notably, \cref{eq:multi_integral_discrete4,eq:multi_energy_discrete4} demonstrate that the multi-element SBP-SAT semi-discretization \cref{eq:multi_disc} is not just globally conservative and energy-stable but also locally on each element.
\end{remark}
\section{Existence and construction}
\label{sec:existence}

We next establish the existence and construction of the proposed sub-cell SBP operator.
Our main result in \Cref{thm:main} states that a sub-cell SBP operator exists if and only if there exist SBP operators on the sub-cells.
\Cref{thm:main} subsequently tells us how to construct sub-cell SBP operators.
Finally, we present two illustrative examples of sub-cell SBP operators in \Cref{sub:operators_examples}.

% Notation for domain
We consider a function space $\mathcal{F} \subset C^1(\Omega_{\rm ref})$ with basis $\{f_k\}_{k=1}^K$ and use the same notation as in \Cref{sec:operators}.
Furthermore, we denote by $V = [ \mathbf{f}_1, \dots, \mathbf{f}_K ] \in \R^{N \times K}$ and $V' = [ \mathbf{f}'_1, \dots, \mathbf{f}'_K ] \in \R^{N \times K}$ the associated Vandermonde matrices for the nodal values of the basis elements and their derivatives at the grid points.

\subsection{Existence and construction of sub-cell SBP operators}
\label{sub:operators_existence}

Consider the reference element $\Omega_{\rm ref} = [\omega_L, \omega_R]$ with sub-cells $\Omega_L = [\omega_L, \omega_M]$ and $\Omega_R = [\omega_M, \omega_R]$.
Let $\mathbf{x} \in \R^N$, $\mathbf{x}_L \in \R^{N_L}$, and $\mathbf{x}_R \in \R^{N_R}$ be nodes on $\Omega_{\rm ref}$, $\Omega_L$ and $\Omega_R$, respectively, where $\mathbf{x} = [\mathbf{x}_L, \mathbf{x}_R]$.

\begin{theorem}\label{thm:main}
	Let the boundary operators $B_L$ and $B_R$ be as in \cref{eq:single_construction_B_form}.
	The discrete derivative operator $D = P^{-1}( S + B/2 ) \in \R^{(N_L + N_R) \times (N_L + N_R)}$ is an $\mathcal{F}$-exact sub-cell SBP operator for $(\Omega_{\rm ref}, \Omega_L, \Omega_R)$ if and only if
	(a) $D = \diag( D_L^{1,1}, D_R^{2,2} )$, $P = \diag( P_L^{1,1}, P_R^{2,2} )$, $S = \diag( S_L^{1,1}, S_R^{2,2} )$, $B = \diag( B_L^{1,1}, B_R^{2,2} )$, and
	(b) $D_L^{1,1} = ( P_L^{1,1} )^{-1}( S_L^{1,1} + B_L^{1,1} / 2 ) \in \R^{N_L \times N_L}$ and $D_R^{2,2} = ( P_R^{2,2} )^{-1}( S_R^{2,2} + B_R^{2,2} / 2 ) \in \R^{N_R \times N_R}$ are $\mathcal{F}$-exact SBP operators on $\Omega_L$ and $\Omega_R$, respectively.
\end{theorem}

\begin{proof}
	See \Cref{app:proofs}.
\end{proof}

\Cref{thm:main} establishes that a sub-cell SBP operator (defined in \Cref{def:subcell_SBP}) exists if and only if SBP operators exist on the individual sub-cells.
We next present a straightforward procedure for constructing $\mathcal{F}$-exact sub-cell SBP operators.
As established in \Cref{thm:main}, such operators can be built by first designing SBP operators on the sub-cells $\Omega_L$ and $\Omega_R$, and then assembling them into a global operator.

This observation motivates the following construction approach:
\begin{enumerate}
	\item
	Construct $\mathcal{F}$-exact SBP operators $\tilde{D}_{L/R} = \tilde{P}_{L/R}^{-1} ( \tilde{S}_{L/R} + \tilde{B}_{L/R}/2 ) \in \R^{N_{L/R} \times N_{L/R}}$ on $\Omega_{L/R}$ with sub-cell nodes $\mathbf{x}_{L/R} \in \R^{N_{L/R}}$.

	\item
	Get an $\mathcal{F}$-exact sub-cell SBP operator $D = P^{-1}( S + B/2 ) \in \R^{N}$ with $N = N_L + N_R$ by assembling $\tilde{D}_{L/R}$.
	Specifically, choose
	$D = \diag( \tilde{D}_L, \tilde{D}_R )$,
	$P_L = \diag( \tilde{P}_L, 0 )$,
	$P_R = \diag( 0, \tilde{P}_R )$,
	$P = P_L + P_R$;
	$S_L = \diag( \tilde{S}_L, 0 )$,
	$S_R = \diag( 0, \tilde{S}_R )$,
	$S = S_L + S_R$; and
	$B_L = \diag( \tilde{B}_L, 0 )$,
	$B_R = \diag( 0, \tilde{B}_R )$,
	$B = B_L + B_R$.

\end{enumerate}
The first step requires constructing $\mathcal{F}$-exact SBP operators on $\Omega_L$ and $\Omega_R$.
Numerous works have addressed the construction of such operators: see \cite{fernandez2014generalized,svard2014review,fernandez2014review} for polynomial-based SBP operators, and \cite{glaubitz2022summation,glaubitz2024optimization,glaubitz2026summation} for FSBP operators based on general function spaces.
We present some illustrative examples of the resulting sub-cell SBP operators below in \Cref{sub:operators_examples}.

\begin{remark}\label{rem:coupling}
	Coupling procedures similar to the one proposed above have previously been outlined for \emph{non}-overlap domain settings in \cite{chan2018discretely,chan2019efficient,del2019extension} and \cite{lundquist2022multi} using polynomial SBP operators combined with central numerical fluxes and inner-product-preserving interface treatment, respectively.
	In addition, \cite{ranocha2021broad} considered the same strategy as in \cite{chan2018discretely,chan2019efficient,del2019extension} for polynomial upwind SBP operators \cite{mattsson2017diagonal,ranocha2025robustness,glaubitz2025generalized} together with fully upwind fluxes---again for \emph{non}-overlap problems.
	Although upwind operators were not included in our present discussion, these earlier coupling approaches can be viewed as special cases of our general sub-cell SBP framework, which accommodates arbitrary FSBP operators and---at least formally---a broad class of numerical fluxes in the overset domain context.
\end{remark}

\subsection{On the boundary operators}
\label{sub:boundaryOp}

% General approach
We start by addressing the existence and construction of $B_L$ and $B_R$ satisfying (iv) in \Cref{def:subcell_SBP}.
We then get $B$ as $B = B_L + B_R$.
Recall from \Cref{sec:application} that it is desirable to express the boundary operators as
\begin{equation}\label{eq:single_construction_B_form}
	B_L = \mathbf{e}_{M_L} \mathbf{e}_{M_L}^T - \mathbf{e}_{L} \mathbf{e}_{L}^T, \quad
	B_R = \mathbf{e}_{R} \mathbf{e}_{R}^T - \mathbf{e}_{M_R} \mathbf{e}_{M_R}^T,
\end{equation}
where $\mathbf{e}_{L}, \mathbf{e}_{M_L}, \mathbf{e}_{M_R}, \mathbf{e}_{R} \in \R^N$ are discrete projection operators satisfying $\mathbf{e}_{L}^T \mathbf{f} \approx f(\omega_L)$, $\mathbf{e}_{M_L}^T \mathbf{f} \approx f(\omega_M)$, $\mathbf{e}_{M_R}^T \mathbf{f} \approx f(\omega_M)$, and $\mathbf{e}_{R}^T \mathbf{f} \approx f(\omega_R)$.
The operators $B_L$ and $B_R$ in \cref{eq:single_construction_B_form} satisfy the exactness conditions (iv) in \Cref{def:subcell_SBP} if and only if the projection operators $\mathbf{e}_{L}, \mathbf{e}_{M_L}, \mathbf{e}_{M_R}, \mathbf{e}_{R}$ satisfy
\begin{equation}\label{eq:projections_exactness}
	\mathbf{e}_{L}^T \mathbf{f} = f(\omega_L), \quad
	\mathbf{e}_{M_L}^T \mathbf{f} = f(\omega_M), \quad
	\mathbf{e}_{M_R}^T \mathbf{f} = f(\omega_M), \quad
	\mathbf{e}_{R}^T \mathbf{f} = f(\omega_R), \quad
	\forall f \in \mathcal{F}.
\end{equation}
That is, the projection operators must be exact for all vectors corresponding to the nodal values of functions within the function space for which the sub-cell SBP is exact.

% e_L and e_R
For simplicity, we assume that $x_1 = \omega_L$ and $x_N = \omega_R$.
Under this assumption, $\mathbf{e}_{L} = [1,0,\dots,0]^T$ and $\mathbf{e}_{R} = [0,\dots,0,1]^T$ satisfy \cref{eq:projections_exactness}.
This choice also satisfies the assumption that $\mathbf{e}_{L}$ is supported only on the nodes in $\Omega_L = [\omega_L,\omega_M]$, which allowed us to establish conservation and energy stability in \Cref{sec:application}.

% The e_M's
Notably, the one-sided projection operators $\mathbf{e}_{M_L}$ and $\mathbf{e}_{M_R}$ generally differ from each other.
Henceforth, we restrict $\mathbf{e}_{M_L}$ and $\mathbf{e}_{M_R}$ to be supported on nodes $\mathbf{x}_L \in \Omega_L^{N_L}$ and $\mathbf{x}_R \in \Omega_R^{N_R}$, respectively.
That is, $(\mathbf{e}_{M_L})_n = 0$ if $x_n \in \mathbf{x}_{R}$ and $(\mathbf{e}_{M_R})_n = 0$ if $x_n \in \mathbf{x}_{L}$.
Recall that this restriction allowed us to prove conservation and energy stability in \Cref{sec:application}.
It is furthermore leveraged in \Cref{cor:structure_S} to prove \Cref{thm:main} and therefore the existence of the desired sub-cell SBP operators.

We introduce some additional notation to ensure the existence of the projection operators $\mathbf{e}_{M_L}$ and $\mathbf{e}_{M_R}$.
Let $\mathbf{x}_L \in \R^{N_L}$ and $\mathbf{x}_R \in \R^{N_R}$ be the grid points in $\Omega_L$ and $\Omega_R$, respectively.
We define the \emph{sub-cell Vandermonde matrices} $V_L \in \R^{N_L \times K}$ and $V_R \in \R^{N_R \times K}$ as $V_L = [f_1(\mathbf{x}_L), \dots, f_K(\mathbf{x}_L)]$ and $V_R = [f_1(\mathbf{x}_R), \dots, f_K(\mathbf{x}_R)]$.
The projection operators $\mathbf{e}_{M_L}$ and $\mathbf{e}_{M_R}$ exist as long as $V_L$ and $V_R$ have linearly independent columns, respectively.
To see this, note that the exactness conditions for $\mathbf{e}_{M_L}$ and $\mathbf{e}_{M_R}$ in \cref{eq:projections_exactness} can be rewritten as the linear systems $\mathbf{e}_{M_L}^T V_L = [f_1(\omega_M), \dots, f_K(\omega_M)]$ and $\mathbf{e}_{M_R}^T V_R = [f_1(\omega_M), \dots, f_K(\omega_M)]$---if $\mathbf{e}_{M_L}$ and $\mathbf{e}_{M_R}$ are supported on $\mathbf{x}_L$ and $\mathbf{x}_R$, respectively.
The above linear systems have a solution for $\mathbf{e}_{M_L}$ and $\mathbf{e}_{M_R}$ if $V_L$ and $V_R$ have linearly independent columns.

\begin{example}
	Assume we have at least $K$ points in both $\Omega_L$ and $\Omega_R$, i.e., $N_L \geq K$ and $N_R \geq K$.
	Let us consider two cases for $N_L$:
	\begin{enumerate}
		\item
		If $N_L = K$, then $V_L$ is invertible and there exists a unique $\mathbf{e}_{M,L}$ given by $\mathbf{e}_{M,L} = [ \tilde{\mathbf{e}}_{M,L}; \mathbf{0} ] \in \R^{N}$ with $\tilde{\mathbf{e}}_{M,L}^T = [f_1(\omega_M), \dots, f_K(\omega_M)] V_L^{-1}  \in \R^{N_L}$ and $\mathbf{0} \in \R^{N_R}$.
		This corresponds to approximating $u(\omega_M)$ by the Langrange interpolation polynomial $L(\omega_M) = \sum_{k=1}^{K} u_k \ell_k(\omega_M)$ on $\Omega_L$, where $\ell_k$ is the $k$th Lagrange basis function satisfying the cardinal property $\ell_k((x_L)_n) = 1$ if $k=n$ and $\ell_k((x_L)_n) = 0$ otherwise, resulting in $\tilde{\mathbf{e}}_{M,L}^T  = [\ell_1(\omega_M), \dots, \ell_N(\omega_M)]$.
		Notably, if $(x_L)_{N_L} = \omega_M$, then $\tilde{\mathbf{e}}_{M,L}^T = [0,\dots,0,1]$.

		\item
		If $N_L > K$, then $\mathbf{e}_{M,L}^T V_L = [f_1(\omega_M), \dots, f_K(\omega_M)]$ is underdetermined with infinitely many solutions.
		In this case, it is common practice \cite{glaubitz2020stable,glaubitz2023multi} to aim for the least squares solution with minimal $\ell^2$-norm.
		This corresponds to approximating $u(\omega_M)$ by the least squares approximation $f(\omega_M) = \sum_{k=1}^K \langle \mathbf{f}_k, \mathbf{u} \rangle f_k(\omega_M)$, where $\langle f, g \rangle = \sum_{n=1}^{N_L} f((x_L)_n) g((x_L)_n)$ is a discrete inner product and $\{f_k\}_{k=1}^K$ is an orthonormal basis of $\mathcal{F}$, resulting in $\tilde{\mathbf{e}}_{M,L}^T = [f_1(\omega_M), \dots, f_K(\omega_M)] V_L^T$.
	\end{enumerate}
	The above strategies can also be used to determine $\mathbf{e}_{M,R}$.
\end{example}

\subsection{Examples of sub-cell SBP operators}
\label{sub:operators_examples}

We present two illustrative examples on how to assemble polynomial SBP operators on Gauss--Lobatto (\Cref{expl:operator1}) and Gauss--Radau (\Cref{expl:operator2}) points to obtain sub-cell SBP operators.
Henceforth, we denote by $\mathcal{P}_d$ the space of all polynomials up to degree $d$.

\begin{example}[Sub-cell DG-SBP operators]\label{expl:operator1}
	Consider $\Omega_{\rm ref} = [-1,1]$ and let $\omega_M = 0$, so that $\Omega_L = [-1,0]$ and $\Omega_R = [0,1]$.
	Our goal is to construct a $\mathcal{P}_1$-exact sub-cell SBP operator on $\Omega_{\rm ref}$, where $\mathcal{P}_1 = \operatorname{span}\{1,x\}$.
	To get the desired $\mathcal{P}_1$-exact sub-cell SBP operator on $\Omega_{\rm ref}$, we first construct $\mathcal{P}_1$-exact SBP operators on $\Omega_{L}$ and $\Omega_{R}$.
	We get these using the Gauss--Lobatto quadrature on $\Omega_{L}$ and $\Omega_{R}$, respectively.
	The corresponding points and weights are $(x_L)_1 = -1$, $(x_L)_2 = 0$, $(x_R)_1 = 0$, $(x_R)_2 = 1$ and $(p_L)_1 = 1/2$, $(p_L)_2 = 1/2$, $(p_R)_1 = 1/2$, $(p_R)_2 = 1/2$; see \cite{abramowitz1972handbook}.
	Notably, the point $\omega_M = 0$ is contained in both $\mathbf{x}_L$ and $\mathbf{x}_R$.
	Furthermore, recall that the Gauss--Lobatto quadrature with $n$ points is $\mathcal{P}_{2n-3}$-exact.
	Hence, the Gauss--Lobatto quadrature is $\mathcal{P}_1$-exact for $n=2$ and we get $\mathcal{P}_1$-exact SBP operators $\tilde{D}_{L/R}$ on $\Omega_{L}$ and $\Omega_{R}$---by forming the Lagrange basis polynomials on $\mathbf{x}_{L/R}$ and evaluating their derivatives at the same points.
	At the same time, we get the projection operators by evaluating the same Lagrange basis polynomials at the respective boundary points.
	Assembling $\tilde{D}_{L/R}$ according to the construction procedure in \Cref{sub:operators_existence}, we finally get the desired $\mathcal{P}_1$-exact sub-cell SBP operator $D = P^{-1}( S + B/2 )$ with $\mathbf{x} = [ -1, 0, 0, 1 ]$ and
	\begin{equation}\label{eq:expl1}
	\begin{aligned}
		& P = \diag( 1/2, 1/2, 1/2, 1/2 ), \quad
		&& B = \diag\left(
			\begin{bmatrix}
				-1 & 0 \\
				0 & 1
			\end{bmatrix},
			\begin{bmatrix}
				-1 & 0 \\
				0 & 1
			\end{bmatrix}
		\right), \\
		& S = \frac{1}{2} \diag\left(
			\begin{bmatrix}
				0 & 1 \\
				-1 & 0
			\end{bmatrix},
			\begin{bmatrix}
				0 & 1 \\
				-1 & 0
			\end{bmatrix}
		\right), \quad
		&& D = \diag\left(
			\begin{bmatrix}
				-1 & 1 \\
				-1 & 1
			\end{bmatrix},
			\begin{bmatrix}
				-1 & 1 \\
				-1 & 1
			\end{bmatrix}
		\right).
	\end{aligned}
	\end{equation}
	We can similarly construct $\mathcal{P}_d$-exact sub-cell SBP operators for other values of $d$ using higher-order Gauss--Lobatto quadratures.
	Some of these are available in our reproducibility repository.
\end{example}

\begin{example}[Mixed sub-cell DG-SBP operators]\label{expl:operator2}
	Consider the same set-up and goal as in \Cref{expl:operator1}.
	This time, we use the left- and right-sided Gauss--Radau quadrature on $\Omega_L$ and $\Omega_R$, respectively, to construct the desired $\mathcal{P}_1$-exact sub-cell SBP operator on $\Omega_{\rm ref}$.
	The corresponding points and weights are $(x_L)_1 = -1$, $(x_L)_2 = -1/3$, $(x_R)_1 = 1/3$, $(x_R)_2 = 1$ and $(p_L)_1 = 1/4$, $(p_L)_2 = 3/4$, $(p_R)_1 = 3/4$, $(p_R)_2 = 1/4$; see \cite{abramowitz1972handbook}.
	Hence, this time, $\omega_M = 0$ is neither contained in $\mathbf{x}_L$ nor in $\mathbf{x}_R$.
	Furthermore, the Gauss--Radau quadrature with $n$ points is $\mathcal{P}_{2n-2}$-exact, positioning it between the Gauss--Lobatto (both end points included) and Gauss--Legendre (both end points excluded) quadratures, which are $\mathcal{P}_{2n-3}$-exact and $\mathcal{P}_{2n-1}$-exact, respectively.
	Hence, the Gauss--Radau quadrature is $\mathcal{P}_2$-exact for $n=2$ and we get $\mathcal{P}_1$-exact SBP operators $\tilde{D}_{L/R}$ on $\Omega_{L}$ and $\Omega_{R}$---by again forming the Lagrange basis polynomials on $\mathbf{x}_{L/R}$ and evaluating their derivatives at the same points.
	This time, assembling $\tilde{D}_{L/R}$ according to the construction procedure in \Cref{sub:operators_existence}, we get the $\mathcal{P}_1$-exact sub-cell SBP operator $D = P^{-1}( S + B/2 )$ with $\mathbf{x} = [ -1, -1/3, 1/3, 1 ]$ and
	\begin{equation}\label{eq:expl2}
	\begin{aligned}
		& P = \diag( 1/4, 3/4, 3/4, 1/4 ), \quad
		&& B = \frac{3}{4}\diag\left(
			\begin{bmatrix}
				-1 & -1 \\
				-1 & 3
			\end{bmatrix},
			\begin{bmatrix}
				-3 & 1 \\
				1 & 1
			\end{bmatrix}
		\right), \\
		& S = \frac{3}{4}\diag\left(
			\begin{bmatrix}
				0 & 1 \\
				-1 & 0
			\end{bmatrix},
			\begin{bmatrix}
				0 & 1 \\
				-1 & 0
			\end{bmatrix}
		\right), \quad
		&& D = \frac{3}{2}\diag\left(
			\begin{bmatrix}
				-1 & 1 \\
				-1 & 1
			\end{bmatrix},
			\begin{bmatrix}
				-1 & 1 \\
				-1 & 1
			\end{bmatrix}
		\right).
	\end{aligned}
	\end{equation}
	We can again use the same strategy to also construct $\mathcal{P}_d$-exact sub-cell SBP operators for other values of $d$ using higher-order Gauss--Radau quadratures, some of which can be found in our reproducibility repository.
\end{example}

\begin{example}[Sub-cell FD-SBP operators]\label{expl:operator3}
	Consider $\Omega_{\rm ref} = [-1,1]$ and let $\omega_M = 0$, so that $\Omega_L = [-1,0]$ and $\Omega_R = [0,1]$.
	Our goal is to construct a second-order sub-cell FD-SBP operator on $\Omega_{\rm ref}$ using equidistant grid points.
	To get the desired second-order sub-cell FD-SBP operator on $\Omega_{\rm ref}$, we first construct second-order FD-SBP operators on $\Omega_{L}$ and $\Omega_{R}$.
	We get these using equidistant grids on $\Omega_{L}$ and $\Omega_{R}$ with $N_L$ and $N_R$ points, respectively.
	The corresponding points are $(x_L)_i = -1 + (i-1) h_L$, $i=1,\dots,N_L$, with $h_L = 1/(N_L-1)$ for $\Omega_L$ and $(x_R)_j = (j-1) h_R$, $j=1,\dots,N_R$, with $h_R = 1/(N_R-1)$ for $\Omega_R$.
	The corresponding weights correspond to trapezoidal rules and are $(p_L)_1 = (p_L)_{N_L} = h_L/2$ as well as $(p_L)_i = h_L$, $i=2,\dots,N_L-1$, for $\Omega_L$ and $(p_R)_1 = (p_R)_{N_R} = h_R/2$ as well as $(p_R)_j = h_R$, $j=2,\dots,N_R-1$, for $\Omega_R$.
	Notably, the point $\omega_M = 0$ is contained in both $\mathbf{x}_L$ and $\mathbf{x}_R$.
	Assembling the corresponding second-order FD-SBP operators $\tilde{D}_{L} \in \R^{N_L \times N_L}$ and $\tilde{D}_{R} \in \R^{N_R \times N_R}$ on $\Omega_L$ and $\Omega_R$ according to the construction procedure in \Cref{sub:operators_existence}, we get the desired second-order sub-cell FD-SBP operator $D = P^{-1}( S + B/2 )$ with $\mathbf{x} = [ \mathbf{x}_L, \mathbf{x}_R ]$, $P = \diag( h_L/2, h_L, \dots, h_L, h_L/2, h_R/2, h_R, \dots, h_R, h_R/2 )$, $B = \diag(-1, 0, \dots, 0, 1, -1, 0, \dots, 0, 1)$, and
	\begin{equation}\label{eq:expl3}
	\begin{aligned}
		& S = \frac{1}{2} \diag\left(
			\begin{bmatrix}
				0 & 1 & 0 & \cdots & 0 \\
				-1 & 0 & 1 & \ddots & \vdots \\
				0 & -1 & 0 & \ddots & 0 \\
				\vdots & \ddots & \ddots & \ddots & 1 \\
				0 & \cdots & 0 & -1 & 0
			\end{bmatrix},
			\begin{bmatrix}
				0 & 1 & 0 & \cdots & 0 \\
				-1 & 0 & 1 & \ddots & \vdots \\
				0 & -1 & 0 & \ddots & 0 \\
				\vdots & \ddots & \ddots & \ddots & 1 \\
				0 & \cdots & 0 & -1 & 0
			\end{bmatrix}
		\right), \\
		& D = \diag\left(
			\frac{1}{2 h_L}
			\begin{bmatrix}
				-2 & 2 & 0 & \cdots & 0 \\
				-1 & 0 & 1 & \ddots & \vdots \\
				0 & -1 & 0 & \ddots & 0 \\
				\vdots & \ddots & \ddots & \ddots & 1 \\
				0 & \cdots & 0 & -2 & 2
			\end{bmatrix},
			\frac{1}{2 h_R}
			\begin{bmatrix}
				-2 & 2 & 0 & \cdots & 0 \\
				-1 & 0 & 1 & \ddots & \vdots \\
				0 & -1 & 0 & \ddots & 0 \\
				\vdots & \ddots & \ddots & \ddots & 1 \\
				0 & \cdots & 0 & -2 & 2
			\end{bmatrix}
		\right).
	\end{aligned}
	\end{equation}
	We can similarly construct higher-order sub-cell FD-SBP operators.
	Notably, we can also have SBP operators with different orders on $\Omega_L$ and $\Omega_R$.
\end{example}

\begin{example}[Mixed sub-cell DG-FD-SBP operators]\label{expl:operator4}
	This example illustrates that we can also combine DG-SBP and FD-SBP operators on $\Omega_L$ and $\Omega_R$ to obtain sub-cell SBP operators on $\Omega_{\rm ref}$.
	For instance, consider the $\mathcal{P}_1$-exact DG-SBP operator from \cref{expl:operator1} on $\Omega_{L} = [-1,0]$ using Gauss--Lobatto points and the second-order FD-SBP operator from \cref{expl:operator3} on $\Omega_{R} = [0,1]$.
	We can then get a mixed DG-FD-SBP operator on $\Omega_{\rm ref}$ by assembling the DG-SBP operator on $\Omega_{L}$ and the FD-SBP operator on $\Omega_{R}$ according to the construction procedure in \Cref{sub:operators_existence}.
	In doing so, we get a mixed sub-cell DG-FD-SBP operator $D = P^{-1}( S + B/2 )$ with $\mathbf{x} = [ \mathbf{x}_L, \mathbf{x}_R ]$ with $(x_L)_1 = -1$, $(x_L)_2 = 0$, $(x_R)_j = h_R j$, $j=1,\dots,N_R$, where $h_R = 1/(N_R-1)$, $P = \diag( 1/2, 1/2, h_R/2, h_R, \dots, h_R, h_R/2 )$, $B = \diag(-1, 1, -1, 0, \dots, 0, 1)$, and
	\begin{equation}\label{eq:expl4}
	\begin{aligned}
		& S = \frac{1}{2} \diag\left(
			\begin{bmatrix}
				0 & 1 \\
				-1 & 0
			\end{bmatrix},
			\begin{bmatrix}
				0 & 1 & 0 & \cdots & 0 \\
				-1 & 0 & 1 & \ddots & \vdots \\
				0 & -1 & 0 & \ddots & 0 \\
				\vdots & \ddots & \ddots & \ddots & 1 \\
				0 & \cdots & 0 & -1 & 0
			\end{bmatrix}
		\right), \\
		& D = \diag\left(
			\begin{bmatrix}
				-1 & 1 \\
				-1 & 1
			\end{bmatrix},
			\frac{1}{2 h_R}
			\begin{bmatrix}
				-2 & 2 & 0 & \cdots & 0 \\
				-1 & 0 & 1 & \ddots & \vdots \\
				0 & -1 & 0 & \ddots & 0 \\
				\vdots & \ddots & \ddots & \ddots & 1 \\
				0 & \cdots & 0 & -2 & 2
			\end{bmatrix}
		\right).
	\end{aligned}
	\end{equation}
	We can similarly construct mixed sub-cell DG-FD-SBP operators with different degrees of exactness and orders.
\end{example}
\section{Numerical tests}
\label{sec:tests}

For the subsequent numerical tests, we focus on the multi-block setting with Gauss--Lobatto SBP operators, which are commonly used in discontinuous Galerkin spectral element methods (DGSEMs).
Furthermore, we couple these SBP operators as described in \Cref{expl:operator1} to construct the proposed sub-cell SBP operators.
In all tests, we choose $\Omega_u = [-1, 0.1]$ and $\Omega_v = [-0.1, 1]$ together with periodic boundary conditions.
To integrate the semi-discretization in time, we use the explicit fourth-order, nine-stage Runge--Kutta method from \cite{ranocha2022optimized} with absolute and relative tolerance $10^{-8}$ for the error-based adaptive time-stepping.
For the convergence tests, we use the stricter tolerances $10^{-14}$ to reduce the error in time, thus focusing on the spatial error.
The overset grid method has been implemented in the Julia package SimpleDiscontinuousGalerkin.jl\footnote{\url{https://github.com/JoshuaLampert/SimpleDiscontinuousGalerkin.jl}}, and the sub-cell operators are integrated into the framework of SummationByPartsOperators.jl \cite{ranocha2021summationbypartsoperators} via the package SummationByPartsOperatorsExtra.jl \cite{lampert2026sbpextra}.
The time integration is performed using OrdinaryDiffEq.jl \cite{rackauckas2017differentialequations}.
All our numerical experiments can be reproduced with the openly available reproducibility repository \cite{glaubitz2025towardsRepro}.

\subsection{Linear advection equation}

We start by considering the linear advection equation \cref{eq:IBVP_scalar} with advection velocity $\alpha = 2$ and the initial data $w_0(x) = \sin(k\pi x)$ for $k \in \mathbb{N}$.
Since the solution propagates solely to the right, it suffices to employ a sub-cell operator only in the element containing $b$, while using standard Gauss--Lobatto operators in all other elements of both meshes.
For $k = 1$, the $L^2$-errors and experimental orders of convergence (EOC) are reported in \Cref{table:eocs_advection} for polynomial degrees $d = 3$ and $d = 4$.
Here, $N$ denotes the number of elements in each of the two domains.

\begin{table}[htb]
	\caption{$L^2$-errors and EOCs for the linear advection equation and polynomial degrees $d=3$ and $d=4$}
	\begin{subtable}{.5\linewidth}
		\subcaption{$d = 3$}
		\centering
		\begin{tabular}{ccc}
			\toprule
			N & error & EOC \\
			\midrule
			10 & 1.95e-05 & --- \\
			20 & 1.22e-06 & 4.00 \\
			40 & 7.63e-08 & 4.00 \\
			80 & 4.77e-09 & 4.00 \\
			\bottomrule
		\end{tabular}
	\end{subtable}%
	\begin{subtable}{.5\linewidth}
		\subcaption{$d = 4$}
		\centering
		\begin{tabular}{ccc}
			\toprule
			N & error & EOC \\
			\midrule
			10 & 3.24e-07 & --- \\
			20 & 1.03e-08 & 4.97 \\
			40 & 3.27e-10 & 4.98 \\
			80 & 1.04e-11 & 4.98 \\
			\bottomrule
		\end{tabular}
	\end{subtable}%
	\label{table:eocs_advection}
\end{table}

Next, we compare the proposed sub-cell semi-discretization with a baseline overset grid method, following the approach in \cite{kopriva2025energy}.
In the baseline scheme, standard Gauss--Lobatto operators are applied in each element, and solution values are interpolated across grids.
For a fair comparison, we match the number of degrees of freedom: the sub-cell method uses $9$ elements in the left mesh and the baseline method $10$, while both employ $10$ elements in the right mesh.
In all cases, the SBP operators are constructed from polynomials of degree $d = 3$ on Gauss--Lobatto points.
To highlight the limitations of the baseline overset approach, we consider a high-frequency wave with $k = 4$ and evolve the solution to a final time of $t = 200$.
The numerical solutions of the sub-cell and baseline overset methods are compared in \Cref{fig:advection_solution}.
The sub-cell discretization is dissipative, but remains stable, whereas the discretization without sub-cell operators exhibits a growing amplitude.
The results in \Cref{fig:advection_errors} compare the $L^2$- and $L^\infty$-errors of both discretizations over time.
The proposed sub-cell SBP operators achieve significantly lower errors.

\begin{figure}[tb]
	\centering
	\begin{subfigure}[b]{0.495\textwidth}
		\includegraphics[width=\textwidth]{%
			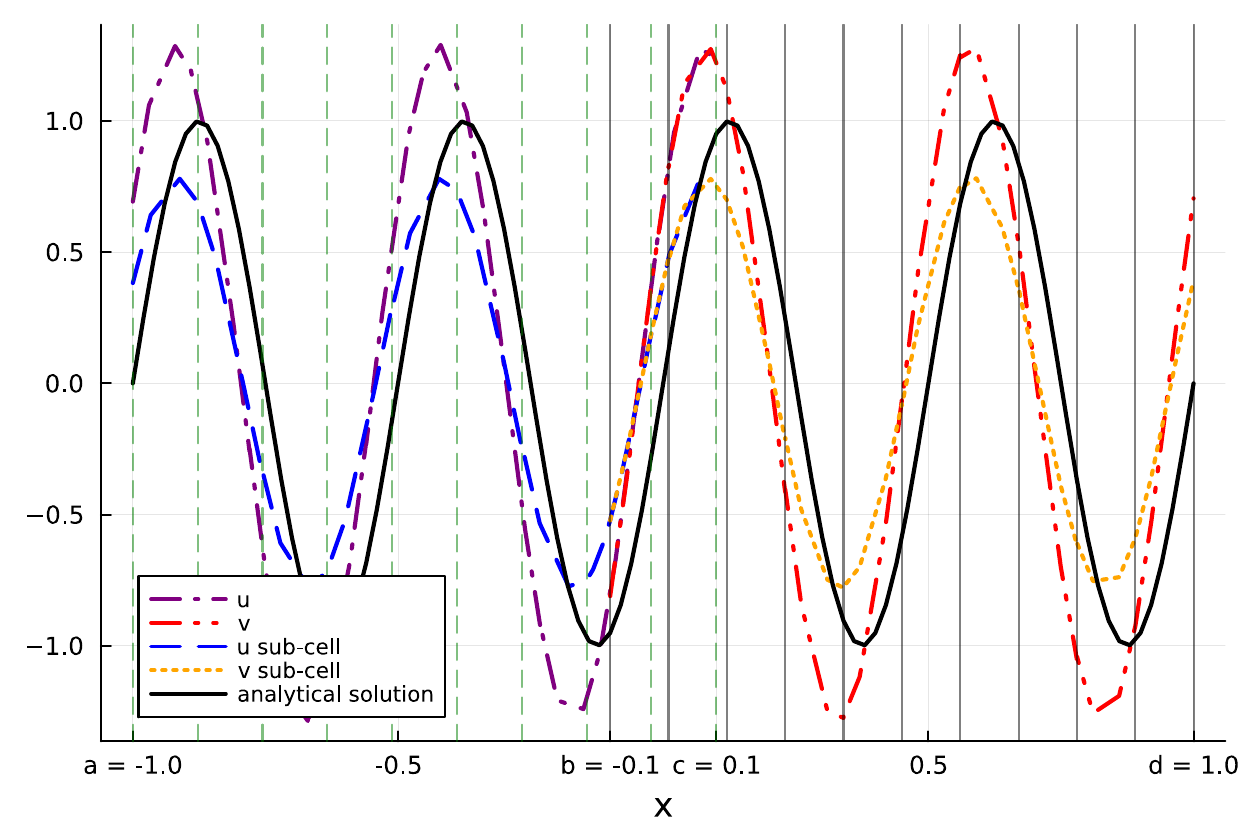}
		\caption{
			Numerical solution at time $t = 200$
		}
		\label{fig:advection_solution}
    \end{subfigure}%
	~
	\begin{subfigure}[b]{0.495\textwidth}
		\includegraphics[width=\textwidth]{%
			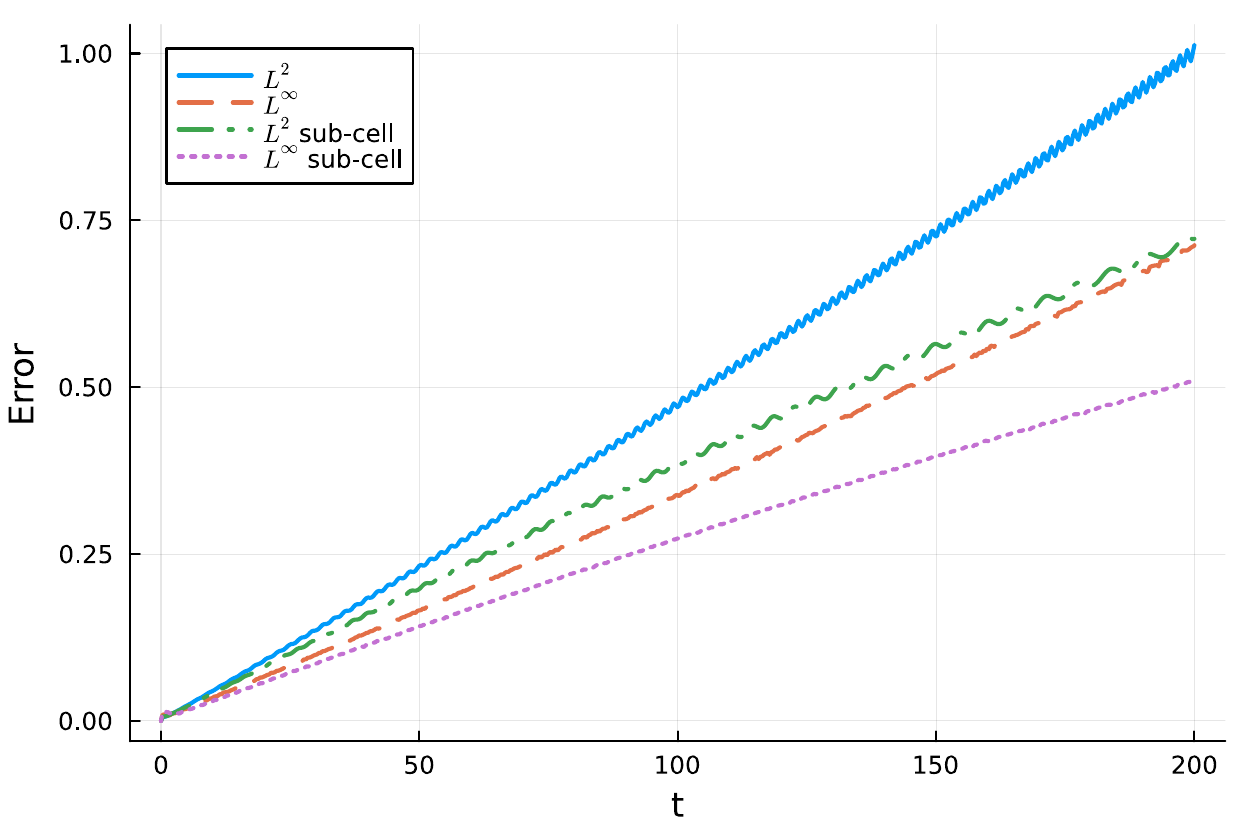}
		\caption{
			$L^2$- and $L^\infty$-errors
		}
		\label{fig:advection_errors}
    \end{subfigure}
    \caption{
        Numerical solutions of the linear advection equation and their errors for an overset grid semi-discretization with and without a sub-cell operator.
		Both discretizations use a polynomial degree $d = 3$.
		The discretization without sub-cell operators uses $10$ elements in both the left and right mesh, and the discretization with one sub-cell operator in the element containing $b$ uses $9$ elements in the left mesh and $10$ elements in the right mesh.
		The (green) dashed and (gray) straight vertical lines in \Cref{fig:advection_solution} indicate the element boundaries for the mesh of $\Omega_u$ and $\Omega_v$, respectively.
	}
    \label{fig:advection_solution_errors}
\end{figure}

To further compare the two methods, we analyze the spectra of the Jacobians of the semi-discrete right-hand-side functions for both the sub-cell and the baseline methods.
The spectra are shown in \Cref{fig:advection_spectra}.
For the baseline overset grid method, several eigenvalues exhibit small positive real parts, permitting growth of the numerical solution in time.
In contrast, the spectrum of the sub-cell method lies entirely in the left half-plane (to machine precision), consistent with the discrete energy estimate.
The only trade-off is a slightly more restrictive CFL condition due to the smaller sub-cell size.

\begin{figure}[tb]
	\centering
	\includegraphics[width=0.8\textwidth]{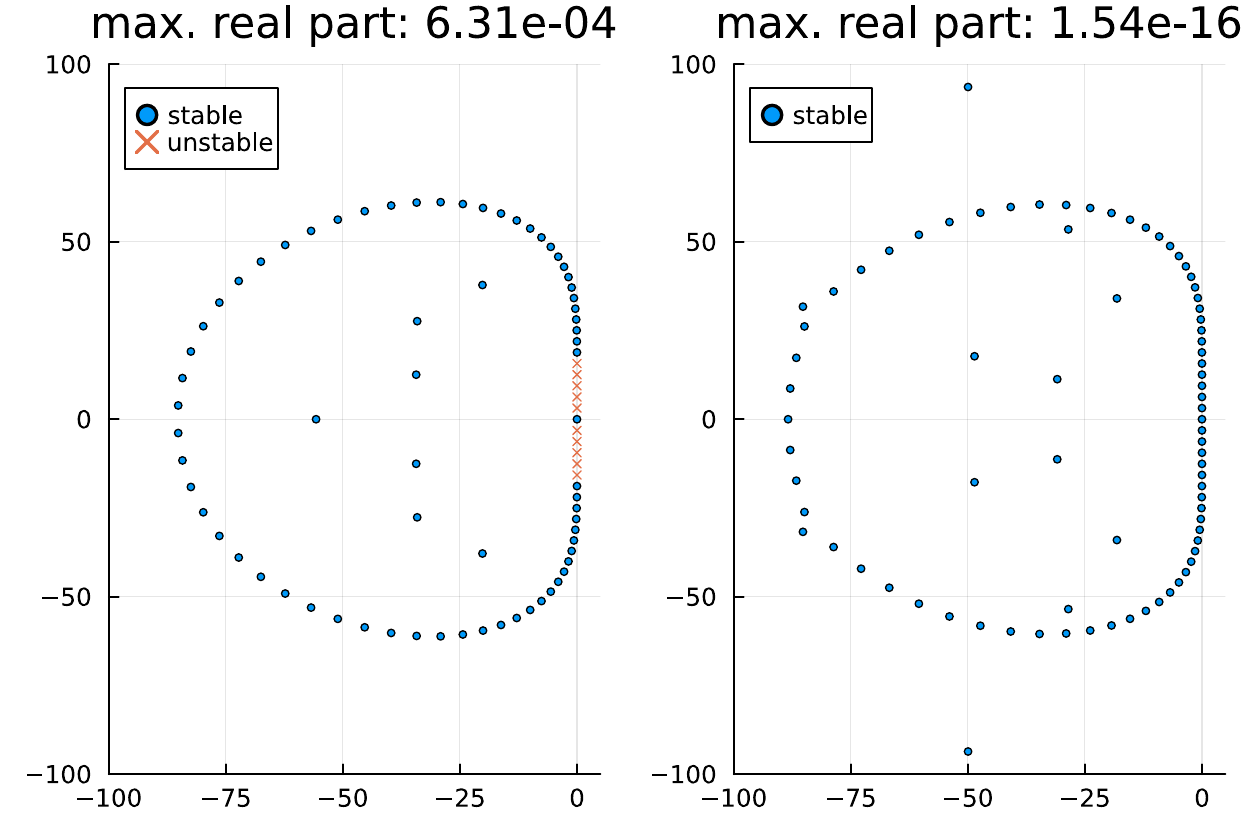}
	\caption{
		Spectra of the Jacobians of the semi-discretization's right-hand sides.
		The discretization in the left subplot uses no sub-cell SBP operators, while the right subplot does use sub-cell SBP operators.
	}
	\label{fig:advection_spectra}
\end{figure}

To examine the long-time behavior of the sub-cell overset discretization, we also consider the same setup as above, but evolve the solution to a final time of $t = 3000$
and use $20$ elements in each subdomain. 
In \Cref{fig:advection_spectra_N20}, we again find positive real parts in the spectrum of the baseline overset grid method, while the spectrum of the sub-cell method lies entirely in the left half-plane.
The errors of the numerical solutions are shown in \Cref{fig:advection_errors_long_time}, where the baseline overset grid method exhibits growing errors, while the errors of the sub-cell method remain significantly lower.

\begin{figure}[tb]
	\centering
	\begin{subfigure}[b]{0.495\textwidth}
		\includegraphics[width=\textwidth]{%
			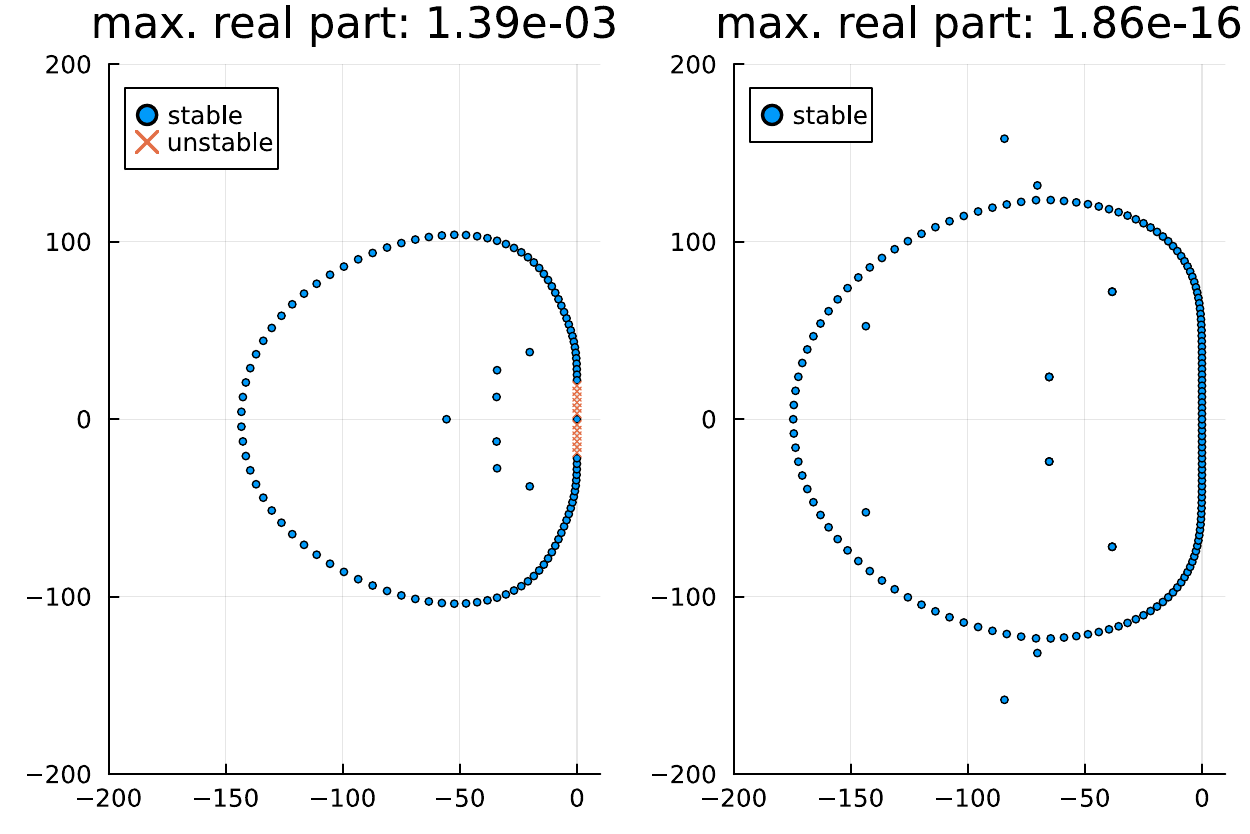}
		\caption{
			Spectra of the Jacobians (left: without sub-cell operators, right: with sub-cell operators)
		}
		\label{fig:advection_spectra_N20}
    \end{subfigure}%
	~
	\begin{subfigure}[b]{0.495\textwidth}
		\includegraphics[width=\textwidth]{%
			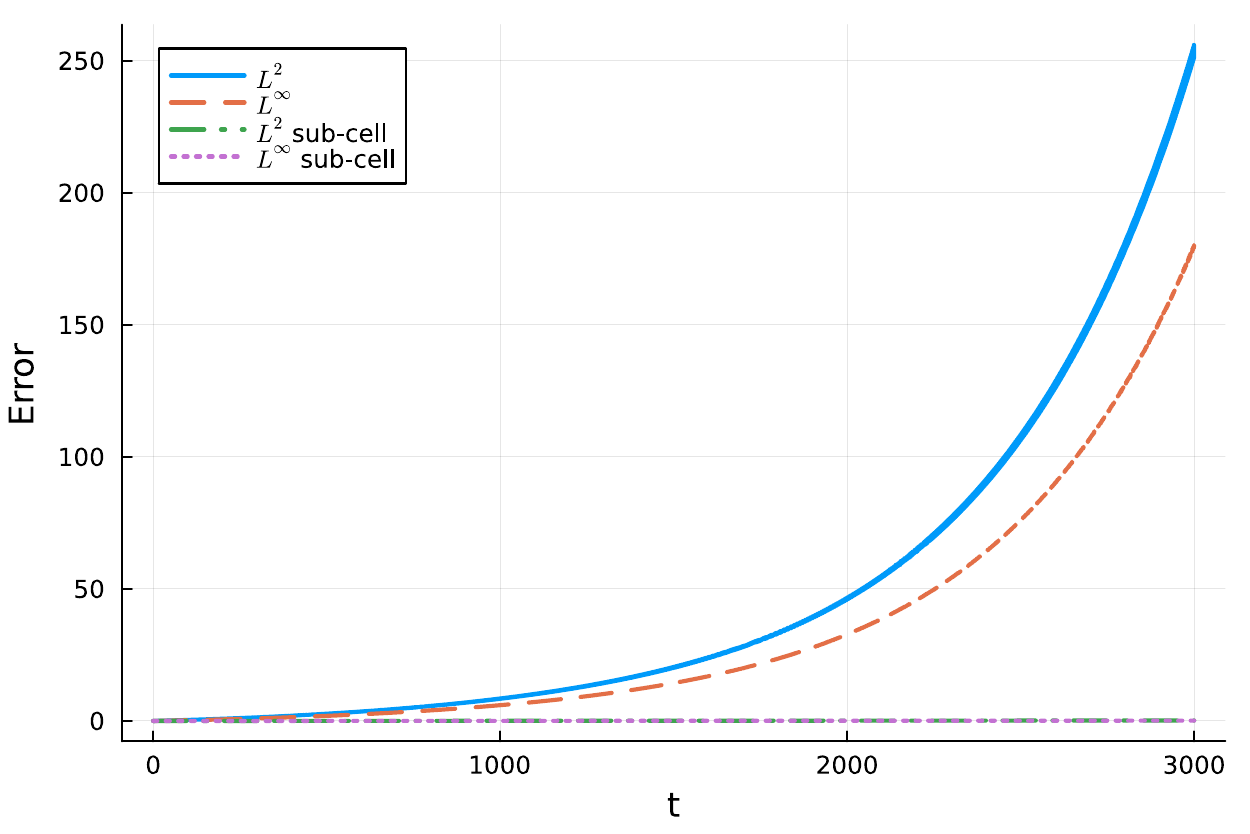}
		\caption{
			$L^2$- and $L^\infty$-errors
		}
		\label{fig:advection_errors_long_time}
    \end{subfigure}
    \caption{
        Spectra of the Jacobians of the semi-discretization's right-hand sides and errors of the numerical solutions of the linear advection equation and their errors for an overset grid semi-discretization with and without a sub-cell operator.
		Both discretizations use a polynomial degree $d = 3$.
		The discretization without sub-cell operators uses $20$ elements in both the left and right mesh, and the discretization with one sub-cell operator in the element containing $b$ uses $19$ elements in the left mesh and $20$ elements in the right mesh.
	}
    \label{fig:advection_long_time}
\end{figure}

The maximum real parts of the spectra for both the sub-cell and baseline overset grid methods are listed in \Cref{table:max_real_part_spectra} for various numbers of elements.
It shows that the maximum real part is positive and increases under grid refinement for the baseline method, while the sub-cell method yields maximum real parts close to zero (within machine precision).

\begin{table}[htb]
	\caption{
		Maximum real part of the eigenvalues of the Jacobians of the semi-discretization's right-hand sides for the linear advection equation with and without sub-cell operators for various values of $N$.
	}
	\centering
	\begin{tabular}{ccc}
		\toprule
		N & without sub-cell operator & with sub-cell operator \\
		\midrule
		5 & 3.78e-06 & -3.44e-16 \\
        10 & 6.31e-04 & 1.54e-16 \\
        20 & 1.39e-03 & 1.86e-16 \\
        40 & 1.40e-03 & -9.71e-17 \\
        80 & 1.40e-03 & -9.03e-18 \\
		\bottomrule
	\end{tabular}
	\label{table:max_real_part_spectra}
\end{table}

Finally, we present results demonstrating the conservation and energy stability of the sub-cell discretization for the linear advection equation---as theoretically investigated in \Cref{sub:application_conservation} and \Cref{sub:application_stability}.
The upper subplots of \Cref{fig:conservation} and \Cref{fig:entropy_stability} illustrate conservation and the semi-discrete change in energy.
As expected, conservation is maintained to machine precision, and the energy time derivative remains nonpositive.
For the baseline scheme, these quantities are omitted since it is not clear how to evaluate the overset integral \cref{eq:overset_integral} for them.

\begin{figure}[tb]
	\centering
	\includegraphics[width=0.9\textwidth]{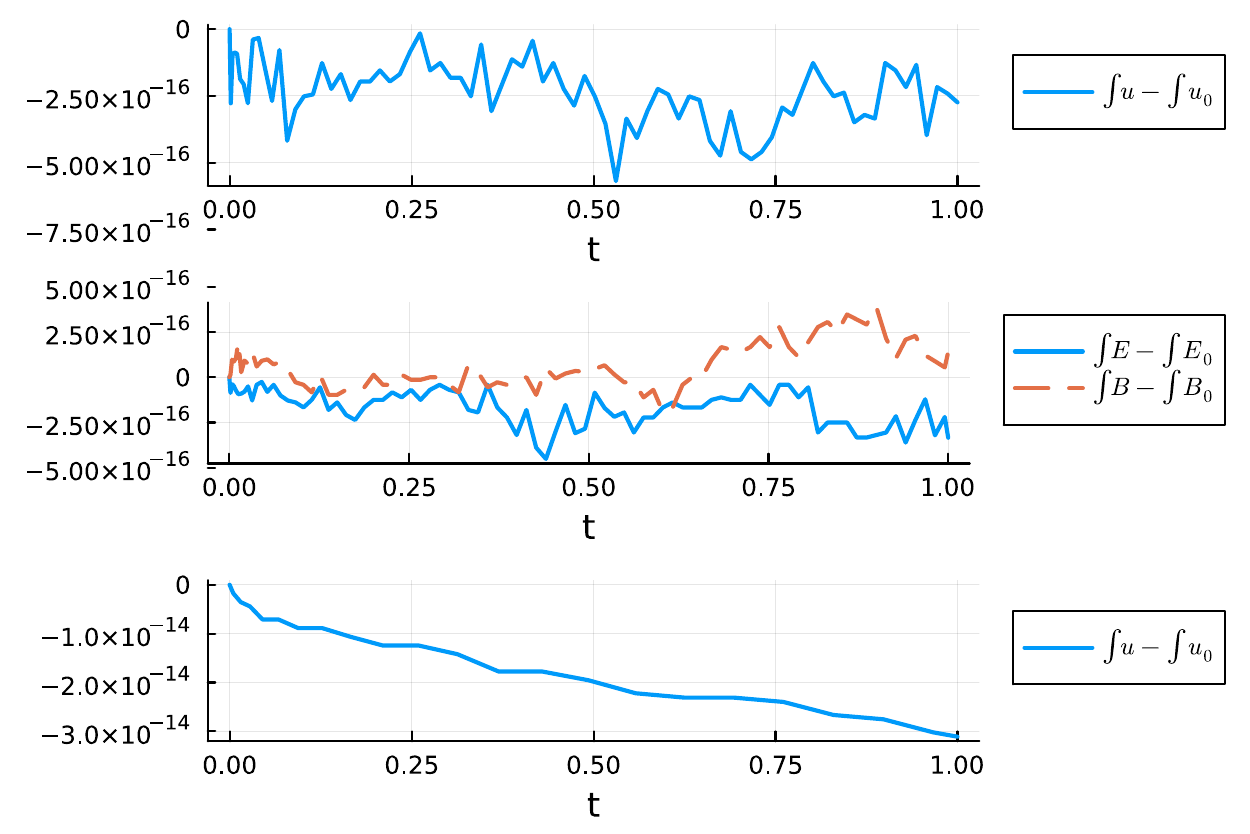}
	\caption{
		Change in the conserved linear quantities for the linear advection equation (top), Maxwell's system (center), and the non-linear Burgers' equation (bottom).
	}
	\label{fig:conservation}
\end{figure}

\begin{figure}[tb]
	\centering
	\includegraphics[width=0.9\textwidth]{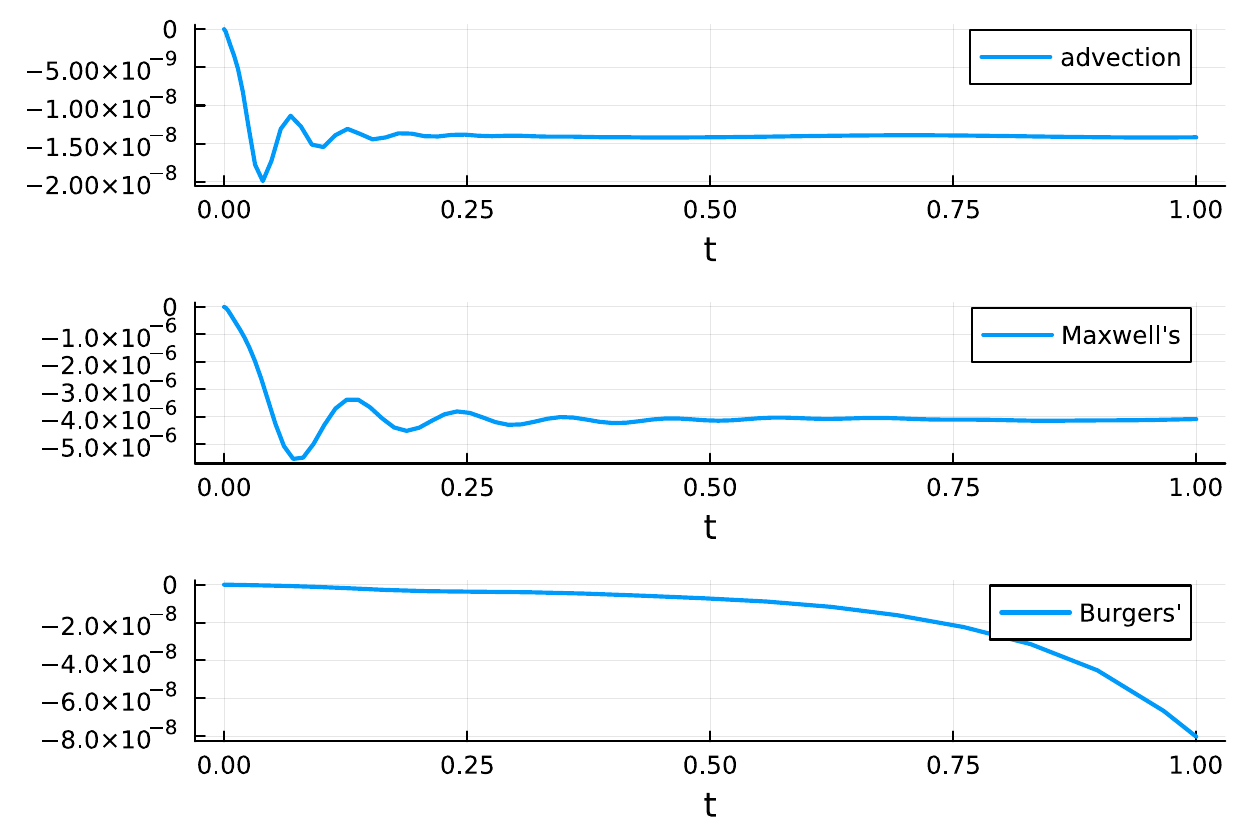}
	\caption{
		Semi-discrete time derivative of the energy for the linear advection equation (top), Maxwell's system (center), and the non-linear Burgers' equation (bottom).
	}
	\label{fig:entropy_stability}
\end{figure}

Notably, \Cref{fig:conservation} and \Cref{fig:entropy_stability} also confirm that the sub-cell overset discretization preserves conservation and entropy stability when applied to both the non-diagonalized linear system of Maxwell's equations and the non-linear inviscid Burgers' equation,
\begin{equation}
	\frac{\partial}{\partial t}
	\begin{pmatrix}
		E \\ B
	\end{pmatrix}
	+
	\frac{\partial}{\partial x}
	\begin{pmatrix}
		c^2 B \\ E
	\end{pmatrix}
	=
	\begin{pmatrix}
		0 \\ 0
	\end{pmatrix},
	\quad
	\frac{\partial}{\partial t} w + \frac{\partial}{\partial x} \left(\frac{w^2}{2}\right) = 0,
\end{equation}
with suitable initial and boundary conditions and $c \in \mathbb{R}^+$.
For the Burgers' equation, the dissipation increases over time as the solution develops a shock discontinuity.
Details on these problem setups are provided in the reproducibility repository \cite{glaubitz2025towardsRepro}.
The linear Maxwell's equations were discretized using the Rusanov (local Lax--Friedrichs) flux, while the Burgers' equation employed the Godunov (upwind) flux \cite{toro2013riemann}.

\subsection{Compressible Euler equations}\label{sec:compEuler}

To further evaluate the performance of the sub-cell discretization, we next consider the compressible Euler equations in one spatial dimension:
\begin{equation}\label{eq:Euler}
	\frac{\partial}{\partial t}
	\begin{pmatrix}
		\rho \\ \rho v \\ \rho e
	\end{pmatrix}
	+
	\frac{\partial}{\partial x}
	\begin{pmatrix}
		\rho v \\ \rho v^2 + p \\ (\rho e + p) v
	\end{pmatrix}
	=
	\begin{pmatrix}
		s_1 \\ s_2 \\ s_3
	\end{pmatrix},
\end{equation}
where $p = (\gamma - 1) ( \rho e - \rho v^2 / 2 )$ denotes the pressure, $\rho$ the density, $v$ the velocity, $e$ the specific total energy, and $\gamma$ the ratio of specific heats.
The Euler system admits a mathematical entropy function $S = -\rho s / (\gamma - 1)$, with specific entropy $s = \log(p / \rho^\gamma)$.
For the spatial semi-discretization, we employ flux differencing \cite{fisher2013high-order,gassner2016split} using the volume flux from \cite{ranocha2018thesis} together with the Harten--Lax--van Leer
(HLL) surface flux \cite{harten1983upstream,toro2013riemann}.

\Cref{table:eocs_compressible_euler} reports on the convergence of the overset grid method with sub-cell operators, measured in terms of the $L^2$-error, for polynomial degrees $d = 3$ and $d = 4$.
The table also lists the experimental orders of convergence (EOCs) for the density $\rho$; results for the other variables are similar and therefore omitted.
For this test, we add the source terms
\begin{equation}
	s_1(x, t) = 0, \quad
	s_2(x, t) = \omega A\cos(\omega (x - t))(2\rho(x, t) - 1/2)(\gamma - 1), \quad
	s_3(x, t) = s_2(x, t)
\end{equation}
to the right-hand side of the compressible Euler equations \cref{eq:Euler} to get the manufactured solution
\begin{equation}\label{eq:initial_condition_Euler}
\begin{aligned}
	\rho(x, t) = c + A\sin(\omega (x - t)), \quad
	(\rho v)(x, t) = \rho(x, t), \quad
	(\rho e)(x, t) = \rho(x, t)^2.
\end{aligned}
\end{equation}
We use periodic boundary conditions, choose $c = 2$, $A = 0.1$, $\omega = \pi$, and consider the time span from $0$ to $2$.
The results in \Cref{table:eocs_compressible_euler} confirm a convergence rate of approximately $d+1$.

\begin{table}[htb]
	\caption{$L^2$-errors and EOCs for the compressible Euler equations and polynomial degrees $d=3$ and $d=4$}
	\begin{subtable}{.5\linewidth}
		\subcaption{$d = 3$}
		\centering
		\begin{tabular}{ccccc}
			\toprule
			N & $\rho$ & $\rho v$ & $\rho e$ & EOC \\
			\midrule
			10 & 5.18e-06 & 2.06e-06 & 1.22e-05 & --- \\
			20 & 3.75e-07 & 1.25e-07 & 7.49e-07 & 3.79 \\
			40 & 2.03e-08 & 7.66e-09 & 4.13e-08 & 4.21 \\
			80 & 1.17e-09 & 4.78e-10 & 2.43e-09 & 4.12 \\
			\bottomrule
		\end{tabular}
	\end{subtable}%
	\begin{subtable}{.5\linewidth}
		\subcaption{$d = 4$}
		\centering
		\begin{tabular}{ccccc}
			\toprule
			N & $\rho$ & $\rho v$ & $\rho e$ & EOC \\
			\midrule
			10 & 2.00e-07 & 3.78e-08 & 3.66e-07 & --- \\
			20 & 7.07e-09 & 1.11e-09 & 1.19e-08 & 4.82 \\
			40 & 1.45e-10 & 3.27e-11 & 2.50e-10 & 5.61 \\
			80 & 4.50e-12 & 1.77e-12 & 7.87e-12 & 5.01 \\
			\bottomrule
		\end{tabular}
	\end{subtable}%
	\label{table:eocs_compressible_euler}
\end{table}

As discussed in \Cref{cla:fullyUpwind}, a fully upwind numerical surface flux is required to maintain convergence and entropy stability in non-linear problems.
To illustrate this, we compare results obtained with the Rusanov (local Lax--Friedrichs) flux, which is not fully upwind, with those from the HLL flux, which is fully upwind for the considered test case.
For this comparison, we again use \cref{eq:initial_condition_Euler} for $t = 0$ as the initial condition, but without source terms, i.e. $\boldsymbol{s} = \boldsymbol{0}$.
We observe in \Cref{fig:entropy_stability_euler} that the HLL flux conserves linear quantities up to a precision of around $10^{-13}$ (top left plot), while the Rusanov flux only preserves these up to a precision of around $10^{-7}$ (bottom left plot).
We also find that the HLL flux yields a small amount of dissipation (top right plot), indicated by small negative values for the semi-discrete time derivative of the mathematical entropy.
In contrast, the Rusanov flux can result in anti-dissipation (bottom right plot).
The results in \Cref{fig:entropy_stability_euler} therefore confirm the observations in \Cref{cla:fullyUpwind}: the Rusanov flux fails to maintain conservation and entropy stability for the sub-cell overset problem, whereas the HLL flux succeeds.

\begin{figure}[tb]
	\centering
	\includegraphics[width=0.99\textwidth]{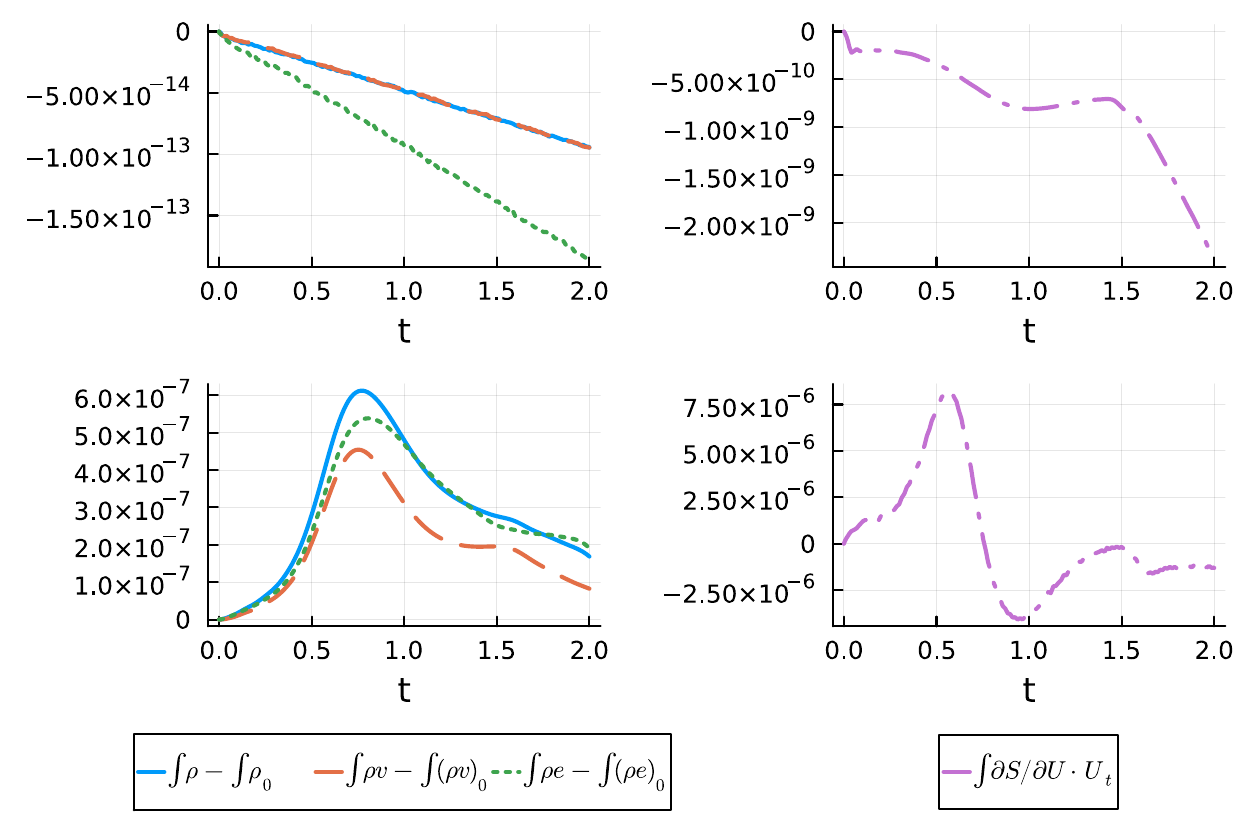}
	\caption{
	    Change in conserved linear quantities (left column) and semi-discrete time derivative of the mathematical entropy (right column) for the compressible Euler equations using the HLL flux (top row) and the Rusanov (local Lax--Friedrichs) flux (bottom row).
	    We observe that the HLL flux conserves linear quantities up to a precision of around $10^{-13}$ (top left plot), while the Rusanov flux only preserves these up to $10^{-7}$ (bottom left plot).
	    Furthermore, the HLL flux yields a small amount of dissipation (top right plot), as indicated by small negative values for the semi-discrete time derivative of the mathematical entropy.
	    In contrast, the Rusanov flux can result in anti-dissipation (bottom right plot).
	}
    \label{fig:entropy_stability_euler}
\end{figure}

\section{Outlook: Extension to multiple dimensions}
\label{sec:extension}

We outline how the proposed framework of sub-cell SBP operators can be extended to multiple dimensions.
To this end, let $d \in \N$ be a positive integer, $\Omega_u, \Omega_v \subset \R^d$ be bounded open domains with piecewise-smooth boundaries $\partial \Omega_u$, $\partial \Omega_v$, and $\boldsymbol{n} = [n_{x_1}, \dots, n_{x_d}]^T$ be the corresponding outward pointing unit normal.
The corresponding overlap domain is $\Omega_O = \Omega_u \cap \Omega_v$, and the non-overlapping subdomains are $\Omega_{\bar{u}} = \Omega_u \setminus \Omega_v$ and $\Omega_{\bar{v}} = \Omega_v \setminus \Omega_u$.
We illustrate this set-up in \cref{fig:schematic_2D} for a two-dimensional example.

\begin{figure}[tb]
\centering
\resizebox{.99\textwidth}{!}{%
\begin{tikzpicture}%[scale=1.75]

	% Diamond (a rotated square)
	\draw[thick, blue1]
    (0,2.0) -- (2.0,0) -- (0,-2.0) -- (-2.0,0) -- cycle;

	% Circle
	\draw[thick, red1] (2,0) circle (2);

	% Optional: fill with transparency
	\fill[blue1!20, opacity=0.5] (0,2) -- (2,0) -- (0,-2) -- (-2,0) -- cycle;
	\fill[red1!20, opacity=0.5] (2,0) circle (2);

	% Overlap region (green)
	\begin{scope}
		\clip (2,0) circle (2); % clip to circle
		\fill[green1!40, opacity=0.7]
        (0,2) -- (2,0) -- (0,-2) -- (-2,0) -- cycle;
	\end{scope}

	% Labels for domains
    \node[blue1] (omegaU) at (-1,0) {$\Omega_u$};
    \node[red1] (omegaV) at (3,0) {$\Omega_v$};
    \node[green1] (omegaO) at (1,0) {$\Omega_O$};

	% Sub-cell IBP formula for Omega_u
	\node[text=blue1] (formulaU) at (-3,3) {\small
	\setlength{\arraycolsep}{1pt}
	$\begin{array}{*3{>{\displaystyle}c}}
		\mathbf{u}^T P_{\bar{u}} ( D_u^{(\xi)} \mathbf{v} )
		+ (D_u^{(\xi)} \mathbf{u}  )^T P_{\bar{u}} \mathbf{v}
		& = &
		\mathbf{u}^T B_{\bar{u}}^{(\xi)} \mathbf{v}
		\\
		\rotatebox{90}{$\!\approx\;$}
  		&&
 		\rotatebox{90}{$\!\!\approx\;$}
  		\\
		\int_{\Omega_{\bar{u}}} u (\partial_{\boldsymbol{\xi}} v)
		+ \int_{\Omega_{\bar{u}}} (\partial_{\boldsymbol{\xi}} u) v
		& = &
		\oint_{\partial \Omega_{\bar{u}}} u v ( \boldsymbol{\xi} \cdot \boldsymbol{n} )
	\end{array}$
	};

	% Arrow from Omega_u to formula
	\draw[-, thick, blue1] (omegaU) -- (formulaU);

	% Sub-cell IBP formula for Omega_v
	\node[text=red1] (formulaV) at (4,3) {\small
	\setlength{\arraycolsep}{1pt}
	$\begin{array}{*3{>{\displaystyle}c}}
		\mathbf{u}^T P_{\bar{v}} ( D_u^{(\xi)} \mathbf{v} )
		+ (D_u^{(\xi)} \mathbf{u}  )^T P_{\bar{v}} \mathbf{v}
		& = &
		\mathbf{u}^T B_{\bar{v}}^{(\xi)} \mathbf{v}
		\\
		\rotatebox{90}{$\!\approx\;$}
  		&&
 		\rotatebox{90}{$\!\!\approx\;$}
  		\\
		\int_{\Omega_{\bar{v}}} u (\partial_{\boldsymbol{\xi}} v)
		+ \int_{\Omega_{\bar{v}}} (\partial_{\boldsymbol{\xi}} u) v
		& = &
		\oint_{\partial \Omega_{\bar{v}}} u v ( \boldsymbol{\xi} \cdot \boldsymbol{n} )
	\end{array}$
	};

	% Arrow from Omega_u to formula
	\draw[-, thick, red1] (omegaV) -- (formulaV);

	% Sub-cell IBP formula for Omega_v
	\node[text=green1] (formulaO) at (1,-3) {\small
	\setlength{\arraycolsep}{1pt}
	$\begin{array}{*3{>{\displaystyle}c}}
		\int_{\Omega_{O}} u (\partial_{\boldsymbol{\xi}} v)
		+ \int_{\Omega_{O}} (\partial_{\boldsymbol{\xi}} u) v
		& = &
		\oint_{\partial \Omega_{O}} u v ( \boldsymbol{\xi} \cdot \boldsymbol{n} )
		\\
		\rotatebox{90}{$\!\approx\;$}
  		&&
 		\rotatebox{90}{$\!\!\approx\;$}
  		\\
		\mathbf{u}^T P_{O_{u/v}} ( D_{u/v}^{(\xi)} \mathbf{v} )
		+ (D_{u/v}^{(\xi)} \mathbf{u}  )^T P_{O_{u/v}} \mathbf{v}
		& = &
		\mathbf{u}^T B_{O_{u/v}}^{(\xi)} \mathbf{v}
	\end{array}$
	};

	% Arrow from Omega_u to formula
	\draw[-, thick, green1] (omegaO) -- (formulaO);

\end{tikzpicture}
}%
\caption{
  	Schematic illustration of a two-dimensional overset domain with domains $\Omega_u$ and $\Omega_v$.
	The corresponding overlap domain is $\Omega_O = \Omega_u \cap \Omega_v$, and the non-overlapping subdomains are $\Omega_{\bar{u}} = \Omega_u \setminus \Omega_v$ and $\Omega_{\bar{v}} = \Omega_v \setminus \Omega_u$.
	We also illustrate the corresponding sub-cell IBP formulas on \textcolor{blue1}{$\Omega_{\bar{u}}$ (in blue)}, \textcolor{red1}{$\Omega_{\bar{v}}$ (in red)}, and \textcolor{green1}{$\Omega_{O}$ (in green)}, as well as how sub-cell MSBP operators mimic these.
}
\label{fig:schematic_2D}
\end{figure}

Consider the following $d$-dimensional IBVP on $\Omega$:
\begin{equation}\label{eq:IBVP_multiD}
\begin{aligned}
    \boldsymbol{w}_t + \sum_{i=1}^d \partial_{x_i} (\boldsymbol{f}_i(\boldsymbol{w}))
    		& = 0,\quad && \boldsymbol{x} \in \Omega, \ t > 0 \\
    \boldsymbol{w}(\boldsymbol{x},0)
    		& = \boldsymbol{w}_0(\boldsymbol{x}), \quad && \boldsymbol{x} \in \Omega, \\
	\boldsymbol{w}(\boldsymbol{x},t)
    		& = \boldsymbol{g}(t), \quad && \boldsymbol{x} \in \Gamma^{-}, \ t>0
\end{aligned}
\end{equation}
where $\Gamma^{-}$ is the inflow part of $\partial \Omega$ and $\Gamma^{+} = \Gamma \setminus \Gamma^{-}$ is the outflow part.
The corresponding overset domain problem seeks the solution to \cref{eq:IBVP_multiD} by finding the solutions $\boldsymbol{u}$ and $\boldsymbol{v}$ on the domains $\Omega_u$ and $\Omega_v$:
\begin{equation}\label{eq:LR_IBVP_multiD}
\begin{aligned}
	& \text{(L-IBVP)} \left\{\begin{aligned}
		\boldsymbol{u}_t + \sum_{i=1}^d \partial_{x_i} (\boldsymbol{f}_i(\boldsymbol{u}))
    			& = 0,
			&& \boldsymbol{x} \in \Omega, \ t > 0 \\
		\boldsymbol{u}(\boldsymbol{x},0)
    			& = \boldsymbol{w}_0(\boldsymbol{x}),
			&& \boldsymbol{x} \in \Omega_u, \\
		\boldsymbol{u}(\boldsymbol{x},t)
			& = \boldsymbol{g}(t),
			&& \boldsymbol{x} \in \Gamma^{-}_u, \ t > 0 \\
		\boldsymbol{u}(\boldsymbol{x},t)
			& = \boldsymbol{v}(\boldsymbol{x},t),
			&& \boldsymbol{x} \in \Gamma^{+}_u, \ t > 0,
	\end{aligned} \right. \\
	& \text{(R-IBVP)} \left\{\begin{aligned}
		\boldsymbol{v}_t + \sum_{i=1}^d \partial_{x_i} (\boldsymbol{f}_i(\boldsymbol{v}))
    			& = 0,
			&& \boldsymbol{x} \in \Omega, \ t > 0 \\
		\boldsymbol{v}(\boldsymbol{x},0)
    			& = \boldsymbol{w}_0(\boldsymbol{x}),
			&& \boldsymbol{x} \in \Omega_v, \\
		\boldsymbol{v}(\boldsymbol{x},t)
			& = \boldsymbol{g}(t),
			&& \boldsymbol{x} \in \Gamma^{-}_v, \ t > 0 \\
		\boldsymbol{v}(\boldsymbol{x},t)
			& = \boldsymbol{u}(\boldsymbol{x},t),
			&& \boldsymbol{x} \in \Gamma^{+}_v, \ t > 0,
	\end{aligned} \right.
\end{aligned}
\end{equation}
where $\Gamma^{-}_u$ and $\Gamma^{-}_v$ are the inflow parts of $\partial \Omega_u$ and $\partial \Omega_v$, and $\Gamma^{+}_u = \Gamma_u \setminus \Gamma^{-}_u$ and $\Gamma^{+}_v = \Gamma_v \setminus \Gamma^{-}_v$ are the outflow parts.

Suppose we restrict the multi-dimensional overset domain problem \cref{eq:LR_IBVP_multiD} to a scalar linear advection equation, i.e., $f_i(w) = \alpha_i w$ with $\alpha_i \in \R$ for $i=1,\dots,d$.
In this case, we can show the conservation and energy stability of exact solutions of the continuous overset domain problem \cref{eq:LR_IBVP_multiD} using an analysis similar to the one in \Cref{sec:continuous}.
Also see \cite{glaubitz2023multi}, where such an analysis was carried out for a multi-dimensional single domain problem.
For brevity, we omit any details.
Crucially, this analysis again relies on applying multi-dimensional IBP to the three subdomains $\Omega_{\bar{u}}$, $\Omega_{\bar{v}}$, and $\Omega_O$.
The resulting multi-dimensional sub-cell IBP formula for the $\boldsymbol{\xi}$-derivative reads
\begin{equation}\label{eq:IBP_multiD}
	\int_{\Omega_w} u (\partial_{\boldsymbol{\xi}} v) \intd \boldsymbol{x} + \int_{\Omega_w} (\partial_{\boldsymbol{\xi}} u) v \intd \boldsymbol{x}
		= \oint_{\partial \Omega_w} u v ( \boldsymbol{\xi} \cdot \boldsymbol{n} ) \intd s
\end{equation}
with $w \in \{ \bar{u}, \bar{v}, O \}$.
Moreover, recall that the directional derivative $\partial_{\boldsymbol{\xi}} f$ is defined as $\partial_{\boldsymbol{\xi}} f = (\nabla f) \cdot \boldsymbol{\xi}$.

To mimic the continuous multi-dimensional analysis of the overset domain problem \cref{eq:LR_IBVP_multiD} on a discrete level for numerical methods, we need a discrete counterpart to the multi-dimensional IBP formula \cref{eq:IBP_multiD}.
That is, we need multi-dimensional SBP (MSBP) operators \cite{nordstrom2003finite,hicken2016multidimensional,glaubitz2023multi} with a sub-cell SBP property.
To formalize this idea, we consider a generic cell $\Omega_{\rm ref} \subset \R^d$ and the two sub-cells $\Omega_{L}, \Omega_{R} \subset \Omega_{\rm ref}$.
For instance, in the context of \cref{fig:schematic_2D} and \cref{eq:LR_IBVP_multiD}, $\Omega_{\rm ref}, \Omega_{L}, \Omega_{R}$ can either be (i) $\Omega_{\rm ref} = \Omega_u$, $\Omega_{L} = \Omega_{\bar{u}}$, and $\Omega_{R} = \Omega_O$ or (ii) $\Omega_{\rm ref} = \Omega_v$, $\Omega_{L} = \Omega_{O}$, and $\Omega_{R} = \Omega_{\bar{v}}$.
We assume that $\Omega_{\rm ref}, \Omega_{L}, \Omega_{R}$ have piecewise-smooth boundaries $\partial \Omega_{\rm ref}, \partial \Omega_{L}, \partial \Omega_{R}$.
Let $X_L = [ (\mathbf{x}_L)_1,\dots,(\mathbf{x}_L)_{N_L}] \in \Omega_L^{N_L}$ and $X_R = [ (\mathbf{x}_R)_1,\dots,(\mathbf{x}_R)_{N_R}] \in \Omega_R^{N_R}$ be a sequence of grid points on $\Omega_L$ and $\Omega_R$, respectively.
Furthermore, let $X = [X_L, X_R] \in \Omega_{\rm ref}^N$ with $N = N_L + N_R$ be the combined sequence of grid points on $\Omega_{\rm ref}$.
We denote the nodal values of the directional derivative $\partial_{\boldsymbol{\xi}} f$ in $\boldsymbol{\xi}$-direction on the node set $X$ as $\mathbf{f_{\boldsymbol{\xi}}} = [ (\partial_{\boldsymbol{\xi}} f)(\mathbf{x}_1), \dots, (\partial_{\boldsymbol{\xi}} f)(\mathbf{x}_N) ]^T$.
Finally, we write $\mathbf{x}_n \in X_{L/R}$ when $\mathbf{x}_n$ is a point in the grid $X_{L/R}$ on $\Omega_{L/R}$, i.e., there exists an index $i$ such that $\mathbf{x}_n = (\mathbf{x}_{L/R})_i$.
We then define the desired sub-cell MSBP operators as follows.

\begin{definition}[Sub-cell MSBP operators]\label{def:subcell_MSBP}
	By $(\Omega_{\rm ref}, \Omega_{L}, \Omega_{R})$ we denote a triple containing a cell $\Omega_{\rm ref} \subset \R^d$ and the two non-overlapping sub-cells $\Omega_{L}, \Omega_{R} \subset \Omega_{\rm ref}$ with $\Omega_{L} \cup \Omega_{R} = \overline{\Omega}_{\rm ref}$ and nodes $X_L \in \Omega_L^{N_L}$ and $X_R  \in \Omega_R^{N_R}$.
	Furthermore, let $\mathcal{F} \subset C^1(\Omega_{\rm ref})$ be a finite-dimensional function space.
	An operator $D^{(\boldsymbol{\xi})} = P^{-1} Q^{(\boldsymbol{\xi})} \in \R^{N \times N}$ approximating $\partial_{\boldsymbol{\xi}}$ on the nodes $X = [X_L, X_R] \in \Omega_{\rm ref}^N$ with $N = N_L + N_R$ is an \emph{$\mathcal{F}$-exact sub-cell MSBP operator for $(\Omega_{\rm ref}, \Omega_{L}, \Omega_{R})$} if the following conditions are satisfied:
	\begin{enumerate}
		\item[(i)]
		$D^{(\boldsymbol{\xi})} \mathbf{f} = \mathbf{f}_{\boldsymbol{\xi}}$ for all $f \in \mathcal{F}$

		\item[(ii)]
		$P_{L/R} = \diag( \mathbf{p}_{L/R} )$ with $(p_{L/R})_n > 0$ if $\mathbf{x}_n \in X_{L/R}$ and $(p_{L/R})_n = 0$ otherwise

		\item[(iii)]
		$P_{L/R} D^{(\boldsymbol{\xi})} = Q_{L/R}^{(\boldsymbol{\xi})}$ with $Q_{L/R}^{(\boldsymbol{\xi})} + ( Q_{L/R}^{(\boldsymbol{\xi})} )^T = B_{L/R}^{(\boldsymbol{\xi})}$

		\item[(iv)]
		the boundary matrices $B_{L/R}^{(\boldsymbol{\xi})}$ satisfy
		\begin{equation}\label{eq:MFSP_boundary_matrix}
			\mathbf{f}^T B_{L/R}^{(\boldsymbol{\xi})} \mathbf{g} = \oint_{\partial \Omega_{L/R}} f g ( \boldsymbol{\xi} \cdot \boldsymbol{n} ) \intd s, \quad \forall f,g \in \mathcal{F}
		\end{equation}

		\item[(v)]
		$P = P_L + P_R$, $B^{(\boldsymbol{\xi})} = B_L^{(\boldsymbol{\xi})} + B_R^{(\boldsymbol{\xi})}$, and $Q^{(\boldsymbol{\xi})} = Q_L^{(\boldsymbol{\xi})} + Q_R^{(\boldsymbol{\xi})}$
	\end{enumerate}
\end{definition}

The motivation for \cref{def:subcell_MSBP} is that if $D_u^{(\boldsymbol{\xi})} = P_u^{-1} Q_u^{(\boldsymbol{\xi})}$ and $D_v^{(\boldsymbol{\xi})} = P_v^{-1} Q_v^{(\boldsymbol{\xi})}$ are sub-cell MSBP operators for $(\Omega_u, \Omega_{\bar{u}}, \Omega_{O})$ and $(\Omega_v, \Omega_{\bar{v}}, \Omega_{O})$ with norm and boundary matrices $P_{u} = P_{\bar{u}} + P_{O_u}$, $P_{v} = P_{\bar{v}} + P_{O_v}$ and $B_{u}^{(\boldsymbol{\xi})} = B_{\bar{u}}^{(\boldsymbol{\xi})} + B_{O_u}^{(\boldsymbol{\xi})}$, $B_{v}^{(\boldsymbol{\xi})} = B_{\bar{v}}^{(\boldsymbol{\xi})} + B_{O_v}^{(\boldsymbol{\xi})}$, then we immediately get
\begin{equation}\label{eq:subcell_MSBP}
\begin{aligned}
	\mathbf{f}^T P_{w} ( D_u^{(\boldsymbol{\xi})} \mathbf{g} ) + ( D_u^{(\boldsymbol{\xi})} \mathbf{f} )^T P_{w} \mathbf{g}
		& = \mathbf{f}^T B_{w}^{(\boldsymbol{\xi})} \mathbf{g},
		\quad
		&& w \in \{ \bar{u}, O_u \}, \\
	\mathbf{f}^T P_{w} ( D_v^{(\boldsymbol{\xi})} \mathbf{g} ) + ( D_v^{(\boldsymbol{\xi})} \mathbf{f} )^T P_{w} \mathbf{g}
		& = \mathbf{f}^T B_{w}^{(\boldsymbol{\xi})} \mathbf{g},
		\quad
		&& w \in \{ \bar{v}, O_v \},
\end{aligned}
\end{equation}
which are the discrete versions of the continuous IBP formulas in \cref{eq:IBP_multiD}.
We also summarize these relations in \cref{fig:schematic_2D}.

As a result, we expect that the one-dimensional existence and construction framework presented in \Cref{sec:existence} extends naturally to the multi-dimensional setting when combined with the MSBP theory developed in \cite{glaubitz2023multi}.
In particular, we anticipate that sub-cell MSBP operators on the reference cell $\Omega_{\rm ref}$ can be constructed by first building MSBP operators on the non-overlapping sub-cells $\Omega_{L}, \Omega_{R} \subset \Omega_{\rm ref}$ and subsequently coupling them as in \cref{thm:main}.
Moreover, the required MSBP operators on the sub-cells $\Omega_{L}$ and $\Omega_{R}$—which may possess more complex geometries than $\Omega_{\rm ref}$—can be obtained using the general framework introduced in \cite{glaubitz2023multi}.
However, such an extension would require the development of suitable multi-dimensional sub-cell SATs (analogous to \cref{eq:SATs}), posed on sub-cell interfaces, whose formulation is significantly more involved than in the one-dimensional case. 
A detailed numerical investigation of the proposed sub-cell MSBP operators, as well as their incorporation into a fully discrete overset grid method, is beyond the scope of the present work.
\section{Summary}
\label{sec:summary}

We introduced the concept of sub-cell SBP operators as a framework for developing provably conservative and energy-stable numerical methods for overset grid discretizations of one-dimensional hyperbolic conservation laws on fixed overset domain problems that do not change with time or under grid refinement.
These results provide sufficient---not necessary---conditions for stability in a one-dimensional, fixed-overlap setting: they complement existing normal-mode and GKS-type stability analyses.
The continuous analysis \cite{kopriva2022theoretical,kopriva2022theoretical2} highlighted that IBP must be applied not only on entire cells but also on sub-cells, which the proposed sub-cell SBP operators mimic discretely.
Furthermore, we proved the existence of sub-cell SBP operators and illustrated their construction.
Computational experiments demonstrated the robustness and accuracy of the resulting sub-cell SBP methods across a range of problems, including linear advection, the inviscid Burgers' equation, Maxwell’s equations, and compressible Euler equations, showing reduced long-time error growth compared to baseline overset methods.
Although much work still remains in this area, our path of future research will prioritize constructing sub-cell SBP operators in higher dimensions and investigating efficient algorithms for their practical implementation, potentially building upon existing results for multi-dimensional SBP operators \cite{hicken2016multidimensional,glaubitz2023multi,hicken2024constructing,worku2025tensor}.

\appendix 
\section{Proofs for \Cref{sub:operators_existence}}
\label{app:proofs}

We start by proving \Cref{lem:existence}, which translates the existence of sub-cell SBP operators into the existence of solutions to a pair of non-linear equations.

Suppose we are given $B_L$ and $B_R$ satisfying (iv) in \Cref{def:subcell_SBP} and let $B = B_L + B_R$.
Note that we can always decompose $Q_L$ and $Q_R$ as $Q_{L/R} = S_{L/R} + B_{L/R} / 2$, where $S_{L/R}$ are the skew-symmetric parts of $Q_{L/R}$ and $B_{L/R}/2$ are the symmetric parts of $Q_{L/R}$.
Henceforth, we always suppose that (i) $Q_{L/R} = S_{L/R} + B_{L/R} / 2$, (ii) $S_{L/R}$ is skew-symmetric, and (iii) $P_{L/R} = \diag( \mathbf{p}_{L/R} )$ with $(p_{L/R})_n > 0$ if $x_n \in \Omega_{L/R}$ and $(p_{L/R})_n = 0$ otherwise.
We begin with \Cref{lem:existence}, which translates the existence of sub-cell SBP operators into the existence of solutions to a pair of non-linear equations.

\begin{lemma}\label{lem:existence}
	Let $B_L$ and $B_R$ satisfy (iv) in \Cref{def:subcell_SBP} and let $B = B_L + B_R$, $P = P_L + P_R$, and $Q = Q_L + Q_R$.
	Furthermore, let $\{f_k\}_{k=1}^K$ be a basis of $\mathcal{F}$ and $V = [ \mathbf{f}_1, \dots, \mathbf{f}_K ] \in \R^{N \times K}$ and $V' = [ \mathbf{f}'_1, \dots, \mathbf{f}'_K ] \in \R^{N \times K}$.
	Then, $D = P^{-1} Q$ is an $\mathcal{F}$-exact sub-cell SBP operator if and only if
	\begin{subequations}\label{eq:existence}
	\begin{align}
		S_L V + S_R V - P_L V' - P_R V' + B V / 2 & = 0, \label{eq:existence1} \\
		P_L S_R - P_R S_L + P_L B_R / 2 - P_R B_L / 2 & = 0 \label{eq:existence2}
	\end{align}
	\end{subequations}
	are satisfied.
\end{lemma}

\begin{proof}
	We have to show that conditions (i) and $P_{L/R} D = Q_{L/R}$ in (iii) of \Cref{def:subcell_SBP} are equivalent to \cref{eq:existence1,eq:existence2}.
	% First equation
	To this end, note that (i) is equivalent to $D V = V'$.
	Multiplying both sides by $P$ shows that (i) is equivalent to $Q V = P V'$.
	Substituting $Q = Q_L + Q_R$, $Q_{L/R} = S_{L/R} + B_{L/R} / 2$, and $P = P_L + P_R$, we can rewrite $Q V = P V'$ as \cref{eq:existence1}.

	% Second equation
	Next, we rewrite $P_{L} D = Q_{L}$ in (iii) of \Cref{def:subcell_SBP} as $P_L P^{-1} (S + B/2) = S_L + B_L/2$.
	Multiplying both sides by $P$ and noting that $P P_L = P_L P$---since diagonal matrices commute with each other, we get $P_L(S + B/2) = P(S_L + B_L/2)$.
	Substituting $S = S_L + S_R$ and $B = B_L + B_R$, we rewrite $P_{L} D = Q_{L}$ as \cref{eq:existence2}.
	Using the same arguments as above, we also rewrite $P_{R} D = Q_{R}$ as \cref{eq:existence2}, which shows that $P_{L/R} D = Q_{L/R}$ in (ii) is equivalent to \cref{eq:existence2}.
\end{proof}

\Cref{lem:existence} implies several necessary conditions on the structure of sub-cell SBP operators and their existence.
Specifically, it reveals fundamental restrictions on the structure of sub-cell SBP operators.
In particular, the existence of such operators necessitates a specific block-diagonal structure.
We formalize this insight in \Cref{cor:structure_S}, which ultimately leads to the proof of our main result, \Cref{thm:main}.

\begin{lemma}\label{cor:structure_S}
	Let $D = P^{-1} Q$ be a sub-cell SBP operator with $Q = S + B/2$, $B = B_L + B_R$, and $B_L$, $B_R$ as in \cref{eq:single_construction_B_form}.
	\begin{enumerate}
		\item
		The discrete projection operators $\mathbf{e}_L, \mathbf{e}_{M_L}$ and $\mathbf{e}_{M_R}, \mathbf{e}_{R}$ that define $B_L$ and $B_R$ via \cref{eq:single_construction_B_form} are supported on $\Omega_L$ and $\Omega_R$, respectively.

		\item
		The skew-symmetric matrix $S$ is $S = S_L + S_R$ with
		\begin{equation}
			S_L = \begin{bmatrix}
				S_L^{1,1} & 0 \\
				0 & 0
			\end{bmatrix},
			\quad
			S_R = \begin{bmatrix}
				0 & 0 \\
				0 & S_R^{2,2}
			\end{bmatrix},
		\end{equation}
		where $S_L^{1,1} \in \R^{N_L \times N_L}$ and $S_R^{2,2} \in \R^{N_R \times N_R}$.
	\end{enumerate}
\end{lemma}

\begin{proof}
	We demonstrate the assertions of \Cref{cor:structure_S} for $B_L$ and $S_L$.
	The proof for $B_R$ and $S_R$ is analogous.
	We start by writing $S_L, P_L, D, B_L \in \R^{N \times N}$ as
	\begin{equation}\label{eq:structure_S_proof1}
		S_L = \begin{bmatrix}
			S_L^{1,1} & S_L^{1,2} \\
			S_L^{2,1} & S_L^{2,2}
		\end{bmatrix}, \quad
		P_L = \begin{bmatrix}
			P_L^{1,1} & 0 \\
			0 & 0
		\end{bmatrix}, \quad
		D = \begin{bmatrix}
			D^{1,1} & D^{1,2} \\
			D^{2,1} & D^{2,2}
		\end{bmatrix}, \quad
		B_L = \begin{bmatrix}
			B_L^{1,1} & B_L^{1,2} \\
			B_L^{2,1} & B_L^{2,2}
		\end{bmatrix}
	\end{equation}
	with $S_L^{1,1}, D^{1,1}, B_L^{1,1} \in \R^{N_L \times N_L}$,
	$S_L^{1,2}, D^{1,2}, B_L^{1,2} \in \R^{N_L \times N_R}$,
	$S_L^{2,1}, D^{2,1}, B_L^{2,1} \in \R^{N_R \times N_L}$, and
	$S_L^{2,2}, D^{2,2}, B_L^{2,2} \in \R^{N_R \times N_R}$.
	Furthermore, recall that $S_L = P_L D - B_L / 2$, which yields
	\begin{equation}\label{eq:structure_S_proof2}
		\begin{bmatrix}
			S_L^{1,1} & S_L^{1,2} \\
			S_L^{2,1} & S_L^{2,2}
		\end{bmatrix}
		= \begin{bmatrix}
			P_L^{1,1} D^{1,1} & P_L^{1,1} D^{1,2} \\
			0 & 0
		\end{bmatrix}
		- \frac{1}{2} \begin{bmatrix}
			B_L^{1,1} & B_L^{1,2} \\
			B_L^{2,1} & B_L^{2,2}
		\end{bmatrix}.
	\end{equation}
	Specifically, \cref{eq:structure_S_proof2} implies $S_L^{2,1} = -B_L^{2,1}/2$ and $S_L^{2,2} = -B_L^{2,2}/2$.
	Now recall that $S_L$ is skew-symmetric, that is, $[S_L^{1,1}]^T = - S_L^{1,1}$, $[S_L^{2,2}]^T = -S_L^{2,2}$, and $[S_L^{2,1}]^T = -S_L^{1,2}$.
	At the same time, $B_L$ is symmetric, that is, $[B_L^{1,1}]^T = B_L^{1,1}$, $[B_L^{2,2}]^T = B_L^{2,2}$, and $[B_L^{2,1}]^T = B_L^{1,2}$.
	Hence, $[B_L^{2,2}]^T = -2[S_L^{2,2}]^T = 2S_L^{2,2} = -B_L^{2,2}$ implies $B_L^{2,2} = 0$ and therefore $S_L^{2,2} = 0$.
	Specifically, $B_L^{2,2} = 0$ means that $\mathbf{e}_L$ and $\mathbf{e}_{M_L}$ are supported on $\Omega_L$, which shows the first assertion of \Cref{cor:structure_S}.
	Moreover, in this case, $B_L^{2,1} = 0$ and $B_L^{1,2} = 0$ due to $B_L = \mathbf{e}_{M_L}\mathbf{e}_{M_L}^T - \mathbf{e}_L\mathbf{e}_L^T$.
	Finally, this yields, $S_L^{2,1} = 0$ and $S_L^{1,2} = 0$, implying the second assertion of \Cref{cor:structure_S}.
\end{proof}

\Cref{cor:structure_S} imposes structural constraints on sub-cell SBP operators.
In particular, the projection vectors $\mathbf{e}_{M_L}$ and $\mathbf{e}_{M_R}$ must be one-sided, relying solely on values to the left and right of $x_M$, respectively. 
For additional details on the projection and boundary operators, see \Cref{sub:boundaryOp} and \Cref{sub:operators_examples}.
Furthermore, \Cref{cor:structure_S} implies that three of the four blocks in $S_L$ and $S_R$ must be zero, thereby reducing the degrees of freedom in designing these operators.
With these structural insights in place, we are now ready to prove our main result, \Cref{thm:main}.

\begin{proof}[Proof of \Cref{thm:main}]
	We first show that $D = P^{-1}( S + B/2 )$ being an $\mathcal{F}$-exact sub-cell SBP operator implies (a) and (b).
	To this end, note that (a) follows from \Cref{cor:structure_S} and its proof.
	To demonstrate (b) for $D_L^{1,1}$, we have to validate
	\begin{enumerate}
		\item[(b.i)]
		$D_L^{1,1,} f(\mathbf{x}_L) = f'(\mathbf{x}_L)$ for all $f \in \mathcal{F}$;

		\item[(b.ii)]
		$P_L^{1,1} = \diag( \mathbf{p}_L^{1,1} )$ with $(\mathbf{p}_L^{1,1})_n > 0$ for all $n=1,\dots,N_L$;

		\item[(b.iii)]
		$S_L^{1,1}$ is skew-symmetric;

		\item[(b.iv)]
		$f(\mathbf{x}_L)^T B_L g(\mathbf{x}_L) = fg|_{\partial \Omega_L}$.

	\end{enumerate}
	Note that (b.i) follows from (i) in \Cref{def:subcell_SBP}; (b.ii) from (ii); (b.iii) from $S$ being skew symmetric; and (b.iv) from (iv), where we have used the decompositions in (a).
	We can analogously show that (b) holds for $D_R^{2,2}$. 
	Second, we show that (a) and (b) imply that $D = P^{-1}( S + B/2 )$ is a sub-cell SBP operator.
	To this end, we lift-up the components of the SBP operators on $\mathbf{x}_L \in \R^{N_L}$ and $\mathbf{x}_R \in \R^{N_R}$ to the larger grid $\mathbf{x} \in \R^{N_L + N_R}$ and choose $P_L = \diag( P_L^{1,1}, 0 )$, $P_R = \diag( 0, P_R^{2,2} )$, $S_L = \diag( S_L^{1,1}, 0 )$, $S_R = \diag( 0, S_R^{2,2} )$, and $B_L = \diag( B_L^{1,1}, 0 )$, $B_R = \diag( 0, B_R^{2,2} )$.
	Then, (i) in \Cref{def:subcell_SBP} follows from $D \mathbf{f} = [ D_L^{1,1} f(\mathbf{x}_L); D_R^{2,2} f(\mathbf{x}_R) ] = [ f'(\mathbf{x}_L); f'(\mathbf{x}_R)] = \mathbf{f}'$ for all $\mathcal{F}$;
	(ii) from $P_L^{1,1}$ and $P_R^{2,2}$ respectively being diagonal positive definite norm operators on $\mathbf{x}_L$ and $\mathbf{x}_R$;
	(iii) from (b), specifically, $S_L^{1,1}$ and $S_R^{2,2}$ being skew-symmetric;
	(iv) from $\mathbf{f}^T B_L \mathbf{g} = f(\mathbf{x}_L)^T B_L^{1,1} g(\mathbf{x}_L) = fg|_{\partial \Omega_L}$ and $\mathbf{f}^T B_R \mathbf{g} = f(\mathbf{x}_R)^T B_R^{2,2} g(\mathbf{x}_R) = fg|_{\partial \Omega_R}$; and (v) from the decompositions in (a).
\end{proof}

\section*{Acknowledgments}
JG acknowledges support by the Swedish Research Council (VR) Starting Grant \#2025-05370, the Zenith Career Development Grant \#26.07, and the National Academic Infrastructure for Supercomputing in Sweden (NAISS) grants \#2025/22-1599 and \#2024/22-1207.
JL acknowledges the support by the Deutsche Forschungsgemeinschaft (DFG) within the Research Training Group GRK 2583 ``Modeling, Simulation and Optimization of Fluid Dynamic Applications."
JN was supported by the VR grant \#2021-05484 and the University of Johannesburg's Global Excellence and Stature Initiative Funding.

\bibliographystyle{siamplain}
\bibliography{references}

\end{document}